\documentclass[a4paper,11pt]{amsart}
\usepackage{amssymb,amsfonts,amsxtra,fullpage,
	mathrsfs,placeins,graphicx,verbatim,stmaryrd,cite,color,lmodern,hyperref, mathtools, tikz-cd}
\usepackage[all]{xy}
\usepackage[T1]{fontenc}
\usepackage{hyperref}
\xyoption{line}
\usepackage{euscript}
\newtheorem{theorem}{Theorem}[section]
\newtheorem{cor}[theorem]{Corollary}
\newtheorem{lem}[theorem]{Lemma}
\newtheorem{prop}[theorem]{Proposition}

\newtheorem{example}[theorem]{Example}
\newtheorem{defi}[theorem]{Definition}
\newtheorem{rem}[theorem]{Remark}

\numberwithin{equation}{section}
\DeclareMathOperator{\Hom}{Hom}
\DeclareMathOperator{\coker}{coker}
\DeclareMathOperator{\MC}{MC}
\DeclareMathOperator{\MCdg}{MC_{dg}}

\DeclareMathOperator{\Tw}{Tw}
\DeclareMathOperator{\Twfg}{Tw_{{fg}}}
\DeclareMathOperator{\Twfgt}{Tw^2_{{fg}}}
\DeclareMathOperator{\Ho}{Ho}
\DeclareMathOperator{\im}{Im}

\DeclareMathOperator{\Dcoderivedc}{D_c^{{II}}}
\DeclareMathOperator{\Dcoderived}{D^{{II}}}
\DeclareMathOperator{\Perfcoderivedc}{Perf_c^{{II}}}
\DeclareMathOperator{\Perfcoderiveddgc}{Perf_{dgc}^{{II}}}
\DeclareMathOperator*{\Perf}{Perf}
\DeclareMathOperator*{\Mod}{Mod}
\DeclareMathOperator*{\Comod}{Comod}
\DeclareMathOperator*{\colim}{colim}
\DeclareMathOperator{\ev}{ev}
\DeclareMathOperator{\Z}{\mathbb{Z}}
\DeclareMathOperator{\holim}{holim}
\DeclareMathOperator*{\hoilim}{\underleftarrow\holim}

\newcommand{\into}{\hookrightarrow}
\def\Coalg{\mathbf{Coalg}}

\def\Alg{\mathbf{Alg}}

\def\sSet{\mathbf{sSet}}
\def\ground{\mathbf k}
\def\C{\mathscr C}

\def\MCmod{\mathscr {MC}}
\def\End{\operatorname{End}}
\def\id{\operatorname{id}}

\def\OCat{\ensuremath{\Omega}_\mathrm{cat}}

\tikzset{%
	symbol/.style={%
		draw=none,
		every to/.append style={%
			edge node={node [sloped, allow upside down, auto=false]{$#1$}}}
	}
}

\newcommand{\cog}{\ensuremath{\mathbf{Cog}}}
\newcommand{\alg}{\ensuremath{\mathbf{Alg}}}
\newcommand{\acog}{\ensuremath{\mathbf{Cog}^\mathrm{coaug}}}
\newcommand{\cncog}{\ensuremath{\mathbf{Cog}^\mathrm{conil}}}
\newcommand{\aalg}{\ensuremath{\mathbf{Alg}^\mathrm{aug}}}
\newcommand{\ccog}{\ensuremath{\mathbf{cuCog}}}
\newcommand{\calg}{\ensuremath{\mathbf{cuAlg}}}
\newcommand{\accog}{\ensuremath{\mathbf{cuCog}^\mathrm{coaug}}}
\newcommand{\acalg}{\ensuremath{\mathbf{cuAlg}^\mathrm{aug}}}
\newcommand{\ccogp}{\ensuremath{\mathbf{cuCog}_*}}
\newcommand{\calgp}{\ensuremath{\mathbf{cuAlg}_\varnothing}}

\newcommand{\algqim}{\ensuremath{\mathbf{Alg}_\mathrm{q.i.}}}
\newcommand{\algmcm}{\ensuremath{\mathbf{Alg}_\mathrm{MC}}}
\newcommand{\algexm}{\ensuremath{\mathbf{Alg}_\mathrm{ex}}}
\newcommand{\pcog}{\ensuremath{\mathbf{ptdCog}}}
\newcommand{\dgcat}{\ensuremath{\mathbf{dgCat}}}
\newcommand{\sw}{\ensuremath{\kern -1pt\rhd \kern -2pt}}
\newcommand{\pccalg}{\ensuremath{\mathbf{pccuAlg}}}
\newcommand{\fdcalg}{\ensuremath{\mathbf{cuAlg}^\mathrm{fd}}}
\newcommand{\fdccog}{\ensuremath{\mathbf{cuCog}^\mathrm{fd}}}
\newcommand{\algmc}{\ensuremath{\mathrm{Alg}_\mathrm{MC}}}
\newcommand{\pcaalg}{\ensuremath{\mathbf{pcAlg}^\mathrm{aug}}}
\newcommand{\fdaalg}{\ensuremath{\mathbf{Alg}_\mathrm{fd}^\mathrm{aug}}}
\begin{document}
	
	\title[Global Koszul duality]{Global Koszul duality}
	\author{M.~Booth and A.~Lazarev}
	\thanks{This work was partially supported by EPSRC grant EP/T029455/1 and the Additional Funding Programme for Mathematical Sciences, delivered by EPSRC (EP/V521917/1) and the Heilbronn Institute for Mathematical Research.}	
		\address{\mbox{Department of Mathematics,
Imperial College London,
SW7 2AZ, United Kingdom}\newline
\indent Heilbronn Institute for Mathematical Research,
Bristol, BS8 1UG, United Kingdom}
	\email{matt.booth@imperial.ac.uk}
	\address{School of Mathematical Sciences\\
		Lancaster University\\ LA1 4YF\\United Kingdom}
		\email{a.lazarev@lancaster.ac.uk}

	\begin{abstract}
	We construct a monoidal model structure on the category of all curved coalgebras and show that it is Quillen equivalent, via the extended bar-cobar adjunction, to another model structure we construct on the category of curved algebras.  When the coalgebras under consideration are conilpotent and the algebras are dg, i.e.\ uncurved, this corresponds to the ordinary dg Koszul duality of Positselski and Keller--Lef\`evre. As an application we construct global noncommutative moduli spaces for flat connections on vector bundles, holomorphic structures on almost complex vector bundles, dg modules over a dg algebra, objects in a dg category, and others.
	\end{abstract}
	
	\maketitle
    \setcounter{tocdepth}{1}
	\tableofcontents
	\section{Introduction}
    
	Koszul duality is a phenomenon occurring  across a wide range of subfields of algebra, geometry and homotopy theory. Its earliest manifestation goes back to the work of Quillen on rational homotopy theory \cite{Quillen69} and takes the form of a Quillen equivalence between  categories of differential graded (dg) Lie algebras and cocommutative dg coalgebras, under moderately severe grading restrictions. These restrictions were removed much later in an influential work of Hinich \cite{Hinich01}. Hinich's breakthrough was the realisation that the correct notion of a weak equivalence on the coalgebra side is \emph{not} that of quasi-isomorphism, but a finer notion which implies quasi-isomorphism but is not implied by it. Hinich's approach to Koszul duality underpins the modern approach to deformation theory, cf. \cite{GLST20} for an elementary introduction to this circle of ideas. 
	
	The next important development came with the work \cite{Lefevre03} of Keller's student Lef\`evre, cf. also \cite{Keller03} for an overview. This was an associative analogue of Hinich's work and it had an important addition of the module-comodule level Koszul duality. The modern definitive treatment of associative Koszul duality belongs to Positselski \cite{Positselski11} who formulated the theory in a more general framework than Lef\`evre and corrected some inaccuracies present in Lef\`evre's thesis. An updated survey containing interesting bits of a truly complicated history of the subject is given in \cite{PosSurvey}.
	
	In the present paper we are interested in developing further, and hopefully deeper, our understanding of associative Koszul duality. As explained in the papers \cite{PosSurvey, Positselski11} mentioned above, there is an equivalence of $\infty$-categories (or, more precisely, a Quillen equivalence of model categories) between the categories of associative dg algebras and conilpotent coalgebras. Though there are no grading restrictions (present in Quillen's work), there are still some significant caveats: 
	\begin{enumerate}\item The dg algebras under consideration are arbitrary, however the corresponding coalgebras are conilpotent, which is a severe restriction indeed;
		\item If the dg algebras under consideration are not augmented, then their Koszul dual coalgebras are not dg but curved; however there is no corresponding theory where algebras are curved and coalgebras are not coaugmented;
		\item Weak equivalences of dg algebras are quasi-isomorphisms; however for (curved or not) coalgebras the corresponding notion is very different from a quasi-isomorphism, and has no easy intrinsic definition.
	\end{enumerate} 
	These issues lead one to seek a more general version of Koszul duality where they are not present. Another motivation is the desire to have a theory on which to base the study of global moduli spaces (as opposed to deformations over a local base). In fact, this turns out to be closely related to the point (1) above. 	 
	
	One indication that a more general theory exists was present already in \cite[Section 6.7]{Positselski11} where Koszul duality was established between comodules over a \emph{nonconilpotent} (curved) coalgebra and modules over its cobar construction. However, obtaining a symmetric result involving a \emph{bar} construction was not possible at the time of writing of op.cit.\ simply because the ordinary bar construction of a dg algebra is manifestly conilpotent. A key new ingredient was, therefore, the \emph{nonconilpotent} extended bar construction of \cite{AnelJoyal}. Indeed, using this notion, the symmetric result alluded to above was obtained in \cite{GL20}.
	
	There is still a long distance between the results of \cite{GL20} and a full-fledged Koszul duality between (dg or curved) algebras and (dg or curved) coalgebras. The other missing key ingredient is the appropriate definition of a weak equivalence for algebras and coalgebras. We understand by now that the familiar notion of a quasi-isomorphism is inadequate for coalgebras, but it has now stopped working for algebras as well (e.g.\ because we need to include curved algebras in our theory, but also because the extended bar construction does not preserve quasi-isomorphisms when the latter make sense). 
	
	In the present work we construct a Quillen equivalence (dubbed `global Koszul duality') between curved algebras and curved coalgebras which is free of the issues listed above (and which are attributes of ordinary, or `local', Koszul duality). We call the weak equivalences in the categories of curved algebras and coalgebras \emph{Maurer-Cartan} (or MC) equivalences because they are closely related to Maurer-Cartan elements and related dg categories. An MC equivalence is not directly comparable to a quasi-isomorphism since the latter notion is ill-defined for curved (co)algebras. In the uncurved case however, it is strictly finer than a quasi-isomorphism.
	
	More precisely, our main results are as follows:
	\begin{enumerate}\item We construct a left proper combinatorial model structure on the category of curved coalgebras $\ccogp$ (modified by adding a final object) where the weak equivalences are the MC equivalences and the cofibrations are the injections (Theorem \ref{thm:coalgebrasmodel}). All objects in this model structure are cofibrant. This model structure is monoidal (Theorem \ref{cogmonoidal}).
		\item We construct a right proper combinatorial model structure on the category $\calgp$ of initialised curved algebras where the weak equivalences are the MC equivalences and the fibrations are the strong fibrations, a certain subclass of the surjections. Every curved algebra is fibrant. This model category is model enriched over $\ccogp$ (Theorem \ref{modelenrich}). This is contained in Theorems \ref{mcmostralg}, \ref{rightproperlem}, and \ref{modeltheoreticfibs} below. A slight asymmetry with the coalgebra case is that not all surjections are model-theoretic fibrations.
		\item There is a Quillen equivalence between the model structures of (1) and (2), induced by the bar-cobar adjunction. This is contained in Theorem \ref{thm:barcobarQequivalence} below. 
	\end{enumerate} 
	As an aside, note that there is a philosophical similarity between global Koszul duality and the Quillen equivalence relating topological spaces and simplicial sets; here coalgebras are viewed as analogous to simplicial sets and algebras to topological spaces.
	
	Since the result above concerns non-(co)augmented and curved algebras and coalgebras, it is not directly comparable with local Koszul duality. To make a comparison, one should restrict it to the undercategories of $\ground$, the ground field. Then we obtain a Quillen equivalence between the categories of dg algebras (albeit still with MC equivalences, a finer relation than quasi-isomorphism) and coaugmented curved coalgebras. There is a further coreflective Quillen adjunction relating this Quillen equivalence to local Koszul duality (Proposition \ref{qiMC}). Somewhat imprecisely, one can say that global Koszul duality becomes local Koszul duality when restricted to conilpotent coalgebras. 
	
	In a similar way, global Koszul duality, when restricted to \emph{pointed} coalgebras, becomes the categorical Koszul duality of \cite{HL2020}, see Theorem \ref{categoricalcorefl} below. One informal and surprising consequence of it is that the homotopy category of small dg categories is a coreflective subcategory of the homotopy category of curved algebras. Under this correspondence, dg algebras correspond to dg categories with a distinguished object.
	
	An important invariant of a dg algebra is its \emph{derived category}; it is well-known that two quasi-isomorphic dg algebras have equivalent derived categories. There is also the notion of a \emph{derived category of the second kind}, cf.\ \cite[Chapter 7]{PosSurvey}. Of greatest relevance for global Koszul duality is the \emph{compactly generated} derived category of the second kind, cf. \cite{GL20}. It can be defined for a dg algebra but also for a curved algebra (which does not have an ordinary derived category). We show that two MC equivalent curved algebras have equivalent compactly generated derived categories of the second kind, cf.\ Corollary \ref{MCisDII}. This leads one to take up the study of derived Morita equivalences of the second kind; the associated model structures appear to be an interesting problem for future work. The notion of an MC equivalence should also be relevant to the study of categories of \emph{matrix factorisations}, cf. \cite{Dyckerhoff11}, since these categories are formed by $\mathbb{Z}/2$-graded twisted modules and so are invariant under MC equivalences (cf.\ Remark \ref{mfrmk}).
	
	Our main application is the construction of homotopy invariant moduli spaces in various situations. The prototypical example, to which any other moduli space considered here is reduced, is the moduli space of MC elements in a given curved algebra. Two MC equivalent curved algebras have isomorphic moduli of MC elements, essentially by definition. Note that these moduli are not quasi-isomorphism invariant, even when the notion of quasi-isomorphism makes sense (i.e. for dg algebras). We introduce the notion of an MC stack, which is, roughly speaking, a functor on the category of finite dimensional curved algebras with values in an $\infty$-category that preserves finite homotopy limits. We consider two main examples of MC stacks -- those taking values in simplicial sets $\sSet$ and in dg categories $\dgcat$. In the latter case we call them \emph{noncommutative moduli spaces}. We prove a representability result stating that any MC stack is representable by a curved coalgebra (or equivalently a pseudocompact curved algebra), cf. Corollary \ref{stackrepcor} and Proposition \ref{ncmodrep}. Since a curved coalgebra corresponds, by global Koszul duality, to a curved algebra, one can also say that any MC stack is controlled by a curved algebra, which is the formalism often seen in the deformation theory literature. We also give a definition of the tangent space to an MC stack, and compute some examples. Our treatment of tangent spaces roughly follows Lurie's in \cite{luriedagx}.
	
	When restricted to the category of finite dimensional \emph{local} dg algebras, we obtain a notion of derived deformation functor that is essentially equivalent to Lurie's notion of a noncommutative formal moduli problem (cf.\ \cite[Section 3.2]{luriedagx}); more precisely, our theory is a nonconnective version of Lurie's. Among global moduli problems that can be handled by MC stacks is the moduli of objects in a dg category. This problem has already been treated in \cite{toenvaquie}; one advantage of our approach is that there is a representing MC stack for this moduli problem with no restrictions on the dg category in question.
	
	Other global moduli problems represented by MC stacks include flat connections on vector bundles on smooth manifolds, holomorphic structures on almost complex bundles on complex analytic manifolds, dg modules up to quasi-isomorphism over a dg algebra, twisted modules over a dg algebra up to homotopy, and $\infty$-local systems on topological spaces.
	
	At the same time, there are natural moduli problems that cannot be included in our framework. These are given by functors defined on \emph{commutative} algebras of various flavours (simplicial, dg etc.) that do not have a natural extension to associative algebras. Such are moduli of complex structures on a given smooth manifold, of various operadic algebras (commutative, associative, Lie etc.) on a given graded vector space or, more generally, problems `controlled' by an $L_\infty$ algebra. In this connection it is natural to ask whether an analogue of global Koszul duality exists in other contexts, e.g. between dg Lie algebras and cocommutative coalgebras. It is clear that the solution does not carry over from the associative case in a straightforward manner. For example, it is well-known that cocommutative coalgebras split as a coproduct of conilpotent ones (see e.g.\ \cite{KDcocomm}), which is not true for coassociative coalgebras, and this indicates that some new ideas are needed to construct an analogue of global Koszul duality in other contexts. It seems that the framework developed in the present paper should work for non-$\Sigma$ operads without too many changes (cf. \cite{Ginzburg94}) but it is not clear whether examples of such operadic global Koszul duality are abundant `in nature'.

	\subsection{Organisation of the paper} 
 
   Section \ref{section:bar} contains background material on curved algebras and coalgebras, bar and cobar constructions, and Maurer--Cartan elements.

    Section \ref{section:higher} introduces the notion of $n$-homotopy for curved algebras, curved coalgebras, and simplicial sets. For $n=1$ this reduces to ordinary homotopy of simplicial sets and derivation homotopy of (co)algebras. The case $n=3$ is the one relevant to global Koszul duality but other cases, particularly $n=\infty$, have independent interest; one can speculate that there are various interesting $(\infty,1)$-category structures on $\sSet$ based on the notion of $n$-homotopy.

    Section \ref{section:twisted} surveys the categories of twisted (co)modules over curved (co)algebras and the associated compactly generated coderived categories. 

    Section \ref{section:MC} introduces the Maurer--Cartan dg category $\MCdg(A)$ associated to a curved algebra $A$. We show that the functor $A\mapsto\MCdg(A)$ admits a left adjoint. We then develop further the properties of MC elements in curved algebras and their homotopies.
    
    Section \ref{section:mclifting}, which is independent of the others, shows that if $A \to B$ is a square zero extension of curved algebras then the induced map $\MCdg(A) \to \MCdg(B)$ is a fibration. We moreover analyse the fibres in terms of the fibre of $A \to B$.

    Section \ref{section:structure} is a detailed analysis of the structure theory of curved coalgebras. In particular we decompose injections of curved coalgebras as relative cell complexes for two simple kinds of morphism: injections between finite dimensional cosemisimple curved coalgebras, and cosquare zero extensions of finite dimensional curved coalgebras. Along the way we develop a structure theory of finite dimensional curved semisimple algebras, which turns out to be a mild extension of the ordinary classification of finite dimensional semisimple algebras over a perfect field.

    In section \ref{section:morita} we introduce the notion of {\normalfont II}-Morita equivalences between curved algebras. We use our structure theorems and our results on lifting MC elements to show that these are preserved by convolution with arbitrary curved coalgebras.
	
	Section \ref{section:MC2} contains the crucial definition of an MC equivalence: a morphism $C\to C'$ of curved coalgebras is an MC equivalence exactly when it induces a quasi-equivalence of dg categories $\MCdg\Hom(C',A)\to \MCdg\Hom(C,A)$ for every curved algebra $A$. MC equivalences for curved algebras are defined analogously. Using our results on {\normalfont II}-Morita equivalence, we prove that the unit and counit of the bar-cobar adjunction between curved algebras and curved coalgebras are MC equivalences (Theorem \ref{thm:barcobarequiv}).
	
	In section \ref{section:strongcof} we define the notion of a strong cofibration of curved coalgebras as a map $C\to C'$ inducing a fibration of dg categories $\MCdg\Hom(C',A)\to \MCdg\Hom(C,A)$ for every curved algebra $A$. These will be our model-theoretic cofibrations of curved coalgebras. We use our structure theorems to show that, surprisingly, the condition of being a strong cofibration is equivalent to simply being injective (Theorem \ref{thm:strongcof}). 

    Section \ref{section:strongfib} contains analogous material on strong fibrations, i.e.\ those maps $A \to A'$ of curved algebras which induce a fibration $\MCdg\Hom(C,A) \to \MCdg\Hom(C,A')$ for every curved coalgebra $C$. These will, dually, be our model-theoretic fibrations of curved algebras. We obtain a characterisation in terms of lifting properties.

	Section \ref{section:MCmodel} is dedicated to proving the main result of the paper described above -- the existence of model structures on the categories of curved algebras and coalgebras, where the weak equivalences are the MC equivalences, and such that the bar-cobar adjunction is a Quillen equivalence between them. We also give corresponding results for various slice categories such as dg algebras, augmented dg algebras, dg coalgebras, and coaugmented dg coalgebras in Theorem \ref{slicethm}. We discuss the model enrichment of algebras over coalgebras (Theorem \ref{modelenrich}) and give small sets of generating (co)fibrations. We also discuss the relation of our global Koszul duality to Positselski's conilpotent Koszul duality and Holstein--Lazarev's categorical Koszul duality.
	
	Finally, in Section \ref{section:moduli} we give our main application of global Koszul duality to the construction of global moduli spaces or MC stacks as described above. We discuss how our approach compares to Lurie's \cite{luriedagx} and give various examples.

    The following graph illustrates the dependencies of the various sections:

       $$ \begin{tikzcd}
        \ref{section:bar} \ar[dr] & \ref{section:higher}\ar[r] & \ref{section:twisted}\ar[dl] & \ref{section:structure}\ar[d]\\
        \ref{section:moduli} & \ref{section:MC}\ar[r] & \ref{section:mclifting}\ar[r] & \ref{section:morita}\ar[d] \\
        \ref{section:MCmodel}\ar[u] & \ref{section:strongfib}\ar[l] & \ref{section:strongcof}\ar[l] & \ref{section:MC2}\ar[l]
    \end{tikzcd}$$
	
	\subsection{Acknowledgements} We are grateful to B. Keller, J. Chuang, J. Holstein, L. Positselski and J. Pridham for many conversations we have had over the years, directly or indirectly related to the topics of the present work. We would also like to thank S. Opper, G. Raptis, and K. Rasmussen for their helpful comments.

\subsection{Notation}
Throughout we will work over a fixed perfect field $\ground$. This perfect hypothesis will be imposed primarily because we want finite dimensional semisimple $\ground$-algebras to be separable. Unadorned tensor products and Homs will be assumed to be taken over $\ground$. Simplicial chain coalgebras and cochain algebras will be understood with coefficients in $\ground$.

We will primarily work with cohomologically graded chain complexes over $\ground$, although we will have occasion to use homological grading when talking about chain coalgebras of simplicial sets. We will always make clear which grading convention is used. We will denote cohomological gradings with superscripts and homological gradings with subscripts; to convert between these simply put $A_i = A^{-i}$. For a complex $A$ 
we denote by $\Sigma A$ its suspension, or shift, given in homological grading by $(\Sigma A)_i = A_{i-1}$. When working in the cohomological grading we will also denote the shift $\Sigma A$ by $A[1]$, since we have $A[1]_i = A_{1+i}$. We will denote the inverse functor of $[1]$ by $[-1]$.

We denote the category of unital dg-$\ground$-algebras by $\alg$ and the category of counital dg-$\ground$-coalgebras by $\cog$. We will also consider the category $\aalg$ of $\ground$-augmented dg algebras and the category $\acog$ of $\ground$-coaugmented dg coalgebras, both of which are obtained as slice categories of objects over and under $\ground$, respectively. We denote by $\calg$ the category of curved $\ground$-algebras, and by $\ccog$ the category of curved $\ground$-coalgebras. For results and terminology about coalgebras we refer the reader to Positselski's recent survey \cite{PosSurvey}.

We denote by $\calgp$ the category obtained from $\calg$ by formally adjoining an initial object $\varnothing$ and we denote by $\ccogp$ the category obtained from $\ccog$ by formally adjoining a final object $*$.

The category of small dg categories will be denoted by $\dgcat$.

The category of simplicial sets will be denoted by $\sSet$.

If $X$ is a differential graded object, we denote its underlying graded object by $X^\#$. Similarly if $X$ is a curved (co)algebra, we denote its underlying graded (co)algebra by $X^\#$.

If $Y$ is a topological space or a simplicial set, we use $C(Y,\ground)$ to denote the dg algebra of normalised simplicial chains on $Y$ with values in $\ground$. 
 
\section{The bar-cobar adjunction}\label{section:bar}
The classical bar and cobar constructions give an adjunction between the categories of coaugmented conilpotent dg coalgebras and augmented dg algebras. In this section, we describe a non-conilpotent version of this adjunction, before extending it to curved algebras. The cobar construction remains the same, but we must replace the bar construction with the {extended bar construction}; the difference is essentially that one must replace the cofree conilpotent coalgebra functor (the tensor coalgebra) with the cofree coalgebra functor (a much larger coalgebra). We primarily follow \cite{GL20}, mentioning also that the extended bar construction already appears in \cite{AnelJoyal}.

\subsection{Pseudocompact algebras}
If $V_i$ is a cofiltered system of finite dimensional $\ground$-vector spaces, then the cofiltered limit $\varprojlim_iV_i$ can be equipped with the inverse limit topology, regarding a finite dimensional vector space as discrete. Say that a topological vector space is {pseudocompact} if it is isomorphic to such a cofiltered limit. Any vector space $V$ is the filtered colimit of its finite dimensional subspaces, so $V^*$ is canonically pseudocompact. Similarly, if $A_i$ is a cofiltered system of finite dimensional $\ground$-algebras then its limit also can be equipped with the inverse limit topology; such a topological algebra is also called {pseudocompact}. A topological algebra is pseudocompact precisely when it is complete Hausdorff, with a basis of finite-codimensional ideals.

If $C$ is a coalgebra, it is the union of its finite dimensional subcoalgebras $C_i$ by a well-known result of Sweedler, and hence its linear dual $C^*\cong\varprojlim C_i$ is naturally a pseudocompact algebra. If $A$ is a pseudocompact algebra, then its topological dual $A^*$ is a coalgebra. With the convention that the dual of a pseudocompact algebra always means the topological dual, we have $C^{**}\cong C$ and $A^{**}\cong A$. Moreover, the linear and topological duals together form a contravariant equivalence between the category of pseudocompact algebras and the category of coalgebras.

If $A$ is any algebra, its {pseudocompact completion} is the pseudocompact algebra $\check A$ obtained as the cofiltered limit of the system of finite dimensional quotients of $A$. Pseudocompact completion is functorial, and in fact the left adjoint to the functor forgetting the topology.

If $V$ is a vector space then $T(V)$ denotes the tensor algebra on $V$. If $V=\varprojlim V_i$ is a pseudocompact vector space, its {pseudocompact tensor algebra} is the pseudocompact algebra $\check{T}(V)\coloneqq \varprojlim_i\check{T}(V_i)$. The functor $\check T$ is left adjoint to the forgetful functor from pseudocompact algebras to pseudocompact vector spaces. If $U$ is any vector space then its linear dual $U^*$ is pseudocompact, and the topological dual of $\check{T}(U^*)$ is precisely the cofree coalgebra on $U$. Note that the cofree \textit{conilpotent} coalgebra on $U$ is simply given by the tensor coalgebra $T(U)$ with the usual deconcatenation coproduct.

\subsection{The bar and cobar constructions}
A dg algebra $A$ is {augmented} if the canonical unit morphism $\ground \to A$ admits a retract $A \to \ground$; an {augmentation} on $A$ is a choice of such a retract. The {augmentation ideal} is the ideal $\bar A\coloneqq \ker(A \to \ground)$; it is a nonunital subalgebra of $A$. Similarly, a dg coalgebra $C$ is {coaugmented} if the counit morphism $C \to \ground$ admits a section $\ground \to C$. In this case the {coaugmentation coideal} is the quotient $\bar C \coloneqq\mathrm{coker}(\ground \to C)$, which inherits the structure of a noncounital coalgebra from $C$.

\begin{defi}
\hfill	
\begin{enumerate}
	\item Let $C$ be a coaugmented dg coalgebra. The {cobar construction} on $C$ is the dg algebra $\Omega C$ whose underlying graded algebra is $T\Sigma^{-1}\bar C$, the tensor algebra on the desuspension of the coaugmentation coideal of $C$. The differential is the usual cobar differential combining the differential and the comultiplication on $C$.
	\item Let $A$ be an augmented dg algebra. The {dual extended bar construction} is the pseudocompact dg algebra whose underlying pseudocompact graded algebra is $\check T(\Sigma^{-1} \bar{A}^*)$, the pseudocompact tensor algebra on the desuspension of the dual of the augmentation ideal of $A$. The differential combines the differential on $A$ with the multiplication; see \cite[Definition 2.5]{GL20} for a concrete formula. The {extended bar construction} on $A$ is the dg coalgebra $\check B A$ obtained as the topological dual of its dual extended bar construction.

	\end{enumerate}	

	\end{defi}

\begin{rem}
The underlying graded coalgebra of $\check B A$ is the cofree coalgebra on $\Sigma \bar A$, the suspension of the augmentation ideal of $A$. The differential is similar to the usual bar differential. 
	\end{rem}

\subsection{Maurer--Cartan elements}
Let $X$ be a not necessarily unital dg algebra. A {Maurer--Cartan element} (MC element for short) is an element $x\in X$ of cohomological degree one such that $dx+x^2=0$. The set of all MC elements of $X$ is denoted $\MC(X)$; this construction is functorial in $X$.

Let $C$ be a dg coalgebra and $A$ be a dg algebra. Then the dg vector space $\Hom(C,A)$ admits a product, the convolution product, induced from the multiplication on $A$ and the comultiplication on $C$. More precisely, if $\Delta:C\to C\otimes C$ is the coproduct on $C$ and $m:A\otimes A\to A$ is the product in $A$ then the formula
\[
fg=m\circ(f\otimes g)\circ\Delta
\]
determines a dg algebra structure on $\Hom(C,A)$ called the \emph{convolution algebra} associated to $C$ and $A$. Note that this construction still makes sense in case $C$ or $A$ are non(co)unital.

Observe that if one regards $C^*\cong\varprojlim_i C_i^*$ as a pseudocompact algebra, the dg algebra $\Hom(C,A)$ is the same as the completed tensor product $C^*\hat\otimes A\coloneqq \varprojlim_i( C^*_i\otimes A)$.

\begin{prop}\label{mcrep}\hfill
	\begin{enumerate}
		\item Let $C$ be a coaugmented dg coalgebra and $X$ any dg algebra. Then there is a natural isomorphism $\Hom(\Omega C,X)\cong \MC(\Hom(\bar C,X))$.
		\item Let $Y$ be a dg coalgebra and $A$ an augmented dg algebra. Then there is a natural isomorphism $\Hom(Y,\check B A)\cong \MC(\Hom(Y,\bar A))$.
		\end{enumerate}
	\end{prop}
\begin{proof}
The first part is completely standard: a map $\Omega C \to X$ is the same thing as a linear map $f:\Sigma^{-1}\bar C \to X$ compatible with the differentials, and this latter compatibility is precisely the MC equation for $f$. The second claim is similar; cf.\ \cite[Proposition 2.6]{GL20}.
\end{proof}
Just as in the conilpotent setting, bar and cobar are adjoint. If $C,C'$ are coaugmented coalgebras, let $\Hom_*(C,C')$ denote the set of morphisms between them respecting the coaugmentation. We use the analogous notation for maps of augmented algebras.
\begin{cor}
Let $C$ be a coaugmented dg coalgebra and $A$ an augmented dg algebra. Then there is a natural isomorphism $$\Hom_*(\Omega C,A)\cong \Hom_*(C,\check B A).$$
	\end{cor}
\begin{proof}
	Both are naturally isomorphic to $\MC(\Hom(\bar C,\bar A))$.
	\end{proof}

\begin{rem}
	A {pro-finite dimensional algebra} is a pro-object in the category of finite dimensional dg algebras, i.e.\ a cofiltered diagram of finite dimensional algebras. The category of pro-finite dimensional algebras is equivalent to the category of pseudocompact algebras, and one can phrase the above definitions and results in this language; see \cite{B22} for an example of this approach in the conilpotent setting. We will freely pass between these two notions when necessary.
\end{rem}

\subsection{Curved (co)algebras}
In this section we follow \cite{HL2020}. A {curved algebra} is a graded algebra $A$ equipped with a cohomological degree $1$ derivation $d$ and an element $h \in A^2$ such that $d(h)=0$, and $d^2(a)=[h,a]=ha-ah$ for all $a\in A$. We call $h$ the {curvature}. Note that a dg algebra is precisely a curved algebra with zero curvature.

A morphism of curved algebras $A \to B$ is a pair $(f,b)$ where $f:A \to B$ is a map of graded algebras, and $b \in B$ is a cohomological degree 1 element satisfying the formulas
\begin{enumerate}
	\item $f(da)=d(fa) + [b,fa]$
	\item $f(h_A) = h_B + db + b^2$.
	\end{enumerate} 
Morphisms compose by putting $(g,b)(f,a)=(gf, b+g(a))$. In particular, if $A$ and $B$ are dg algebras, then a morphism $A \to B$ in the category of curved algebras is a pair $(f,b)$ where $b\in \MC(B)$ and $f:A \to B^{[b]}$ is a dg algebra morphism. 

Say that a morphism is uncurved if $b=0$. Note that any morphism $(f,b)$ of curved algebras decomposes as the composition $(\id,b)(f,0)$ of an uncurved morphism with an isomorphism. We will often use this to replace a general morphism by an uncurved one.

A {Maurer-Cartan element} of a curved algebra $A$ is an element $x\in A^1$ such that $h+dx+x^2=0$. Note that $\MC(A)$ is in bijection with $\Hom(k,A)$. Putting $d^x(a)=da+[x,a]$, one can compute that $d^xd^x=0$. Moreover, $d^x$ is a derivation, so that $A^x=(A,d^x)$ is a dg algebra. The morphism $(\id, x): A^x \to A$ is an isomorphism of curved algebras. In particular, if $A$ admits an MC element then it is isomorphic to a dg algebra.

Recall that if $\C$ is a category and $c\in \C$ an object, then the undercategory $\C_{c/}$ is the category whose objects are maps $c \to x$ and whose morphisms are commutative triangles. If $\C$ has coproducts then the projection map $\pi:\C_{c/} \to \C$ is a right adjoint, whose left adjoint $\pi^!$ sends $x$ to the natural map $c \to c\sqcup x$. Similarly one may define an overcategory $\C_{/c}$ whose objects are maps $x \to c$.

\begin{lem}\label{algslice1}
	There is an equivalence $\alg\simeq \calg_{\ground/}$.
	\end{lem}
\begin{proof}
Every dg algebra $A$ admits a morphism $\ground \to A$ of dg algebras which is in particular a morphism of curved algebras, so that the natural inclusion $\alg\to \calg$ has image contained in $\calg_{\ground/}$. Let $i:\alg \to \calg_{\ground/}$ be the natural inclusion; we wish to show that it is an equivalence. A map $\ground \to A$ of curved algebras is the data of an MC element $x \in A$. In this case, $A$ is isomorphic in the undercategory to the dg algebra $A^x$, so that $i$ is essentially surjective. To see that it is fully faithful, let $A$ and $B$ be two dg algebras. The natural map $\ground \to A$ corresponds to the MC element $0$, so a map $A \to B$ in $\calg_{\ground/}$ is given by a pair $(f,0)$, where $f$ is a map of graded algebras satisfying $f(da)=d(fa)$. This is precisely a morphism of dg algebras.
	\end{proof}

If $\C$ is a category and $f:c\to d$ is a morphism, then one may form the double slice category $\C_{c/d}$, whose objects are diagrams $c \to x \to d$ factoring $f$ and whose morphisms are given by maps $x\to x'$ making the obvious square commute. One can check that there are equivalences $(\C_{c/})_{/f}\cong \C_{c/d} \cong (\C_{/d})_{f/}$. In particular the projection functor $\C_{c/d} \to \C$ need not have an adjoint. If $\C$ is the category $\calg$ and $f=\id_\ground$, then $\calg_{\ground/\ground}\simeq \alg_{/\ground}\simeq \aalg$ is the category of augmented dg algebras.

If $A_1,A_2$ are two curved algebras, then the tensor product $A_1\otimes A_2$ is a curved algebra: the differential has the usual formula and the curvature is given by $h_1\otimes1 + 1\otimes h_2$. 

One can similarly define {pseudocompact curved algebras} by adding the word `pseudocompact' to the above definition. Alternately, a pseudocompact curved algebra can be viewed as a pro-object in finite dimensional curved algebras. We denote the category of pseudocompact curved algebras by $\pccalg$.

A {curved coalgebra} is a graded coalgebra $C$ equipped with a coderivation $d$ of cohomological degree $1$ and a cohomological degree $2$ functional $h:C\to \ground$ such that $(C^*, d^*, h^*)$ is a pseudocompact curved algebra. This is equivalent to the two conditions $h\circ d =0$ and $d^2(x)=h(x^1)x^2 - x^1h(x^2)$, where we use Sweedler notation $\Delta(x)=x^1\otimes x^2$. Morphisms of curved coalgebras are defined analogously as morphisms of pseudocompact curved algebras. 

Since every curved coalgebra is the colimit of its finite dimensional sub-curved coalgebras \cite[Lemma 3.32]{HL2020}, the linear dual provides an equivalence between $\ccog$ and the opposite of the category of pseudocompact curved algebras.
\begin{lem}\label{cogslicelem}
	There is an equivalence $\cog\simeq \ccog_{/\ground}$.
\end{lem}
\begin{proof}
	This is a pseudocompact version of \ref{algslice1}.
\end{proof}

\subsection{Curved bar and cobar constructions}
Let $C$ be a curved coalgebra and $A$ a curved algebra. As in \cite{GL20}, one can define a {cobar construction} $\Omega C \in \calg$ and an {extended bar construction} $\check BA \in \ccog$. The {convolution algebra} $\Hom(C,A)$ is defined to be the completed tensor product $C^* \hat \otimes A$, where we regard $A$ as a constant pro-object in curved algebras. Note that $\Hom(C,A)$ need not be a pseudocompact curved algebra, unless $A$ was finite dimensional. 

Observe that since $(C\otimes D)^*$ is naturally isomorphic to $C^*\hat\otimes D^*$ as pro-objects in curved algebras, we may deduce the hom-tensor adjunction for convolution algebras: there is a natural isomorphism
$$\Hom(C\otimes D,A)\cong \Hom(C,\Hom(D,A)).$$

We let $\MC(C,A)$ denote the set $\MC\Hom(C,A)$ of MC elements in the convolution algebra.

\begin{rem}\label{proMC}
	If $Z=\{Z_i\}_i$ is a pro-object in curved algebras then we may define $\MC(Z)\coloneqq \Hom(k, Z)$, where we take the Hom in the category of pro-objects in curved algebras. It is easy to see that $\MC(Z)\cong \MC(\varprojlim_iZ_i)$, naturally in $Z$. In particular when taking $\MC$ we may forget that $\Hom(C,A)$ naturally has the structure of a pro-object in curved algebras. 
	\end{rem}
Bar and cobar are adjoints:
\begin{prop}[{\cite[Proposition 4.4]{GL20}}]
	There are isomorphisms, natural in $C$ and $A$,
	$$\Hom(\Omega C,A)\cong \MC\Hom(C,A)\cong\Hom(C,\check B A).$$
	\end{prop}
\begin{cor}\label{htconv}
	Let $C,D$ be curved coalgebras and let $A$ be a curved algebra. There is a natural isomorphism $$\calg(\Omega(C\otimes D),A)\cong \calg(\Omega C,\Hom(D,A)).$$
\end{cor}
\begin{proof}
	There are natural isomorphisms 
	\begin{align*}	
		\calg(\Omega(C\otimes D),A)&\cong \MC\Hom(C\otimes D,A)\\
		&\cong \MC\Hom(C,\Hom(D,A))\\
		&\cong \calg(\Omega C,\Hom(D,A))\\
	\end{align*}
	where in the middle we use the hom-tensor adjunction for convolution algebras.
\end{proof}

Suppose that $L:\mathcal{C}\leftrightarrow\mathcal{D}:R$ is an adjunction. Pick $c\in \mathcal{C}$ and put $d=Lc$. If the unit $c \to Rd$ is an isomorphism, then one can check that the adjunction slices to an adjunction $L:\mathcal{C}_{/c}\leftrightarrow\mathcal{D}_{/d}:R$. There is also a dual version for undercategories. In particular, suppose that $\C$ is the category of curved coalgebras and $c=\ground$. Certainly the unit condition is satisfied, so we obtain a sliced adjunction $\Omega:\ccog_{/\ground}\leftrightarrow\calg_{/\ground}:\check B$. Slicing again, we obtain another adjunction $\Omega:\ccog_{\ground/\ground}\leftrightarrow\calg_{\ground/\ground}:\check B$ which one can see is the usual extended bar-cobar adjunction between coaugmented dg coalgebras and augmented dg algebras. This reasoning proves:
\begin{prop}The bar-cobar adjunction
	$$\Omega:\ccog\longleftrightarrow\calg:\check B$$
	slices to adjunctions
	$$\Omega:\cog\longleftrightarrow\acalg:\check B$$
	$$\Omega:\accog\longleftrightarrow\alg:\check B$$
	$$\Omega:\acog\longleftrightarrow\aalg:\check B$$
	where the bottom adjunction is the usual extended bar-cobar adjunction.
	\end{prop}
\begin{proof}
The first adjunction uses \ref{cogslicelem} to identify $\cog\simeq \ccog_{/\ground}$. Note that the equivalence $\acalg\simeq \calg_{/\ground}$ is definitional. The second adjunction is similar and uses \ref{algslice1} instead. The third makes use of the equivalences $\calg_{\ground/\ground}\simeq \aalg$ and $\ccog_{\ground/\ground}\simeq \acog$.
	\end{proof}

\subsection{Limits and colimits}

The category $\calg$ of curved algebras lacks an initial object. We denote by $\calgp$ the category obtained from $\calg$ by formally adjoining an initial object, which we denote by $\varnothing$. Note that the curved algebra $0$ is a final object in both categories. By construction there are no maps $\ground \to \varnothing$ and hence by \ref{algslice1} we have equivalences $$(\calgp)_{\ground/}\simeq \calg_{\ground/}\simeq \alg.$$

Similarly, $\ccog$ lacks a final object, and by formally adjoining a final object $*$ we obtain a category which we denote by $\ccogp$; the curved coalgebra $0$ is initial in both categories. There is an equivalence $\cog\simeq (\ccogp)_{/\ground}$.

\begin{prop}\label{limscolims}
	The categories $\calgp$ and $\ccogp$ are complete and cocomplete. The forgetful functor $\calg \to \mathbf{grAlg}_\ground$ preserves and reflects limits, and creates products and cofiltered limits. The forgetful functor $\ccog \to \mathbf{grCog}_\ground$ preserves and reflects colimits, and creates coproducts and filtered colimits.
\end{prop}
\begin{proof}
This is essentially \cite[Lemma 3.30]{HL2020}, which is itself a non-conilpotent version of \cite[Lemma 9.2]{Positselski11}.
	\end{proof}
We remark that $\calg \to \mathbf{grAlg}_\ground$ does not create limits, since $\calg$ is not closed under limits and $\mathbf{grAlg}_\ground$ is: the issue is that $\calg $ does not have equalisers. A similar statement holds for coalgebras.
\begin{prop}
	The categories $\calgp$ and $\ccogp$ are locally presentable. Every curved coalgebra is the colimit of its finite dimensional sub-curved coalgebras.
	\end{prop}
\begin{proof}
	The statements about coalgebras are contained in \cite[Lemma 3.32]{HL2020} and its proof. For algebras, it is enough to show that $\calg$ is locally presentable. To do this, for every $n>0$ and every $\underline {a} \in \Z^n$ we define a curved algebra $K(n,\underline {a})$ as follows. The generators of the underlying graded algebra of $K(n,\underline {a})$ are $x_1,\ldots, x_n$, where $x_i$ has cohomological degree $a_i$, along with $y_1,\ldots, y_n$, where $y_i$ has cohomological degree $a_{i+1}$, along with a single generator $h$ of cohomological degree $1$. The curvature element is $h$, and the differential is defined by $dx_i=y_i$, $dy_i=[h,x_i]$, and $dh=0$. Clearly if $A$ is a curved algebra then every finitely generated subalgebra of $A$ is in the image of some morphism $K(n,\underline {a}) \to A$; we can even choose the morphism to be of the form $(f,0)$. Since every $A$ is the filtered colimit of its finitely generated subalgebras, the $K(n,\underline {a})$ generate $\calg$ under filtered colimits. Moreover, each $K(n,\underline {a})$ is small, because it is finitely generated.
	\end{proof}

If we modify the bar and cobar functors slightly by declaring that $\check B(0)\coloneqq *$, $\check B (\varnothing)\coloneqq 0$, $\Omega(0)\coloneqq \varnothing$, and $\Omega(*)\coloneqq 0$, then it is easy to see that the bar-cobar adjunction extends to an adjunction between $\ccogp$ and $\calgp$. Moreover, we can extend the definition of the convolution algebra functor by declaring that $\Hom(C,A)$ is $\varnothing$ when either $A=\varnothing$ or $C=*$, with the exception that $\Hom(0,\varnothing)=\Hom(*,0)=0$. If we declare that $\MC(\varnothing)$ is the empty set, then we have isomorphisms $$\Hom(\Omega C,A)\cong \MC\Hom(C,A)\cong\Hom(C,\check B A)$$ as before, where now $C$ is allowed to be $*$ and $A$ is allowed to be $\varnothing$.

\section{Higher homotopy for simplicial sets and curved (co)algebras}\label{section:higher}
In this section we consider the notion of $n$-homotopy for $n=1,2,\ldots,\infty$ in the categories of simplicial sets, curved algebras and curved coalgebras. For $n=1$ this specialises to the usual notion of homotopy between simplicial sets, derivation homotopy between dg algebras, and the dual notion for dg coalgebras. The constructions in all three categories are parallel and the results similar, with some minor variations. Our exposition follows  \cite{CHL21} but has a different emphasis and is more systematic.

Unlike the category of chain complexes or, more generally, the category of (co)modules over a curved (co)algebra where there is an essentially unique reasonable notion of homotopy - i.e.\ chain homotopy - the category of curved (co)algebras admits infinitely many such notions; these are motivated by topological considerations. 

\subsection{Simplicial sets} Consider the category $\C$ having two objects and two mutually inverse morphisms between them:
	\[
\xymatrix{
	0\ar@/^/^{(01)}[r]&1\ar@/^/^{(10)}[l]
}
\]
 Its classifying space $B\C$ is a simplicial set having two nondegenerate simplices $a_n, b_n$ in each dimension $n=0,1,2,\ldots$. The geometric realisation of $B\C$ is the infinite sphere ${\operatorname S}^\infty$ with its standard cell decomposition having two cells in each dimension. We will abuse notation by referring to the simplicial set $B\C$ as ${\operatorname S}^\infty$. For $n=1,2,\ldots$ we let ${\operatorname D}^n$ be the simplicial subset of ${\operatorname S}^\infty$ generated by the nondegenerate simplices $a_i$ for $ i=0,\ldots n$ and $b_k$ for $k=0,\ldots, n-1$. It is clear that the geometric realisation of ${\operatorname D}^n$ is an $n$-dimensional disc; e.g. ${\operatorname D}^1$ is the simplicial interval with two vertices and one nondegenerate simplex connecting them. Note that ${\operatorname S}^\infty$ is a Kan complex, as it is the classifying space of a groupoid. The spaces ${\operatorname D}^n$ are not Kan, although for $n>2$ they are grouplike; i.e. their fundamental categories (a.k.a.\ homotopy categories) are groupoids. The simplicial sets $\operatorname{D}^n$ and $\operatorname{S}^\infty$ all have two vertices and we denote by $i_0$ and $i_1$ the corresponding inclusion maps $*\to \operatorname{D}^n$ and $*\to \operatorname{S}^\infty$, where $*$ is the one-point simplicial set.

\begin{defi}\label{def:homotopysimpl}
	Let $f,g:X\to Y$ be two maps in $\sSet$. \begin{enumerate}
		\item 
		We say that $f$ and $g$ are related by an elementary $n$-homotopy for $n=1,2,\ldots$ if there exists a map
		\[
		h: X\times \operatorname{D}^n\to Y
		\] 
		such that $ h\circ (\id_X\times i_0)=f$ and $h\circ(\id_X\times i_1)=g$. 
		If $f$ and $g$ are related by a zig-zag of elementary $n$-homotopies, we will call them $n$-homotopic and write 
		$f{\sim}_n g$.
		\item		We say that $f$ and $g$ are related by an elementary $\infty$-homotopy if there exists a map
		\[
		h:X\times \operatorname{S}^\infty\to Y
		\] 
		such that $ h\circ (\id_X\times i_0)=f$ and $h\circ(\id_X\times i_1)=g$. If $f$ and $g$ are related by a zig-zag of elementary $\infty$-homotopies, we will call them $\infty$-homotopic and write 
		$f{\sim}_\infty g$.
	\end{enumerate}
\end{defi}
As usual, homotopy of maps gives
rise to the notion of homotopy equivalence.
\begin{defi}
	Two simplicial sets $X$ and $Y$ are $n$-homotopy equivalent if there are maps $f:X\to Y$ and $g:Y\to X$ such that $f\circ g\sim_n\id_Y$ and $g\circ f\sim_n\id_X$ where $n=1,\ldots,\infty$. A simplicial set $n$-homotopy equivalent to $*$ is called $n$-contractible.
\end{defi}
The following proposition summarises the basic properties of $n$-homotopies for simplicial sets.
\begin{prop}\label{prop:homotopysimp} Let $1\leq m<n\leq\infty$.
	\begin{enumerate}
		\item The relation of elementary $n$-homotopy in $\sSet$ is reflexive but not symmetric and not transitive for $n<\infty$. The notion of elementary $\infty$-homotopy is reflexive and symmetric but not transitive.
		\item Let $f,g:X\to Y$ be two  maps between two simplicial sets that are elementary $n$-homotopic. Then for any map $r:Y\to W$ the composites $r\circ f$ and $r\circ g$ are elementary $n$-homotopic. Similarly for any map $k:Z\to X$ the composites $f\circ k$ and $g\circ k$ are elementary $n$-homotopic.
		\item	If two maps between simplicial sets are $n$-homotopic then they are $m$-homotopic. If two simplicial sets are $n$-homotopy equivalent then they are $m$-homotopy equivalent.
		\item Two $m$-homotopic maps are not necessarily $n$-homotopic.
		\item If two maps are $n$-homotopic  then they are homotopic in the usual sense, and so induce the same map on homotopy groups. If two simplicial sets are $n$-homotopy equivalent then they are weakly equivalent.
		\item The simplicial set $\operatorname{D}^1$ is $1$-contractible. The simplicial set ${\operatorname S}^\infty$ is $\infty$-contractible.
		\item If $Y$ is a Kan simplicial set then the relations of $n$-homotopy of maps into $Y$ are equivalent for all $n$ and are equivalence relations. Two Kan simplicial sets are $n$-homotopy equivalent for some $n$ if and only if they are $n$-homotopy equivalent for all $n$.	
	\end{enumerate}
\end{prop}
\begin{proof}
		Reflexivity of the elementary $n$-homotopy is obvious. To show that $\infty$-homotopy is symmetric it suffices to note that $\C$ (and thus, $B\C$) possesses an automorphism switching the two vertices.
	
		On the other hand, consider the two inclusion maps $i_0,i_1:*\to \operatorname{D}^n$. These are elementary $n$-homotopic via the homotopy $\id:\operatorname{D}^n\to \operatorname{D}^n$, however this elementary homotopy is not reversible since there is no endomorphism of $\operatorname{D}^n$ switching its two vertices. 
		
		To see the lack of transitivity, consider  the following category $\mathcal{D}$ with three objects:
		\[
		\xymatrix{
			0\ar@/^/^{(01)}[r]&1\ar@/^/^{(10)}[l]\ar@/^/^{(12)}[r]&2\ar@/^/^{(21)}[l]
		}
		\]
		where $(01)$ and $(12)$ are inverse to $(10)$ and $(21)$ respectively. The simplicial set $B\mathcal{D}$ has three vertices $0,1$, and $2$, and it is obvious that the vertices $1$ and $0$, as well as $1$ and $2$ are elementary $1$-homotopic, but $0$ and $2$ are not since there is no $1$-simplex in $B\mathcal{D}$ connecting $0$ and $2$. This example similarly shows that the relation of elementary $n$-homotopy for $n\leq\infty$ is likewise not transitive. This proves (1).

		For (2), if $h:X\times\operatorname{D}^n\to Y$ or $h:X\times\operatorname{S}^\infty\to Y$ is an elementary $n$-homotopy between $f$ and $g$ then $r\circ h$ is an elementary $n$-homotopy between $r\circ f$ and $r\circ g$. Similarly $(k\times\id_{\operatorname{D}^n})\circ h$ or $(k\times\id_{\operatorname{S}^\infty})\circ h$ is an elementary $n$-homotopy between $f\circ k$ and $g \circ k$.
		
		Claim (3) follows from the fact that $\operatorname{D}^m$ is a simplicial subset of both $\operatorname{D}^n$ and $\operatorname{S}^\infty$, compatibly with the inclusion maps $i_0$ and $i_1$.
		
		To prove (4) consider the two inclusions $i_0,i_1:*\to\operatorname{D}^m$. They are certainly $m$-homotopic; indeed they are elementary $m$-homotopic via the identity map on $\operatorname{D}^m$. Note that they are $n$-homotopic if and only if there exists an elementary $n$-homotopy either from $i_0 $ to $ i_1$ or from $i_1$ to $i_0$. Such an elementary $n$-homotopy is precisely a morphism $\phi:\operatorname{D}^n \to \operatorname{D}^m$ which is a bijection on zero-simplices. By induction on $k$ this forces $\phi$ to be a bijection on $k$-simplices for all $0\leq k< m$. We see that the $m$-simplices $\phi(a_m)$ and $\phi(b_m)$ must both be nondegenerate, and hence must be the same. But their boundaries do not agree.
		
		Claim (5) follows from (3) together with the observation that $1$-homotopy is the ordinary simplicial homotopy. 
		
		For (6), the claim for $\operatorname{D}^1$ is well known; it is not hard to write down an elementary $1$-homotopy from $\operatorname{D}^1 \to * \xrightarrow{i_0} \operatorname{D}^1$ to the identity map of $\operatorname{D}^1$. For the $n=\infty$ case, observe that $\C$ has a strict symmetric monoidal structure given on objects by $0\otimes 0=0$, $1\otimes 1=1$, $0\otimes 1 = 0$, and on morphisms by $(01)\otimes(01)=(01)$, $(10)\otimes(10)=(10)$, and $(10)\otimes(01)=\id_0$. Because the classifying space functor preserves products, this makes $B\C=\operatorname{S}^\infty$ into a monoid in simplicial sets. The multiplication map $\operatorname{S}^\infty\times \operatorname{S}^\infty \to \operatorname{S}^\infty$ can be viewed as an $\infty$-homotopy between $\id_{\operatorname{S}^\infty}$ and a self-map of $\operatorname{S}^\infty$ that factors through the map to a point, demonstrating $\infty$-contractibility of $\operatorname{S}^\infty$. We remark that the monoidal structure restricts to $\operatorname{D}^1$ (but not $\operatorname{D}^n$ for $n>1$) and one can carry out a similar proof in this case.
		
		Finally, since $\operatorname{S}^\infty$ and $\operatorname{D}^n$ are weakly equivalent to a point, they can serve as cylinder objects for $*$ in the standard Quillen model structure on $\sSet$; moreover these cylinder objects are \emph{good} in the sense that the canonical maps $*\sqcup *\to \operatorname{S}^\infty$ and $*\sqcup *\to \operatorname{D}^n$ are cofibrations (i.e. in this case injective maps). It is known that in any model category two maps from a cofibrant object to a fibrant object are homotopic if and only if they are homotopic via any given good cylinder object, and so claim (7) follows.
\end{proof}
\begin{rem}
	The simplicial set $\operatorname{D}^n$ is not $n$-contractible for $1<n<\infty$ since the multiplication map $S^\infty\times S^\infty\to S^\infty$ does not restrict to $\operatorname{D}^n$.
\end{rem}
Since $n$-homotopies of simplicial sets are compatible with compositions, the following definition makes sense. It will not be used in the current paper, however its analogues for dg algebras and dg coalgebras will be.
\begin{defi}
	Let $1\leq n\leq\infty$. The $n$-homotopy category of simplicial sets $\Ho_n\sSet$	is the category whose objects are simplicial sets and morphisms are $n$-homotopy classes of maps.
\end{defi}
\begin{rem}It is possible that there exist interesting model structures based on $n$-homotopies that are finer than the ordinary Quillen model structure but we will not investigate this possibility in the present paper.
	\end{rem}
\subsection{Algebras}
Recall that if $X$ is a simplicial set, then we use $C(X,\ground)$ to denote the dg algebra of normalised simplicial chains on $X$. The underlying chain complex of $C(X,\ground)$ computes the cohomology of $X$, and the multiplication is given by the cup product. If $X$ has finitely many simplices in each dimension, then the underlying chain complex of $C(X,\ground)$ is the linear dual of the normalised chain complex associated to the simplicial vector space $k[X]$.

We will denote by $I^n$ and $I^\infty$ the simplicial chain algebras of ${\operatorname D}^n$ and ${\operatorname S}^\infty$. 
\begin{prop}\label{ddescription}
	The algebra $I^\infty$ is isomorphic to the path algebra of the following graded quiver
 \[\xymatrix
{
e\ar@/_/[r]^s&f\ar@/_/[l]_t
}\]
where the arrows $s$ and $t$ have cohomological degree 1. The differential is given by the formula 
\[
d(x)=[x,s+t].
\]
Furthermore, the algebra $I^n$ is the quotient of $I^\infty$ by the dg ideal spanned by the length $n$ monomial $tsts\cdots$.
\end{prop}
\begin{proof}
The vertices $e$ and $f$ correspond to the two vertices $0$ and $1$ of $\operatorname{S}^\infty$ whereas the elements $s$ and $t$ are dual to the $1$-simplices $(01)$ and $(10)$ respectively. The formulas for the multiplication and differential are straightforward to check. The description of $I^n$ is clear; note that $I^\infty$ has an obvious automorphism switching $s$ and $t$ so one could just as well mod out by the length $n$ monomial $stst\cdots$. 
\end{proof}
We will now give another convenient description of the algebras $I^n$. Recall from \cite[Section 2.6]{Loday92} the notion of the algebra of noncommutative differential forms.
\begin{defi}
	Let $A$ be a discrete algebra. The $A$-bimodule $\Omega^1(A)$ of noncommutative $1$-forms is defined as the kernel of the multiplication map $m:A\otimes A\to A$. 
\end{defi}
Since $m:A\otimes A\to A$ is split as a left $A$-module by the map $a\mapsto a\otimes 1$ we see that $\Omega^1(A)$ is isomorphic as a left $A$-module to $A\otimes A/\ground$, and we will write $adb \in \Omega^1(A)$ for the image of $a\otimes b\in A\otimes A/\ground$ across this isomorphism. The right $A$-module structure is determined from the Leibniz rule $d(ab)=d(a)b\pm adb.$
\begin{defi}
The algebra of noncommutative differential forms on a discrete algebra $A$ is defined to be $\Omega(A)\coloneqq T_A(\Sigma^{-1}\Omega^1(A))$, the bimodule tensor algebra on the desuspension of $\Omega^1(A)$. The formula $d(a)=da$ together with the Leibniz rule determine the structure of a dg algebra on $\Omega(A)$.
\end{defi}	
There is an obvious inclusion $A\hookrightarrow \Omega(A)$ together with the following universal property. Given a dg algebra $B$ and a map of graded algebras $f:A\to B$, then $f$ extends uniquely to a dg algebra map $\Omega(A)\to B$; for this reason $\Omega(A)$ is often referred to as the dg envelope of $A$.

\begin{lem}\label{lem:noncommforms}Let $A\coloneqq\ground\times \ground$ be the product of two copies of the ground field and let $e$ be one of the two nontrivial idempotents of $A$. Then the dg algebra $I^\infty$ is isomorphic to $\Omega(A)$ and the dg algebra $I^n$ is isomorphic to the quotient of $\Omega(A)$ by the ideal generated by $e(de)^n$.
\end{lem}
\begin{proof}
The degree zero part of $I^\infty$ is precisely $A$, which yields a morphism $A \to I^\infty$ of graded algebras. This extends to a unique morphism $\Omega(A) \to I^\infty$ which is defined on $A$-algebra generators by sending $de \mapsto t-s$ and $df \mapsto s-t$. It is straightforward to check that this map is an isomorphism. An easy computation shows that $e(de)^n$ is (up to a sign) the length $n$ monomial $tsts\cdots$ and the claim about $I^n$ follows.
\end{proof}

Lemma \ref{lem:noncommforms} allows one to define certain diagonal maps on the dg  algebras $I^1$ and $I^\infty$ (alternatively we could use the monoid structure on $\operatorname{S}^\infty$). The following result holds.
\begin{prop}\label{prop:bialg}
	There exist unique dg bialgebra structures on $I^\infty$ and $I^1$ for which 
	\[
	\Delta(e)=e\otimes e.
	\]
\end{prop}
\begin{proof} Note that the algebra $A\coloneqq\ground\times \ground\cong \langle e\rangle\oplus\langle f\rangle$ has the structure of a bialgebra specified by $\Delta(e)=e\otimes e$ (and then necessarily $\Delta(f)=1\otimes 1-e\otimes e$). The composite map \[\Delta: A\to A\otimes A\hookrightarrow \Omega(A)\otimes\Omega(A)\cong I^\infty\otimes I^\infty\] extends uniquely, by the universal property of $\Omega(A)$, to a map (denoted by the same symbol)
\[
\Delta: I^\infty\cong \Omega(A)\to \Omega(A)\otimes\Omega(A)\cong I^\infty\otimes I^\infty
\]
which one can check is coassociative (it is enough to check this on algebra generators), giving $I^\infty$ the structure of a dg bialgebra. The two-sided ideal generated by $ede$ is also a two-sided coideal, and hence the quotient $I^1\cong\Omega(A)/(ede)$ is a dg bialgebra.
\end{proof}
Note that for $1\leq n \leq \infty$, the algebra $I^n$ has two `evaluation' maps 
$\ev_{0,1}:I^n\to\ground$ obtained by setting either $e$ or $f$ to zero. This suggests that the $I^n$ can play a role of path objects for $\ground$ in the category of associative dg algebras. 
	Lemma \ref{lem:noncommforms} underscores the noncommutative nature of the notion of $\infty$-homotopy: the ordinary commutative de Rham algebra of $\ground\times \ground$ has $\ground\times \ground$ for its zeroth (co)homology and thus, cannot serve as a path object for $\ground$.
	
\begin{defi}\label{def:homotopyalg}
	Let $f,g:A\to B$ be two maps in $\calg$. \begin{enumerate}
		\item 
	We say that $f$ and $g$ are related by an elementary $n$-homotopy for $n=1,2,\ldots$ if there exists a map of curved algebras
	\[
	h:A\to B\otimes I^n
	\] 
	such that $(B \otimes \ev_0)\circ h=f$ and $(B \otimes \ev_1)\circ h=g$. 
	If $f$ and $g$ are related by a zig-zag of elementary $n$-homotopies, we will call them $n$-homotopic and write 
	$f{\sim}_n g$.
	\item We say that $f$ and $g$ are related by an elementary $\infty-$homotopy if 
	there exists a map of curved algebras
	\[
	h:A\to B\hat\otimes I^\infty
	\] 
	such that $(B \otimes \ev_0)\circ h=f$ and $(B \otimes \ev_1)\circ h=g$, where the notation $\hat\otimes$ denotes the complete tensor product $B\hat\otimes I^\infty\coloneqq \varprojlim_n B\otimes I^n$. If $f$ and $g$ are related by a zig-zag of elementary $\infty$-homotopies, we will call them $\infty$-homotopic and write 
	$f{\sim}_\infty g$.
\end{enumerate}
\end{defi}
An obvious modification of the above definition gives a notion of homotopy for dg algebras.

Note that an elementary $\infty$-homotopy $f\to g$ is the same thing as a compatible system of elementary $n$-homotopies $f \to g$, one for each $n$.

As usual, homotopy of maps gives
 rise to the notion of homotopy equivalence.
\begin{defi}
	Two curved algebras $A$ and $B$ are $n$-homotopy equivalent if there are maps $f:A\to B$ and $g:B\to A$ such that $f\circ g\sim_n\id_B$ and $g\circ f\sim_n\id_A$ where $n=1,\ldots,\infty$. A dg algebra is called $n$-{contractible} if it is $n$-homotopy equivalent to $\ground$.
\end{defi}
Note that a dg algebra $A$ is $n$-contractible if and only if there is a map $A\to\ground$ such that the identity map on $A$ is $n$-homotopic to the composition $A \to \ground \to A$.
\begin{rem}
The notion of strong (or $\infty$-) homotopy was introduced in \cite{CHL21} alongside the notion of a $K_n$-homotopy, where $K_n$ stood for the simplicial cochain algebra of the simplicial $n$-sphere $\operatorname{S}^n$ inside $\operatorname{S}^\infty$. Since $K_n$ is not acyclic for $n<\infty$, the notion of an $n$-homotopy based on $\operatorname{D}^n\subset \operatorname{S}^n$ appears more natural.

\end{rem}
The next proposition summarises some of the basic properties of $n$-homotopy.
\begin{prop}\label{prop:homotopyalg} Let $1\leq m<n\leq\infty$.
	\begin{enumerate}
		\item The relation of elementary $m$-homotopy is reflexive but not symmetric and not transitive. The notion of elementary $\infty$-homotopy is reflexive and symmetric but not transitive.
		\item Let $f,g:A\to B$ be two maps between curved algebras $A$ and $B$ that are elementary $n$-homotopic. Then for any map $r:B\to C$ the composites $r\circ f$ and $r\circ g$ are elementary $n$-homotopic. Similarly for any for any map $k:D\to A$ the composites $f\circ k$ and $g\circ k$ are elementary $n$-homotopic.
\item	If two maps between curved algebras are $n$-homotopic then they are $m$-homotopic. If two curved algebras are $n$-homotopy equivalent, then they are $m$-homotopy equivalent.
\item Two $m$-homotopic maps are not necessarily $n$-homotopic. 
	\item If two maps of dg algebras are $n$-homotopic  then they are chain homotopic and hence induce the same map on homology. If two dg algebras are $n$-homotopy equivalent then they are quasi-isomorphic. 
	\item The dg algebra $I^1$ is $1$-contractible and the dg algebra $I^\infty$ is $\infty$-contractible. 
\item If $A$ is a cofibrant dg algebra then the notions of $n$-homotopy of maps out of $A$ are equivalent for all $n$ and are equivalence relations. Two cofibrant dg algebras are quasi-isomorphic if and only if they are $n$-homotopy equivalent for any $n$.	
\end{enumerate}
\end{prop}
\begin{proof}
	To show $(1)$, one can just adapt the proof of Claim 1 of \ref{prop:homotopysimp}. Note that $\infty$-homotopy is symmetric because the automorphism switching the idempotents of the pseudocompact dg algebra $I^\infty$ is continuous, since it is induced by the corresponding automorphism of the coalgebra of chains on $\operatorname{S}^\infty$.
	
	Claims (2) and (3) are obtained by the obvious modification of the corresponding statements of Proposition \ref{prop:homotopysimp}.

For (4) consider the identity map $h:I^m\to I^m$; it can be interpreted as an $m$-homotopy between the two evaluation maps $\ev_{0,1}:I^m\to\ground$. The only curved maps $I^n\to\ground$ are the above evaluation maps, so it follows that if $\ev_0$ and $\ev_1$ are $n$-homotopic they must be elementary $n$-homotopic. Such a homotopy is a morphism $I^m\to I^n$ which restricts to an isomorphism in degree zero. The curved algebra $\ground \times \ground $ has exactly two automorphisms, namely the identity and the one that switches $e$ and $f$. By Lemma \ref{lem:noncommforms}, $I^m$ is generated as a dg algebra by its two degree zero idempotents $e$ and $f$. It is easy to see that neither of the two automorphisms of $\ground \times \ground$ extends to an algebra morphism $I^m\to I^n$ .

To prove (5) it suffices to consider the case of a $1$-homotopy $h:A\to B\otimes I^1$. In components such a map consists of two algebra maps $A\to B$ (which must be equal to $f$ and $g$) and another map $\tilde{h}:A\to B$ of homological degree $1$. Compatibility of $h$ with the differential implies $d_B\circ \tilde{h}+\tilde{h}\circ d_B=f-g$, i.e. $\tilde{h}$ is a chain homotopy between $f$ and $g$. 

For (6) note that the diagonal maps $I^1\to I^1\otimes I^1$ and  $I^\infty\to I^\infty\otimes I^\infty$ (cf. Proposition \ref{prop:bialg}) constitute an elementary $1$-homotopy or $\infty$-homotopy between the identity maps on $I^1$ and $I^\infty$ and maps  that factors through $\ground$. This implies that $I^1$ is $1$-contractible and $I^\infty$ is $\infty$-contractible. 

Finally, since $I^n$ and $I^\infty$ are quasi-isomorphic to $\ground$, they can serve as path objects for $\ground$ in the standard model category of dg algebras (with quasi-isomorphisms as weak equivalences); moreover these path objects are \emph{good} in the sense that the canonical maps $I^n\to\ground\times \ground$ and $I^\infty\to\ground\times \ground$ are fibrations (in this case surjective maps).  Claim (7) follows readily.
\end{proof}	
\begin{rem}
	For $1<n<\infty$ the dg algebra $I^n$ is not $n$-contractible, since the diagonal map $I^\infty\to I^\infty\otimes I^\infty$ does not descend to a diagonal map on $I^n$.
\end{rem}
Since $n$-homotopies of dg algebras are compatible with compositions, the following definition makes sense:
\begin{defi}
Let $1\leq n\leq\infty$. The n-homotopy category of dg-algebras $\Ho_n\Alg$	is the category whose objects are dg algebras and morphisms are n-homotopy classes of maps.
\end{defi}
\subsection{Coalgebras}We now describe how $n$-homotopies are constructed for curved coalgebras. The idea is to dualise the construction for algebras. For $3\leq n \leq \infty$, let $I_n$ be the linear dual of the dg algebra $I^n$. Because $I^n$ is finite dimensional in each degree, the complex $I_n$ naturally admits the structure of a dg coalgebra. For a more direct construction, when $n<\infty$ one could take $I_n$ (resp.\ $I_\infty$) to be the coalgebra of normalised chains on the simplicial set $\operatorname{D}^n$ (resp.\ $\operatorname{S}^\infty$). We denote by $i_{0,1}$ the two maps $\ground\to I_n$ corresponding to the two vertices; equivalently these are the linear duals of the evaluation maps.
\begin{defi}\label{def:homotopycoalg}
	Let $f,g:A\to B$ be two maps in $\ccog$. \begin{enumerate}
		\item 
		We say that $f$ and $g$ are related by an elementary $n$-homotopy for $n=1,2,\ldots$ if there exists a curved coalgebra map
		\[
		h:A\otimes I_n\to B
		\] 
		such that $h\circ (A \otimes i_0)=f$ and $h\circ (A \otimes i_1)=g$. 
		If $f$ and $g$ are related by a zig-zag of elementary $n$-homotopies, we will call them $n$-homotopic and write 
		$f{\sim}_n g$.
		\item We say that $f$ and $g$ are related by an elementary $\infty$-homotopy if 
		there exists a curved coalgebra map
		\[
		h:A\otimes I_\infty\to B
		\] 
		such that $h\circ (A \otimes i_0)=f$ and $h\circ (A \otimes i_1)=g$. If $f$ and $g$ are related by a zig-zag of elementary $\infty$-homotopies, we will call them $\infty$-homotopic and write 
		$f{\sim}_\infty g$.
	\end{enumerate}
\end{defi}
As before, there is an analogous definition of homotopy for dg coalgebras.
\begin{rem}
	The above definition can be equivalently formulated in the language of pseudocompact curved algebras. In this language, it is essentially a pseudocompact version of Definition \ref{def:homotopyalg}.
\end{rem}
As usual, the notion of homotopy between maps gives
rise to the notion of homotopy equivalence.
\begin{defi}
	Two curved coalgebras $A$ and $B$ are $n$-homotopy equivalent if there are maps $f:A\to B$ and $g:B\to A$ such that $f\circ g\sim_n\id_B$ and $g\circ f\sim_n\id_A$ where $n=1,\ldots,\infty$. A dg coalgebra $n$-homotopy equivalent to $\ground$ is called $n$-contractible.
\end{defi}
\begin{prop}\label{prop:homotopycoalg} Let $1\leq m<n\leq\infty$.
	\begin{enumerate}
		\item The relation of elementary $n$-homotopy in $\ccog$ is reflexive but not symmetric and not transitive for $n<\infty$. The notion of elementary $\infty$-homotopy is reflexive and symmetric but not transitive.
		\item Let $f,g:A\to B$ be two maps between curved coalgebras $A$ and $B$ that are elementary $n$-homotopic. Then for any map $r:B\to C$ the composites $r\circ f$ and $r\circ g$ are elementary $n$-homotopic. Similarly for any for any map $k:D\to A$ the composites $f\circ k$ and $g\circ k$ are elementary $n$-homotopic.
		\item	If two maps between curved coalgebras are $n$-homotopic then they are $m$-homotopic. Similarly, if two curved coalgebras are $n$-homotopy equivalent, then they are $m$-homotopy equivalent.
		\item If two maps of dg coalgebras are $n$-homotopic then they are chain homotopic and so induce the same map on homology. If two dg coalgebras are $n$-homotopy equivalent then they are quasi-isomorphic.
		\item Two $m$-homotopic maps are not necessarily $n$-homotopic. 
		\item The dg coalgebra $I_1$ is $1$-contractible. The dg coalgebra $I_\infty$ is $\infty$-contractible.
			\end{enumerate}
\end{prop}
\begin{proof}	
All the claims above admit obvious reformulations in terms of pseudocompact curved or dg algebras instead of coalgebras, and in this reformulation the proof is the same as the proof of Proposition \ref{prop:homotopyalg}.
\end{proof}	
\begin{rem}
	The reason that an analogue of item (7) of Proposition \ref{prop:homotopyalg} was omitted from Proposition \ref{prop:homotopycoalg} is that we do not yet have a model structure on the category $\ccog$. There is, on the other hand, a model structure on the category of \emph{conilpotent} curved coalgebras \cite[Chapter 9]{Positselski11} and an appropriate analogue of Proposition \ref{prop:homotopyalg}(7) holds for conilpotent curved coalgebras. We omit the details.
\end{rem}
\begin{rem}
	The dg coalgebra $I_n$ is not $n$-contractible for $1<n<\infty$.
\end{rem}
Since $n$-homotopies of dg coalgebras are compatible with compositions, the following definition makes sense.
\begin{defi}
	Let $1\leq n\leq\infty$. The $n$-homotopy category of dg coalgebras $\Ho_n\Coalg$ is the category whose objects are dg coalgebras and morphisms are $n$-homotopy classes of maps.
\end{defi}

Given a simplicial set $X$, we can form its normalised simplicial chain coalgebra $C_*(X)$. This construction has strong multiplicative properties, and in particular for two simplicial sets $X$ and $Y$ there is a natural coalgebra map -- the Eilenberg--Zilber map -- $C_*(X)\otimes C_*(Y)\to C_*(X\times Y)$. The following result holds.
\begin{prop} Let $n=1,2,\ldots,\infty$ and let $X$ and $Y$ be two simplicial sets.
\begin{enumerate}\item	If $f,g:X\to Y$ are $n$-homotopic, then the induced maps of dg coalgebras $f_*, g_*:C_*(X)\to C_*(Y)$ are $n$-homotopic.
	\item If $X$ and $Y$ are $n$-homotopy equivalent, then the dg coalgebras $C_*(X)\to C_*(Y)$ are $n$-homotopy equivalent.
	\item If $X$ and $Y$ are weakly equivalent Kan complexes then $C_*(X)\to C_*(Y)$ are $n$-homotopy equivalent for all $n$.
	\end{enumerate}
\end{prop}
\begin{proof}
	For (1), given a homotopy $h: X\times \operatorname{D}^n\to Y$ between $f$ and $g$, take chains and apply the Eilenberg--Zilber map to get a homotopy $h': C_*(X)\times \operatorname{D}^n\to C_*(Y)$ between $f_*$ and $g_*$. Statement (2) is clear from the functoriality of statement (1). For (3), just observe that two Kan complexes are weakly equivalent if and only if they are $n$-homotopy equivalent for all $n$, and then apply (1).
	\end{proof}

	\section{Categories of twisted (co)modules and 3-homotopies}\label{section:twisted}
In this section we consider coderived categories and their corresponding derived categories of the second kind, in the sense of Positselski \cite{Positselski11}. We also consider their associated dg categories, and show how they are naturally related to 3-homotopies of (co)algebras introduced above.

\subsection{Twisted modules} Given a curved algebra $A$, we will consider a certain triangulated category $\Dcoderivedc(A)$, the compactly generated coderived category of $A$. A quick construction of $\Dcoderivedc(A)$ proceeds as follows. Recall that if $A$ is a curved algebra then $A^{\#}$ denotes its underlying graded algebra. A twisted $A$-module is a dg $A$-module whose underlying graded $A^{\#}$-module is of the form $A^\#\otimes V$, where $V$ is a graded vector space. For example, sums of shifts of free $A$-modules are twisted modules, but in general there are more.
A differential on such a twisted module is the same thing as a Maurer--Cartan element of the curved algebra $A\otimes \End(V)$. If $V$ is finite dimensional, such a twisted $A$-module will be called \emph{finitely generated} (f.g.). The dg categories of twisted $A$-modules and of f.g. twisted $A$-modules will be denoted by $\Tw(A)$ and $\Twfg(A)$, and their homotopy categories by $\Ho(\Tw(A))$ and $\Ho(\Twfg(A))$ respectively. Then $\Dcoderivedc(A)$ is defined to be the smallest triangulated subcategory of $\Ho(\Tw(A))$ containing 
$\Ho(\Twfg(A))$ and closed under arbitrary direct sums. Its subcategory of compact objects will be called the \emph{perfect} compactly generated coderived category of $A$ and denoted by 
$\Perfcoderivedc(A)$; it is clear that $\Perfcoderivedc(A)$ is the idempotent completion of $\Ho(\Twfg(A))$. Note that $\Perfcoderivedc(A)$ is the homotopy category of the dg category $\Perfcoderiveddgc(A)$ consisting of all perfect twisted $A$-modules. When $A$ is a dg algebra, then this dg category $\Perfcoderiveddgc(A)$ is the Morita fibrant replacement of the dg category $A$ \cite{Tabuada05}.

In fact, \cite{GL20} constructs a model structure on the category $A$-$\Mod$ whose homotopy category is $\Dcoderivedc(A)$. The weak equivalences in this model structure are the maps $M\to N$ which induce quasi-isomorphisms of dg vector spaces $\underline{\Hom}_A(L,M)\to\underline{\Hom}_A(L,N)$ for all f.g. twisted $A$-modules $L$. The fibrations are the surjections.

When $A$ is a dg algebra, the category $\Dcoderivedc(A)$ is analogous to the ordinary derived category $\operatorname{D}(A)$ of $A$, but is in general a finer invariant. Isomorphisms in $\Dcoderivedc(A)$ are quasi-isomorphisms, but the converse is not true. Hence one can think of $\operatorname{D}(A)$ as the localisation of $\Dcoderivedc(A)$ at the quasi-isomorphisms. For a cofibrant dg algebra $A$ the categories  $\Dcoderivedc(A)$ and $\operatorname{D}(A)$ (as well as various other versions of the derived category of $A$) all coincide, cf. \cite[Section 9.4]{Positselski11}, \cite[Section 3.3]{GL20}.

A map of curved algebras $f:A\to B$ determines a dg functor $f_*:A$-$\Mod\to B$-$\Mod$ given by $M\mapsto M\otimes_A B$; the dg $B$-module $M\otimes_A B$ is called the \emph{induced} module. The functor $f_*$ is called \emph{induction}; it clearly restricts to a dg functor $\Twfg(A)\to\Twfg(B)$ and hence to a functor between the corresponding homotopy categories. This commutes with direct sums and hence induces functors $\Dcoderivedc(A)\to\Dcoderivedc(B)$ and $\Perfcoderivedc(A)\to \Perfcoderivedc(B)$; we will denote all of these functors by $f_*$. The following result holds.

\begin{prop}\label{prop:3homotopy}
	Let $A,B$ be curved algebras and let $f,g:A\to B$ be two maps that are 3-homotopic. Then the induced functors 
	$f_*,g_*:\Dcoderivedc(A)\to\Dcoderivedc(B)$ (and therefore, also their restrictions $\Perfcoderivedc(A)\to \Perfcoderivedc(B)$) are isomorphic. 
\end{prop}
\begin{proof}
	Let $(M,d_M)$ be a twisted $A$-module, so that there is an isomorphism of $A^{\#}$-modules $M^{\#}\cong A^\#\otimes V$ for some graded vector space $V$. Let $H:A\to B\otimes I^3$ be a $3$-homotopy between $f$ and $g$; it determines a $3$-homotopy $H\otimes\id: A\otimes\End(V)\to B\otimes\End(V)\otimes I^3$ between $f\otimes \id$ and $g \otimes \id$. The differential $d_M$ in $M$ is a MC element in the curved algebra $A\otimes\End(V)$. Since the set $\mathcal{B}\coloneqq\{e,f,s,t,st,ts,sts\}$ is a basis for $I^3$, we can write $(H\otimes\id)(d_M)\in\MC(B\otimes\End(V)\otimes I^3)$ in components as 
$(H\otimes\id)(d_M)=\sum_{b\in\mathcal{B}}(H\otimes\id)(d_M)_bb$.
	Then $[(H\otimes\id)_*(d_M)]_e$ and $[(H\otimes\id)_*(d_M)]_f$ are the MC elements in $B\otimes\End(V)$ that correspond to $f_*(d_M)$ and $g_*(d_M)$ respectively; $[(H\otimes\id)_*(d_M)]_s$ and $[(H\otimes\id)_*(d_M)]_{t}$ correspond to $B$-module maps $f_*(M)\to g_*(M)$ and $g_*(M)\to f_*(M)$ respectively, and these maps are inverse up to homotopies given by the elements $[(H\otimes\id)_*(d_M)]_{st}$ and $[(H\otimes\id)_*(d_M)]_{ts}$, interpreted as endomorphisms of $f_*(M)$ and $g_*(M)$. 
	
	In other words,  $[(H\otimes\id)_*(d_M)]_s$ gives a natural transformation $f_*\to g_*$ and $[(H\otimes\id)_*(d_M)]_t$ gives its inverse natural transformation $g_*\to f_*$, as required.
\end{proof}
\begin{cor}\label{cor:3homotopy}
Let $A$ and $B$ be curved algebras that are $3$-homotopy equivalent. Then their coderived categories $\Dcoderivedc(A)$ and $\Dcoderivedc(B)$ (as well as $\Perfcoderivedc(A)$ and $\Perfcoderivedc(B)$) are equivalent.
\end{cor}

It is clear that $\Dcoderivedc(\ground)\simeq\operatorname{D}(\ground)\simeq\operatorname{grVect}_{\ground}$, the category of graded $\ground$-vector spaces.
\begin{cor}\label{cor:homotopytrivial}
	The unit map $\ground\to I^\infty$ induces an equivalence of categories $\operatorname{grVect}_{\ground}\simeq \Dcoderivedc(I^\infty)$.
\end{cor}
\begin{proof}
	By Proposition \ref{prop:homotopyalg} (5) the dg-algebra $I^\infty$ is $\infty$-contractible, and thus a fortiori 3-contractible. The conclusion follows from Corollary \ref{cor:3homotopy}.
	\end{proof}
\begin{rem}
	We will shortly see that an analogue of Corollary \ref{cor:homotopytrivial} holds for the dg algebra $I^n$ for $n\geq 3$, despite the fact that it is \emph{not} $n$-contractible. 
\end{rem}
\subsection{Restriction and induction as a Quillen adjunction} Recall that given a map $f:A\to B$ of curved algebras, the induction functor $f_*:A$-$\Mod\to B$-$\Mod$ has a right adjoint, the restriction functor $f^*:B$-$\Mod\to A$-$\Mod$; for a dg $B$-module $N$ the dg $A$-module $f^*(N)$ has the same underlying graded vector space as $N$ and the action of $A$ defined through the map $f:A\to B$.

To construct the derived version of the functor $f^*$ requires some more work, since the restriction of a twisted $B$-module is not necessarily a twisted $A$-module. One way to do it is to use the compactly generated model structure on $A$-$\Mod$ (and $B$-$\Mod$) mentioned above. In this model structure, the cofibrant modules are the retracts of the twisted modules that are unions of their f.g. twisted submodules, and we can define the value of the derived functor of $f^*$ at a $B$-module $N$ to be $f^*$ applied to a cofibrant replacement of $N$. In order to do this, we need the following result.
\begin{prop}\label{prop:adjunctmod}
	The pair $(f_*,f^*)$ is a Quillen adjunction between the categories $A$-$\Mod$ and $B$-$\Mod$.
\end{prop}
\begin{proof}
We only need to prove that the functor $f^*:B$-$\Mod\to A$-$\Mod$ preserves fibrations and acyclic fibrations. Fibrations are surjections, and it is clear that $f^*$ preserves these. It is enough to check that $f^*$ preserves all weak equivalences.
It is not hard to see that $M \to N$ is a weak equivalence if and only if its mapping cone is weakly trivial (i.e.\ weakly equivalent to the zero module), so it suffices to show that $f^*$ preserves weakly trivial modules. Let $N$ be a weakly trivial $B$-module. We need to show that if $L$ is any f.g. twisted $A$-module, the dg vector space $\underline{\Hom}_A(L,f^*N)$ is acyclic. But clearly
\[
\underline{\Hom}_A(L,f^*N)\cong \underline{\Hom}_B(f_*L,N)
\]
and $\underline{\Hom}_B(f_*L,N)$ is acyclic since $f_*L$ is a f.g. twisted $B$-module, and $N$ was assumed to be weakly trivial.
\end{proof}

\subsection{Twisted comodules} Given a curved coalgebra $C$, we will consider its coderived category $\Dcoderived(C)$, cf. \cite{Positselski11}. It is similar to the compactly generated coderived category of a curved algebra described above, and it admits a similar construction in terms of twisted comodules. A twisted comodule is a dg $C$-comodule whose underlying graded 
$C^{\#}$-comodule is $C^\#\otimes V$, for $V$ some graded vector space. 
Alternatively, by dualising, a twisted comodule is a twisted pseudocompact module over the pseudocompact dg algebra $C^*$, i.e. a pseudocompact module whose underlying $(C^*)^{\#}$-module is $(C^*)^\#\otimes W$, for $W\cong V^*$ some pseudocompact vector space. A differential on such a (co)module is the same thing as an MC element of the dg algebra $C^*\otimes \End(V)$; to see this, think of $C^*\otimes \End(V)$ as a pro-algebra and observe that the $\MC$ functor commutes with limits (and in particular cofiltered limits). The dg category of twisted $C$-comodules will be denoted by $\Tw(C)$, and its homotopy category $\Ho(\Tw)(C)$ is $\Dcoderived(C)$, the coderived category of the curved coalgebra $C$. It is compactly generated by the triangulated subcategory of finite dimensional dg $C$-comodules, cf. \cite[Section 5.5]{Positselski11}. We will denote the idempotent completion of this latter category by 
$\Perf(C)$.

In fact, $\Dcoderived(C)$ is the homotopy category of a certain model category structure on the category $C$-$\Comod$ of dg $C$-comodules, cf. \cite[Section 8.2]{Positselski11}. The weak equivalences of dg $C$-comodules are those maps $M\to N$ which induce quasi-isomorphisms  $\underline{\Hom}_C(N,L)\to\underline{\Hom}_C(M,L)$ for all twisted comodules $L$; this follows directly from op.cit. since the fibrant $C$-comodules are precisely the retracts of the twisted $C$-comodules. Cofibrations are injective maps.

The coderived category of a curved coalgebra $C$ is equivalent to the coderived category $\Dcoderived(\Omega C)$ of $\Omega C$ which, in this case, coincides with $\Dcoderivedc(\Omega C)$, cf. \cite[Section 6.7]{Positselski11}, \cite[Section 3.3]{GL20}. Conversely, for a curved algebra $A$, there is an equivalence between $\Dcoderivedc(A)$ and $\Dcoderived(\check BA)$, the coderived category of $\check BA$, cf. \cite{GL20}.

\begin{rem} The topological significance of the coderived category of a dg coalgebra is underscored by the fact that for a simplicial set $X$, the category $\Dcoderived(C(X)_*)$ is equivalent by \cite{GL20} to the category of linear representations of (any $\infty$-category categorically equivalent to) $X$, cf.\cite[Theorem 5.2]{HL2020}.
\end{rem}
Given a map $f:A\to B$ of curved coalgebras, the cotensor product functor $f_*:M\mapsto M\boxempty_BA$ restricts to dg functors $\Tw(B)\to\Tw(A)$ and $\Twfg(B)\to\Twfg(A)$. It moreover descends to functors between the corresponding homotopy categories, thus giving triangle functors $\Dcoderivedc(B)\to\Dcoderivedc(A)$ and $\Perf(B)\to \Perf(A)$. We will denote all of these functors by $f_*$ and refer to them as the {coinduction} functors. The following result holds.
\begin{prop}\label{prop:3cohomotopy}
	Let $A,B$ be curved coalgebras and $f,g:A\to B$ be maps that are 3-homotopic. Then the induced functors 
	$f_*,g_*:\Dcoderivedc(B)\to\Dcoderivedc(A)$ (and hence also their restrictions $\Perfcoderivedc(B)\to \Perfcoderivedc(A)$) are isomorphic. 
\end{prop}
\begin{proof}
	Dualising, we obtain a pair of 3-homotopic maps of pseudocompact curved algebras $B^*\to A^*$. Given a $B$-comodule $M$, its linear dual $M^*$ is a pseudocompact $B^*$-module, and its induced pseudocompact $A$-module $M^*\otimes_{B^*}A^*$ is the dual of the coinduced comodule $M\boxempty_B A$.  After this translation of the statement to the language of pseudocompact modules, the proof of Proposition \ref{prop:3homotopy} carries over to yield the desired result. 
\end{proof}	
\begin{cor}\label{cor:3cohomotopy}
	Let $A$ and $B$ be curved coalgebras that are $3$-homotopy equivalent. Then their coderived categories $\Dcoderivedc(A)$ and $\Dcoderivedc(B)$ (as well as $\Perf(A)$ and $\Perf(B)$ are equivalent.
\end{cor}
\begin{proof}
	This is immediate from Proposition \ref{prop:3cohomotopy}.
\end{proof}	
\begin{prop}\label{prop:Sinfty}
	The counit map $I_\infty\to\ground$ induces a equivalence $\operatorname{grVect}_\ground\to \Dcoderivedc(I_\infty)$.
\end{prop}
\begin{proof}
	By Proposition \ref{prop:homotopycoalg}(5), the dg-coalgebra $I_\infty$ is $\infty$-contractible, and thus a fortiori 3-contractible. The desired conclusion follows from Corollary \ref{cor:3cohomotopy}.
\end{proof}
Even though for $n\geq 3$ the dg coalgebra $I_n$ is not 3-contractible, it turns out that its compactly generated coderived category behaves as if it were, and the corresponding statement for dg algebras also holds.
\begin{prop}\label{prop:coderivedinfinity}Let $n\geq 3$ and $m=1,2$. Then:
\begin{enumerate}\item	The counit map $I_n\to \ground$ induces an equivalence $\operatorname{grVect}_\ground\to \Dcoderivedc(I_n)$.
	\item The unit map $\ground\to I^n$ induces an equivalence $\operatorname{grVect}_\ground\to \Dcoderivedc(I^n)$.
	\item The counit map $I_m\to \ground$ does not induce an equivalence $\operatorname{grVect}_\ground\to \Dcoderivedc(I_n)$.
\item	The unit map $\ground\to I^m$ does not induce an equivalence $\operatorname{grVect}_\ground\to \Dcoderivedc(I^n)$.
	\end{enumerate}
\end{prop}
\begin{proof} Consider first the $n=3$ case. The category $\Dcoderivedc(I_3)$ is equivalent by \cite[Theorem 3.41(2)]{HL2020} to the derived
	category of the dg category $\OCat(I_3)$, where $I_3$ is viewed as a split curved coalgebra over its coradical $ke\oplus kf$. Note that although $I_3$ is not curved as a plain dg coalgebra, it has nontrivial curvature when viewed as a split curved coalgebra. We compute $\OCat(I_3)$. The coaugmentation coideal of $I_3$ is five-dimensional over $k$, spanned by elements $s', t', l,r,w$; in terms of the description of $I^3$ given in \ref{ddescription} these elements are the linear duals of the elements $s,t,st,ts,sts$ respectively. One can check that $s',t'$ are cycles, that $dl=s'-t'=-dr$, and $dw=l-r$. The reduced comultiplication on the coaugmentation coideal is given by $\Delta(s')=\Delta(t')=0$, $\Delta(l)=s'\otimes t'$, $\Delta(r)=t'\otimes s'$, and $\Delta(w)=s'\otimes r + l \otimes s'$. The curvature functional is given by $h=-\alpha^\vee - \beta^\vee$; in terms of the dual algebra $I^3$ we simply have $h=-st-ts$. Using the description of the cobar construction given in 
	\cite[Definition 3.17]{HL2020} we see that $\OCat(I_3)$ is the dg category with two objects $e,f$, and with morphisms freely generated by two arrows $s':e\to f, t':f\to e$ of degree zero, two arrows $r:e\to e$ and $l:f\to f$ of homological degree one, and one arrow $w:e\to f$ of homological degree two. The differential is zero on $s'$ and $t'$, and we have $d(r)=t's'-1$, $d(l)=s't'-1$ and finally $d(w)=sr-ls$. The dg category $\OCat(I_3)$ is a cofibrant resolution of the linear category $\C$ with two objects and an isomorphism between them, cf.\ \cite[Example 3.7]{Drinfeld04}. It follows that $\operatorname{D}(\OCat(I_3))$ is equivalent to the category of graded vector spaces over $\ground$, so Statement (1) is proved in this case. 

Next, consider the category $\Dcoderivedc (I_\infty)$ (even though Proposition \ref{prop:Sinfty} already tells us what it is). As above, we conclude that it is equivalent to $\operatorname{D}(\OCat(I_\infty))$. Note that the dg category $\OCat(I_\infty)$ is the canonical resolution of $\C$ given by the bar-cobar construction (this, incidentally,  gives an alternative proof of Proposition \ref{prop:Sinfty}). It follows that the functor $F_3:\OCat(I_3)\into \OCat(I_\infty)$ induced by the inclusion of simplicial sets ${\operatorname D}^3\into{\operatorname S}^\infty$ is a quasi-equivalence, because it fits into the commutative triangle $$\begin{tikzcd}\OCat(I_3)\ar[rr,"F_3"]\ar[rd]&& \OCat(I_\infty)\ar[ld]\\
	&\C&
	\end{tikzcd}$$
whose diagonal legs are quasi-equivalences. Hence $F_3$ is an acyclic cofibration, and so lifts against the fibration $\OCat(I_3) \to *$, providing a retraction $\OCat(I_\infty)\to \OCat(I_3)$ which is necessarily a homotopy inverse to $F_3$. Since $F_3$ factors through the inclusion $F_n:\OCat(I_n)\to \OCat(I_\infty)$, it follows that $\OCat(I_3)$ is a retract of $\OCat(I_n)$ and that this retract is necessarily a quasi-equivalence. Statement (1) for arbitrary $n\geq 3$ now follows.

	For (2), consider the dg category of dg $I^n$-modules which are homotopy retracts of those dg $I^n$-modules whose underlying $(I^n)^\#$-module has the form $I^n\otimes V$, where $V$ is a finite dimensional graded vector space. The homotopy category of this dg category is equivalent to both $\Perfcoderivedc(I^n)$ and $\Perfcoderivedc(I_n)$, and by part (1) is equivalent to the category of finite dimensional graded $\ground$-vector spaces. Since $\Dcoderivedc(I^n)$ is generated as a triangulated category by $\Perfcoderivedc(I^n)$ under taking all direct sums, it follows that it is equivalent to $D(\ground)$; clearly this equivalence is induced by the unit map of $\ground\to I^n$.
	
Statements (3) and (4) follow from similar considerations.  As before, the derived category of $\OCat(I_m)$ is equivalent to $\Dcoderivedc(I_m)$. The category $\OCat(I_1)$ has two objects and a single degree zero morphism between them; clearly its derived category is not triangle equivalent to the category of graded vector spaces. The category $\OCat(I_2)$ has two objects $e,f$ and the morphisms are freely generated by two closed degree zero morphisms $s':e\to f, t':f\to e$ and one homological degree 1 arrow $r:e\to e$ with $d(r)=t's'-1$. The derived category of $\OCat(I_2)$ is equivalent to the category of graded vector spaces with an idempotent morphism. It has two indecomposable objects and thus is not equivalent to the category of graded vector spaces, from which Statement (3) is clear. The same proof shows that $\Perfcoderivedc(I_m)\simeq \Perfcoderivedc(I^m)$ is not equivalent to the category of finite dimensional graded vector spaces, from which (4) follows.
\end{proof}
\begin{rem}
	The proof of Proposition \ref{prop:coderivedinfinity} shows that we have an infinite tower $$\OCat(I_3) \into \OCat(I_4)\into\cdots\into \OCat(I_n)\into\cdots \into\OCat(I_\infty)$$of cofibrant resolutions of the dg category $\C$. The smallest of them is $\OCat(I_3)$, constructed by Drinfeld \cite{Drinfeld04}; the largest is $\OCat(I_\infty)$, the canonical bar-cobar resolution of $\C$. Because the simplicial set ${\operatorname S}^\infty$ is the colimit of the simplicial sets  ${\operatorname D}^n$, it moreover follows that the colimit of the above tower is again $\OCat(I_\infty)$.
\end{rem}
\begin{rem} Let $A$ be a finite dimensional curved algebra; then its linear dual $A^*$ is a curved coalgebra. Given a f.g. twisted $A$-module $M=(A\otimes V,d)$, with $V$ finite dimensional, then $M^*$ is a twisted perfect $A^*$-comodule, with $M^{**}\cong M$. However, it does not follow that the categories $\Dcoderivedc(A)$ and $\Dcoderivedc(A^*)$ are equivalent. Indeed, for an infinite dimensional vector space $U$ there is no obvious counterpart to the twisted comodule of the form ($A^*\otimes U, d)$. Note also that for an infinite dimensional $V$, the twisted module $(A\otimes V,d)$ is not, in general, a cofibrant object in the compactly generated model category of $A$-modules. However this subtlety does not arise when $\Perfcoderivedc(A)$ is semisimple, i.e. is equivalent to the category of graded vector spaces, as is the case of $I^n$ for $n\geq 3$. 
\end{rem}

\subsection{Corestriction and coinduction as a Quillen adjunction} Recall that given a map $C\to D$ of curved coalgebras there is the coinduction
 functor $f_*:D$-$\Comod\to C$-$\Comod$. It is right adjoint to the corestriction functor 
 $f^*:C$-$\Comod\to D$-$\Comod$; 
 for a dg $C$-comodule $N$ the dg $D$-comodule $f^*(N)$ has the same underlying graded vector space as $N$ and the coaction of $N\to N\otimes D$ is obtained using the coaction $N\to N\otimes C$ and the map $C\to D$.

To construct the derived version of the functor $f^*$ requires some more work, since the corestriction of a twisted $C$-comodule is not necessarily a twisted $D$-comodule. One way to do it is to use the compactly generated model structure on $C$-$\Comod$ (and $D$-$\Comod$) constructed in \cite[Section 8.2]{Positselski11}. We can define the value of the derived functor of $f^*$ on a $D$-module $N$ as $f^*$ applied to a fibrant replacement of $N$. More precisely, we have the following result. 
\begin{prop}
	The pair $(f_*,f^*)$ is a Quillen adjunction between the categories $C$-$\Comod$ and $D$-$\Comod$.
\end{prop}
\begin{proof}
	The proof of Proposition \ref{prop:adjunctmod} carries over with obvious modifications.
\end{proof}
\begin{rem}
	Another approach to constructing the derived functors of $f_*$ and $f^*$ was given by Positselski in \cite[Section 4.8]{Positselski11}.
\end{rem}

\section{Maurer--Cartan dg categories}\label{section:MC}

In this section, we introduce the Maurer--Cartan dg category associated to a curved algebra $A$. This is a dg category $\MCdg(A)$ whose objects are the MC elements of $A$ and whose hom-complexes are given by the corresponding two-sided twists of $A$. We show that the $\MCdg$ functor is a right adjoint. We then discuss various notions of homotopy of MC elements, before turning our attention to dg categories of the form $\MCdg\Hom(C,A)$, which can be viewed as dg categories of curved morphisms.
\subsection{Maurer--Cartan dg categories}

Recall from \cite{CHL21} that an MC element $x$ in a dg algebra $A$ determines a right dg $A$-module $^{[x]}A$ whose underlying $A^\#$-module is $A^\#$, with differential given by the formula $^{[x]}d(a)=d_A(a)+xa$. The dg category of dg $A$-modules of the form $^{[x]}A$ will be denoted by $\MCdg(A)$; the assignment $A\mapsto \MCdg(A)$ is functorial. Equivalently, one can think of the objects of $\MCdg(A)$ as the MC elements of $A$, with the space of morphisms $x \to y$ being exactly the space of right $A$-module morphisms between the twists $^{[x]}A$ and $^ {[y]}A$. This hom-space is identified with the two-sided twist $^{[y]}A^{[x]}$ whose differential is $a \mapsto d_A(a)+ya - \tilde{a}x$, where we write $\tilde{a}\coloneqq (-1)^{\mathrm{deg}(a)}a$.

In exactly the same manner, one can extend the definition of $\MCdg$ to the category of curved algebras: note that this remains a dg category. Note that $\MCdg(A)$ may be the empty dg category; indeed if $\MCdg(A)$ has an object, then $A$ must be isomorphic to a dg algebra, as in the proof of \ref{algslice1}.

One can even extend $\MCdg$ to the category $\calgp$ by declaring that $\MCdg(\varnothing)$ is the empty dg category, which is the initial object in the category of dg categories. In fact, in this form, $\MCdg$ is a right adjoint, although we will need to modify the target category slightly, in exactly the same manner as \cite{HL2020}.

Let $\dgcat'$ denote the category of (small) dg categories with nonzero identity morphisms, as well as the zero dg category. The category $\dgcat'$ is complete and cocomplete. By \cite[3.31]{HL2020}, the category $\dgcat'$ admits a model structure where a morphism is a (co)fibration or a weak equivalence precisely when its image in $\dgcat$ is. This model structure is right proper and cofibrantly generated by the usual generating cofibrations in $\dgcat$. The inclusion functor $\dgcat' \into \dgcat$ is a left Quillen equivalence.

Observe that if $A$ is a nonzero curved algebra, then $\MCdg(A)$ has nonzero identity morphisms. If $A$ is the zero curved algebra then $\MCdg(A)$ is the zero dg category. Moreover, the empty dg category has nonzero identities (since it has no identities at all). Hence we may regard $\MCdg$ as a functor from curved algebras to $\dgcat'$.

\begin{prop}\label{mcdgadj}
	The functor $\MCdg:\calgp \to \dgcat'$ admits a left adjoint.
\end{prop}
Before we prove this we will need some recollections on the uncurving functor. If $A$ is a curved algebra, we define its uncurving $HA$ to be the following dg algebra. The underlying graded algebra of $HA$ is $A\langle\eta\rangle$, where $\eta$ has cohomological degree one. The differential $\partial$ on $HA$ is defined by $\partial a = da - [\eta,a]$ and $\partial \eta = h - \eta^2$; essentially we are freely adding an MC element $-\eta$ to $A$ and then twisting by it to obtain a dg algebra. The construction $A\mapsto HA$ is functorial, and $H$ is left adjoint to the inclusion of dg algebras into curved algebras \cite[3.6]{HL2020}.

\begin{proof}[Proof of \ref{mcdgadj}.]
	Given an object $\mathcal{D}\in \dgcat'$, we will define a dg algebra $\algmc(\mathcal{D})$, the {MC algebra} of $\mathcal{D}$, by generators and relations. We will first define an auxiliary dg algebra $\algmc'(\mathcal{D})$ and then modify our construction slightly. Let $\mathcal{D}$ be a small nonempty nonzero dg category with nonzero identities. The generators of $\algmc'(\mathcal{D})$ are
	\begin{itemize}
		\item For every object $x$ of $\mathcal{D}$, a generator $\bar x$ of cohomological degree $1$.
		\item For every morphism $g:x \to y$ in $\mathcal{D}$ of homogenous degree $n$, a generator $\bar g$ of degree $n$. We extend this to nonhomogenous morphisms by linearity.
	\end{itemize}
	and the relations are
	\begin{itemize}
		\item $d(\bar x) = {\bar x} ^2$ (i.e.\ $x$ is an MC element).
		\item If $g: x \to y$ then $\overline{dg} = d\bar g+ \bar y\bar g - \tilde{\bar g} \bar x$.
		\item If $g_1$ and $g_2$ are parallel morphisms and $a,b\in\ground$ then $\overline{ag_1+bg_2}=a\bar {g_1} + b \bar {g_2}$.
		\item $\overline{g_1\circ g_2} = \bar {g_1} \bar {g_2}$ whenever $g_1$ and $g_2$ are composable.
		\item $\overline{\id_x}=1$ for all objects $x$ of $\mathcal{D}$. Note that here is where we need $\mathcal{D}$ to have nonzero identities in order for our construction to make sense.
	\end{itemize}
	Note that we have left $\algmc'({\emptyset})$ and $\algmc'({0})$ undefined. One can check that if $A$ is a dg algebra then we have a natural isomorphism $$\Hom_{\alg}(\algmc'(\mathcal{D}),A)\cong \Hom_{\dgcat'}(\mathcal{D},\MCdg(A))$$with a dg functor $F$ corresponding to the algebra morphism defined on generators by $\bar u \mapsto \bar{F}(u)$.

	However, since we are interested in the category $\calg$ and not $\alg$, the construction $\algmc'(\mathcal{D})$ has one MC element too many. We will need to modify our construction slightly by removing a single generator; the choice will not matter. This is similar to how, in the extended bar construction, one must choose a `fake augmentation' \cite{GL20}. If $\mathcal{D}$ is nonempty and nonzero, choose an object $x$. The {reduced MC algebra} $\mathrm{Alg}^{x}_\mathrm{MC}(\mathcal{D})$ is defined to be the quotient of $\algmc'(\mathcal{D})$ by the single extra relation $\bar x =0$. Equivalently, the definition is the same as that of $\algmc'(\mathcal{D})$, but we do not add the generator corresponding to $x$. 
	
	If $x,y$ are two different objects of $\mathcal{D}$, consider the dg algebra $B_{xy}$ obtained by twisting $\mathrm{Alg}^{x}_\mathrm{MC}(\mathcal{D})$ by $\bar y$. Observe that for every object $z$ of $\mathcal D$, the element $\bar z + \bar y$ is an MC element of $B_{xy}$, and moreover this is compatible with the differential on elements of the form $\bar g$. Hence the function $\phi_{xy}: \mathrm{Alg}^{y}_\mathrm{MC}(\mathcal{D}) \to B_{xy} $ defined on generators by declaring that $\phi_{xy}(\bar z) = \bar z + \bar y$ and $\phi(\bar g) = \bar g$ is a well-defined morphism of dg algebras, and it is easy to check that it is an isomorphism. Since $B_{xy}$ is isomorphic as a curved dg algebra to $\mathrm{Alg}^{x}_\mathrm{MC}(\mathcal{D})$, it follows that $\mathrm{Alg}^{x}_\mathrm{MC}(\mathcal{D})$ and $\mathrm{Alg}^{y}_\mathrm{MC}(\mathcal{D})$ are isomorphic as curved dg algebras; we denote this curved dg algebra by $\algmc(\mathcal{D})$. 
	
	If $\mathcal{D}$ is the empty dg category, we set $\algmc(\mathcal{D})\coloneqq \varnothing$. If $\mathcal{D}$ is the zero dg category, we set $\algmc(\mathcal{D})\coloneqq 0$. As in the proof of \cite[4.1]{GL20}, the construction $\algmc$ defines a functor from $\dgcat'$ to $\calgp$. We wish to check that there are natural isomorphisms $$\Hom_{\calg}(\algmc(\mathcal{D}),A) \cong \Hom_{\dgcat'}(\mathcal{D},\MCdg(A))$$for all $\mathcal{D}$ and all $A$. This is easy to check when $\mathcal{D}$ is empty or zero, or when $A$ is $\varnothing$, so we may assume that $\mathcal{D}\neq 0$ has an object and that $A\in \calg$. If $A$ has no MC elements then both sides are empty, so we may assume that $A$ has an MC element, and so twisting by it we may assume that $A\in\alg$. In other words, we wish to prove that we have a natural isomorphism $$\Hom_{\calg}(\algmc(\mathcal{D}),A) \cong \Hom_{\alg}(\algmc'(\mathcal{D}),A)$$which amounts to proving that we have a natural isomorphism $$H\algmc(\mathcal{D}) \cong \algmc'(\mathcal{D})$$where $H$ is the uncurving functor. We apply a similar reasoning as before: define a function $\phi:\algmc'(\mathcal{D}) \to H\algmc(\mathcal{D})$ on generators by setting $\phi(\bar z) = \bar z - \eta$ and $\phi(\bar g) = \bar g$. One can check that this extends to a morphism of dg algebras; it is clearly bijective since one can check this on the underlying graded algebras. Hence we are done.
\end{proof}

The difference between $\dgcat$ and $\dgcat'$ will only be relevant to us when considering adjunction properties of $\MCdg$. Most of the time we will simply think of $\MCdg$ as an object of $\dgcat$.

\begin{rem}
	One can prove directly from the definitions that $\MCdg$ preserves products and equalisers, and hence all limits. The key point is that the objects of $\MCdg(A)$ are the morphisms $\ground \to A$, and hence are preserved under limits
\end{rem}

\begin{rem}
	If $A$ is a dg algebra, regarded as a one-object dg category, then $\algmc(A)\cong A$.
\end{rem}

\begin{rem}
    In \ref{deligneprop} we will see a `Deligne groupoid' type description of the nerve of $\MCdg$.
\end{rem}

\subsection{Homotopy of MC elements}

We have seen that the notions of $n$-homotopy are generally all inequivalent for different $n$. It turns out that for a certain class of algebras, $3$-homotopy of maps out of them is equivalent to $\infty$-homotopy (and thus, to $n$-homotopy for $3<n<\infty$). A similar, or dual, result holds for a certain class of coalgebras. Let us first discuss $n$-homotopies for MC elements.
\begin{defi}Let $A$ be a curved algebra and $x_0,x_1\in\MC(A)$.
	\begin{enumerate}\item
		An $n$-homotopy between $x_0$ and $x_1$ is an MC element $X\in\MC(A\otimes I^n)$ such that $(1\otimes \ev_i)(X)=x_i$.
		\item An $\infty$-homotopy between $x_0$ and $x_1$ is an MC element $X\in\MC(A\hat\otimes I^\infty)$ such that $(1\otimes \ev_i)(X)=x_i$.
	\end{enumerate}
\end{defi}
\begin{rem}
	Observe that an MC element $x\in A$ is the same as a map of curved algebras $\ground \to A$ and moreover an $n$-homotopy of MC elements is a special case of of an $n$-homotopy of maps. 
\end{rem}

\begin{prop}\label{prop:MC}
	Two MC elements in a curved algebra are $3$-homotopic if and only if they are $\infty$-homotopic.
\end{prop}
\begin{proof}
	Clearly $\infty$-homotopy implies $3$-homotopy.	The reverse direction follows directly from \cite[Theorem 5.1]{CHL21}, where algebras are assumed to be without curvature, but one can check that the proof goes through in the curved setting since it uses only the MC dg category.
\end{proof}
Note that the above result can be viewed as saying that maps out of the curved algebra $\ground$ classifying MC elements are  $\infty$-homotopic if and only if they are 3-homotopic; also observe that $\ground$ is the cobar construction on the curved coalgebra $\ground$. The following two results generalise this.	

\begin{prop}\label{prebcb}
	Let $C$ be a curved coalgebra, let $A$ be a curved algebra, and let $n\leq\infty$. Then the isomorphisms
$$\Hom(\Omega C,A) \cong \MC\Hom(C,A)\cong \Hom(C,\check B A)$$respect $n$-homotopies: two morphisms $\Omega C \to A$ are $n$-homotopic if and only if the corresponding MC elements of $\Hom(C,A)$ are $n$-homotopic, if and only if the corresponding morphisms $C \to \check B A$ are $n$-homotopic.	
	\end{prop}
\begin{proof}
	We prove the statement about maps $\Omega C \to A$; the proof for maps $C \to \check B A$ is completely analogous. First, note that the convolution algebra $\Hom(C,A)$ is naturally a pro-object in curved algebras when regarded as $ C^* \hat\otimes A$ (it is not pseudocompact unless $A$ is finite dimensional). Using this we get a natural isomorphism $\Hom(\Omega C,A)\cong \MC(C^* \hat\otimes A)$. For $n<\infty$, an elementary $n$-homotopy between two maps $f,g:\Omega C \to A$ is a map $\Omega C \to A \otimes I^n$, which corresponds across the above isomorphism to an element of $\MC( C^* \hat\otimes A  \otimes I^n)$, which one can easily check is an $n$-homotopy between the two MC elements corresponding to $f$ and $g$. Conversely an $n$-homotopy of MC elements gives an $n$-homotopy between the corresponding algebra morphisms. For $n=\infty$, the proof is the same but uses the isomorphism $ \Hom(\Omega C,A\hat\otimes I^\infty)\cong \MC(C^* \hat\otimes A\hat\otimes I^\infty)$.
\end{proof}

\begin{cor}\label{cor:Barcobar}
	Let $C$ be a curved coalgebra and let $A$ be a curved coalgebra.
\begin{enumerate}\item Two maps $\Omega C\to A$ are 3-homotopic if and only if they are $\infty$-homotopic. 
		\item Two maps $C\to \check{B}A$ are 3-homotopic if and only if they are $\infty$-homotopic. 
	\end{enumerate}	
\end{cor}
\begin{proof}
Combine Propositions \ref{prop:MC} and \ref{prebcb}.
\end{proof}
\begin{rem}
	Note that if $C$ is a coaugmented dg coalgebra, the dg algebra $\Omega(C)$ need \emph{not} be cofibrant unless $C$ is conilpotent.
\end{rem}	

\begin{prop}\label{prop:MCdg}
	Let $A,A'$ be two curved algebras and $f,g:A\to A'$ be two maps that are 3-homotopic. Then the associated functors $\MCdg(f), \MCdg(g):\MCdg(A)\to\MCdg(A')$ are quasi-isomorphic, i.e.\ the two induced functors $H^0\MCdg(A)\to H^0\MCdg(A')$ are isomorphic.
\end{prop}
\begin{proof}
	The homotopy category $H^0\MCdg(A)$ is naturally a full subcategory of $H^0\Twfg(A)$ and hence of $\Perfcoderivedc(A)$. By Proposition \ref{prop:3homotopy}, the functors $\Perfcoderivedc(f)$ and $\Perfcoderivedc(g)$ are isomorphic, and hence their restrictions $H^0\MCdg(f)$ and $H^0\MCdg(g)$ are likewise isomorphic. 
\end{proof}	

The above results allow us to prove that cohomology is representable in the $3$-homotopy category of dg algebras. The key input to this is the following lemma.

\begin{lem}\label{mccohomlem}
	Let $V$ be a dg vector space and let $X\coloneqq \ground\oplus V$ be the square zero extension of $\ground$ by $V$. There is a bijection $\MCmod(X)\cong H^1(X)$.
\end{lem}
\begin{proof}
	It is clear that $\MC(X)\cong Z^1X$, so we just need to show that two MC elements $x,x'$ are $3$-homotopy equivalent if and only if they differ by a coboundary. To do this, suppose that $H\in \MC(X\otimes I^3)$ is a 3-homotopy $x \to x'$. Following the proof of \cite[Lemma 5.3]{CHL21}, write $H$ out in components as $H=x\otimes e+x'\otimes f+y\otimes s+y'\otimes t+z\otimes ts+z'\otimes st + u\otimes sts$. The MC equation for $H$ reduces to $dH=0$, and writing this out in components we obtain (in addition to $dx=0=dx'$) the equations \begin{align*}dy&=x-x'=-dy'\\
		dz&=-y-y'=dz'\\
		du&=z'-z.
	\end{align*}Clearly if these equations are satisfied then $x$ and $x'$ differ by a coboundary. Conversely if $h$ is an element such that $dh=x-x'$ then setting $y=-y'=h$ and $u=z=z'=0$ gives a $3$-homotopy between $x$ and $x'$.
\end{proof}
\begin{cor}Let $A$ be a dg algebra. if $X_n$ denotes the polynomial algebra $\ground[x]$, where $x$ has cohomological degree $2-n$, then there is an isomorphism $[X_n,A]_3\cong H^n(A)$.
\end{cor}
\begin{proof}
	First observe that there is an isomorphism $X_n\cong \Omega(C_n)$, where $C_n$ is the linear dual of the square zero extension $Y_n\coloneqq \ground[\varepsilon]/\varepsilon^2$, with $\varepsilon$ placed in cohomological degree $n-1$. The convolution algebra controlling morphisms $X_n \to A$ is the square zero extension $\ground\oplus \Hom(\bar C_n, A)\cong \ground \oplus A[n-1]$. Applying \ref{mccohomlem} hence gives us an isomorphism $$[X_n,A]_3\cong \MCmod(\bar{C}_n,A)\cong \MCmod(\ground \oplus A[n-1])\cong H^1(A[n-1])\cong H^n(A)$$as desired.
\end{proof}
\begin{rem}
	Similarly, in the $3$-homotopy category of dg coalgebras, the linear dual of cohomology is corepresentable by the coalgebra $\check B (Y_n)$: we have isomorphisms $H^n(C^*)\cong [C,\check B Y_n]_3$.
\end{rem}

\subsection{DG categories of maps} Given two curved algebras the set of maps between them can sometimes be given the structure of a dg category; this happens when the source curved algebra is the cobar construction on a curved coalgebra. A similar phenomenon happens with curved coalgebras.

\begin{defi}
	Given a curved coalgebra $C$ and a curved algebra $A$ we define a dg category $\MCdg(C,A)\coloneqq \MCdg(\Hom(C,A))$. 
\end{defi}
It is easy to see that the objects of $\MCdg(C,A)$ are the maps $\Omega(C)\to A$ of curved algebras. Two such maps are dg-isomorphic if and only if they are 3-homotopic: by \ref{prebcb}(1) they are $3$-homotopic as maps if and only if the associated MC elements are $3$-homotopic, and the claim then follows from \cite[Theorem 5.1]{CHL21}. Equivalently, one can regard the objects of $\MCdg(C,A)$ as the maps $C \to \check B A$ of curved coalgebras; again two such maps are dg-isomorphic if and only if they are 3-homotopic.

\begin{prop}\label{cor:dgcat1}
	Let $C$ and $C'$ be curved coalgebras and $A$ and $A'$ be curved algebras. \begin{enumerate}\item Let $f,g:A\to B$ be 3-homotopic maps. Then the two induced dg functors
		\[	\MCdg(C,A)\to 	\MCdg(C,A')
		\]
		are quasi-isomorphic; i.e.\ they induce isomorphic functors $$H^0\MCdg(C,A)\to 	H^0\MCdg(C,A').$$
		\item If $A$ and $A'$ are 3-homotopy equivalent, then the dg categories $\MCdg(C,A)$ and  $\MCdg(C,A')$ are dg equivalent (in particular, quasi-equivalent).
		\item Let $h,k:C\to C'$ be 3-homotopic maps. Then the two induced dg functors $$\MCdg(C,A)\to \MCdg(C',A)$$
		are quasi-isomorphic.
		\item  If $C$ and $C'$ are 3-homotopy equivalent, then the dg categories $\MCdg(C,A)$ and  $\MCdg(C',A)$ are dg equivalent (in particular, quasi-equivalent).
	\end{enumerate}
\end{prop}
\begin{proof}
	Starting with (1), we note that by Proposition \ref{prop:MCdg} it is enough to show that the functor $\Hom(C,-)$ preserves 3-homotopies. This is the case because we have a natural isomorphism $$\Hom( C,A'\otimes I^3)\cong \Hom(C,A')\otimes I^3$$and hence a 3-homotopy $A\to A'\otimes I^3$ yields a  3-homotopy $\Hom(C,A)\to\Hom(C,A')\otimes I^3$. Part (2) is an immediate consequence of (1).
The proof of (3) is similar to that of (1), since a 3-homotopy $C\otimes I_3\to C'$ yields a 3-homotopy of convolution algebras using the isomorphism $$\Hom(C\otimes I_3, A)\cong \Hom(C,A)\otimes I^3.$$Part (4) follows from part (3).
\end{proof}

\section{Lifting MC elements}\label{section:mclifting}

In this section, which is completely independent of the others, we prove that a square zero extension of curved algebras induces a fibration on $\MCdg$. We furthermore analyse the fibres before giving similar results for $\Twfg$.

\subsection{Square-zero extensions}
Recall that a morphism $\pi:A \to B$ of curved algebras is a {square zero extension} if it is surjective, and the ideal $L\coloneqq \ker(\pi)$ satisfies $L^2=0$. 
\begin{theorem}\label{squarezero1}
	Let $\pi:A \to B$ be a square zero extension of curved algebras. Then $\MCdg(\pi)$ is a fibration of dg categories.
\end{theorem}

The rest of this section will be taken up with the proof of \ref{squarezero1}. As we will see, a simple argument shows that it suffices to prove the statement for dg algebras, so we will concentrate on the dg case from now on. Recall that if $\pi:A \twoheadrightarrow B$ is a {square zero extension} of dg algebras, then the kernel $L$ of $\pi$ is a $B$-bimodule in a natural way. As graded vector spaces, we have $A\cong B\oplus L$, but this decomposition need not be multiplicative or respect the differential. For $x,y \in B$, we may write $d_A(x)=d_B(x) + \partial(x)$ and $m_A(x,y) = m_B(x,y) + \xi(x,y)$, where $\partial\in \Hom^0_\ground(B, L)$ and $\xi\in \Hom^2_\ground (B\otimes B,L)$ are $\ground$-linear maps. Note that such a pair $(\partial,\xi)$ defines a cohomological degree two element $\partial+ \xi$ in the Hochschild complex $\mathrm{Hoch}(B,L)$ which computes Hochschild cohomology. Moreover, such a pair $(\partial,\xi)$ defines a dg algebra if and only if $\partial + \xi$ is a Hochschild cochain; equivalently, if $d_B+\partial + m_B + \xi$ is a Maurer--Cartan element of $\mathrm{Hoch}(B,L)$. This in turn is equivalent to the three equations 
\begin{enumerate}
	\item $\partial d_B + d_L\partial=0$
	\item $[m_B, \xi]=0$
	\item $[m_B, \partial] + [d_B,\xi]=0$.
\end{enumerate}

We will develop some obstruction theory for lifting MC elements along square zero extensions. In order to state our lemmas, we will need some reminders on one and two-sided twistings.

Let $A$ be a dg algebra. If $a\in A$ we will use the shorthand $\tilde a\coloneqq  (-1)^{\lvert a\rvert}a$. Given $x \in \MC(A)$ we can define a {right twist} $^{[x]}A$, which is a right dg-$A$-module whose underlying graded vector space is $A$ and whose differential is given by $^{[x]}d(a)\coloneqq da + xa$. Similarly there is a {left twist} $A^{[x]}$, which is a left dg-$A$-module whose differential is $d^{[x]}(a)\coloneqq da + \tilde ax$. The {two-sided twisting} is the differential graded algebra $A^x$ whose underlying graded algebra is $A$, and whose differential is given by $d^x(a)\coloneqq da + [x,a]$. Note that $A^x \cong {}^{[x]}A \otimes_A A^{[x]}$ as dg vector spaces. It is not hard to see that the right twist $^{[x]}A$ is an $A^x\text{-}A$-bimodule, and similarly that the left twist $A^{[x]}$ is an $A\text{-}A^x$-bimodule.

If $B$ is another dg algebra and $M$ is an $A$-$B$-bimodule, then we can define the {right twist} $$^{[x]}M\coloneqq ^{[x]}A \otimes_A M\cong \Hom_A(A^{[x]},M)$$which is an $A^x$-$B$-bimodule. Similarly, if $N$ is a $B$-$A$-bimodule then it has a {left twist}$$N^{[x]}\coloneqq N \otimes_A A^{[x]}\cong \Hom_A(^{[x]}A,N)$$which is a $B$-$A^x$-bimodule. In particular if $L$ is an $A$-bimodule then it has a {two-sided twist} $$L^x\coloneqq {}^{[x]}A\otimes_A L \otimes_A A^{[x]}\cong {}^{[x]}(L^{[x]})\cong (^{[x]}L)^{[x]}$$which is an $A^x$-bimodule.

\begin{lem}\label{obst1}
	Let $L \to A \to B$ be a square zero extension, and $(\partial,\xi)$ the associated Hochschild 2-cocycle.
	\begin{enumerate}
		\item An element $x \in \MC(B)$ determines a $d_L^x$-cocycle $\nu(x)\coloneqq \partial(x)+ \xi(x,x)$ in $L$.
		\item An element $x \in \MC(B)$ lifts to $\MC(A)$ if and only if $\nu(x)$ is a $d^x_L$-coboundary.
	\end{enumerate}
	We hence think of $\nu(x)\in H^2(L^x)$ as an obstruction class which vanishes if and only if $x$ lifts to $\MC(A)$.
\end{lem} 

\begin{proof}
	We begin with the first claim. The given element $\nu(x)$ is a $d_L^x$-cocycle by definition if and only if the equation \begin{equation}\label{dxlcocyc}d_L\partial(x) + d_L \xi(x,x) + [\partial(x),x] + [\xi(x,x),x]=0\end{equation} 
	holds. First, using (2) above and the fact that $x$ is an MC element, we compute $[x,\xi(x,x)]=\xi(d_Bx,x)+ \xi(x,d_Bx)$, so \ref{dxlcocyc} becomes \begin{equation}\label{dxlcocyc1}d_L\partial(x)+[\partial(x),x]+[d_B,\xi](x,x)=0.\end{equation}
	
	Next, we compute $[m_B,\partial](x,x)= [x,\partial(x)] + \partial(d_Bx)$, using again that $x$ is an MC element. Hence, using (1) above, the equation \ref{dxlcocyc1} becomes $$[m_B,\partial](x,x) + [d_B,\xi](x,x)$$which is zero by (3) above.
	
	For the second claim, a lift of $x$ is an MC element $x+l \in \MC(A)$. The MC equation for $x+l$ is satisfied if and only if $\nu(x) = -d^x_L(l)$. So a lift provides a coboundary, and conversely if $l$ is an element satisfying the previous equation then $x+l$ is an MC element; i.e.\ $x$ admits a lift.
\end{proof}

Recall that a right $A$-module chain map $^{[x]}A \to {}^{[y]}A$ is the same as an element $f\in A$ such that $\tilde fx=df+yf$. In particular, if $f\in A^0$ is invertible we get $y=fxf^{-1} - df. f^{-1}$, which is the usual formula for right gauge equivalences. Recall that two MC elements $x,y$ are {homotopy gauge equivalent} if there exists a quadruple $(f,g,h_1,h_2)$ such that
\begin{itemize}
	\item $f\in A^0$ satisfies $fx=df+yf$, i.e.\ defines a degree zero chain map ${}^{[x]}A \to {}^{[y]}A$.
	\item $g\in A^0$ satisfies $gy=dg+xg$, i.e.\ defines a degree zero chain map ${}^{[y]}A \to {}^{[x]}A$.
	\item $h_1\in A^{-1}$ satisfies $d^xh_1=gf - 1$.
	\item $h_2\in A^{-1}$ satisfies $d^yh_2=fg - 1$.
\end{itemize}

\begin{rem}
	One can check that the assignment $l \mapsto flg$ is a chain map $L^x \to L^y$. Similarly, $l\mapsto glf$ is a chain map $L^y \to L^x$. These chain maps should be homotopy inverse via homotopies constructed using the $h_i$.
\end{rem}

Say that a square zero extension $L \to A \to B$ is {semisplit} if the decomposition $A\cong B\oplus L$ is multiplicative. This is equivalent to the morphism $A^\# \to B^\#$ of graded algebras admitting a section. A square zero extension is semisplit if and only if $\xi=0$, in which case the obstruction element for $x\in\MC(B)$ is $\nu(x)=\partial(x)$. Moreover, in a semisplit extension, $\partial$ is a derivation. Observe that if $B^\#$ is a free algebra, then any square zero extension with base $B$ is semisplit.

\begin{lem}\label{obst2}
	Let $L \to A \to B$ be a semisplit square zero extension. Suppose that $x,y\in \MC(B)$ are homotopy gauge equivalent. If $x$ admits a lift to $\MC(A)$ then so does $y$.
\end{lem}

\begin{proof}
	Choose a homotopy gauge equivalence $(f,g,h_1,h_2)$ between $x$ and $y$. Because the extension was semisplit, $\partial$ is a derivation, and using this one can check that $d_L^y(\partial(f)g - flg) = \partial(y)fg$. Moreover, using that $\partial$ is a derivation and that $y$ is an MC element, one can check that $d^y_L(\partial(y)h_2) = \partial(y)fg - \partial(y)$. In particular, putting $l'\coloneqq flg - \partial(f)g + \partial(y)h_2$ we hence have $-d^y(l') = \partial(y) = \nu(y)$. Hence by \ref{obst1}(2),  $y+l'$ is a lift of $y$.
\end{proof}

\begin{proof}[Proof of \ref{squarezero1}]
	 Let $L \to A \xrightarrow{\pi} B$ be a square zero extension of curved algebras. We wish to show that $\MCdg(\pi)$ is a fibration of dg categories. If $A$ has no MC elements, then $\MCdg(\pi)$ is a fibration, since any functor whose source is the empty category is an isofibration. So we may assume that $A$ has an MC element $a$; hence $B$ has an MC element $\pi(a)$. Twisting $A$ and $B$ by these MC elements respectively, we may assume that both $A$ and $B$ are dg algebras and that $\pi$ is a morphism of dg algebras.
     
     Obviously $\pi$ is surjective on hom-complexes, so we need to check the following: suppose we are given $a\in \MC(A)$, $b\in \MC(B)$, and a homotopy gauge equivalence $\pi(a)\simeq b$. Then there exists $\tilde b \in \MC(A)$ and a homotopy gauge equivalence $a\simeq \tilde b$ lifting the given one. To show this, we will reduce to the universal example.
	
	Let $W$ be the universal dg algebra containing two MC elements $x,y$ and a homotopy gauge equivalence between then, so that $W^\#$ is simply the algebra $\ground\langle x, y,f,g,h_1,h_2\rangle$. A homotopy gauge equivalence $\pi(a)\simeq b$ gives a map $W \to B$; let $W'$ be the pullback of this map along $A \to B$. Let $g:W' \to A$ be the natural map. It follows that $W'\to W$ is a square zero extension, with kernel $L$. Since the underlying graded algebra of $W$ is free, this extension is semisplit.
	
	 If $u,v$ are two MC elements of an algebra $R$, then as in \cite[Lemma 5.3]{CHL21} they are homotopy gauge equivalent if and only if there exists $h\in \MC(R\otimes I^3)$ evaluating to $u$ and $v$ respectively along the induced maps $ \MC(R\otimes I^3) \to  \MC(R)$. In particular the algebra $W\otimes I^3$ admits an MC element $H$ which is a `universal homotopy gauge equivalence' from $x$ to $y$.
	
	Let $\kappa\coloneqq \ground\langle x; dx+x^2=0\rangle$ be the universal dg algebra admitting an MC element. There is an obvious map $\kappa \to W$ which by assumption admits a lift $\kappa \to W'$ (which corresponds to an MC element $w\in W'$ with $g(w)=a$). Tensoring with $I^3$ hence gives an MC element $X\coloneqq x\otimes 1 \in \MC(W\otimes I^3)$ which admits a lift to an MC element of $W'\otimes I^3$.
	
	Suppose that the MC element $H \in W\otimes I^3$ admits a lift to an MC element  $H' \in W'\otimes I^3$. Then $g(H')$ would be the desired homotopy gauge equivalence $a\simeq \tilde b$ lifting the given one. By \ref{obst2}, it suffices to show that $H$ is homotopy gauge equivalent to $X$, since we know that $X$ admits a lift.
	
	To see this, first observe that the obvious map $\kappa \to W$ is a 3-homotopy equivalence. Indeed, the natural map $\ground \to I_3$ is an MC equivalence of dg coalgebras, and hence $\Omega(\ground) \to \Omega(I_3)$ is a 3-homotopy equivalence of dg algebras. But this latter map is exactly the natural inclusion $\kappa \to W$.
	
	In particular, this 3-homotopy equivalence is witnessed by a map $T:W \to W\otimes I^3$ such that $T_0 = \id_W$ and $T_1$ is the composite $W \to \kappa \to W$. Let $T'$ denote the composite map $$T':W\otimes I^3 \xrightarrow{T\otimes \id}W\otimes I^3 \otimes I^3 \xrightarrow{\id\otimes\sigma}W\otimes I^3 \otimes I^3$$where $\sigma$ is the flip automorphism of $I^3 \otimes I^3$ interchanging the factors (which one can easily check is an algebra automorphism).
	
	By construction, if $Y$ is any MC element of $W\otimes I^3$, then $T'(Y)$ is an MC element of $W\otimes I^3 \otimes I^3$ which yields a homotopy gauge equivalence from $Y$ to $(T_1\otimes \id)(Y)$.
	
	In particular, taking $Y=H$ yields a homotopy gauge equivalence between $H$ and $(T_1\otimes \id)(H)$. But it is easy to check that $(T_1\otimes \id)(H)=X$. Hence $H$ and $X$ are homotopy gauge equivalent, and so $H$ lifts to an MC element of $W'\otimes I^3$, which give the desired lift of the homotopy gauge equivalence $\pi(a)\simeq b$.	
\end{proof}

\begin{rem}
	Given a homotopy gauge equivalence $\pi(a)\simeq b$, it is desirable to write down explicitly a homotopy gauge equivalence $a\simeq \tilde b$ lifting it. When the extension is semisplit, \ref{obst2} gives a formula for $\tilde b$. One could probably use this, together with the explicit characterisation of the universal homotopy gauge equivalence $H$ from the proof of \ref{squarezero1}, to write down an explicit gauge equivalence $a\simeq \tilde b$. We refrain from doing so here. It is less obvious how to proceed when the extension is not semisplit, and this poses an interesting problem. One may be able to adapt the methods of Braun's thesis \cite{BraunThesis}, which works in the setting of curved Lie algebras, to this case.
\end{rem}

If $\pi:A \to B$ is a square zero extension of dg algebras, it will be useful for us to have an explicit description of the fibres of the fibration $\pi_*$. By twisting we may reduce to the fibre above $0\in \MC(B)$, the computation of which is the content of the following lemma:
\begin{lem}\label{fibrelem}
    Let $\pi:A \to B$ be a square zero extension of dg algebras with fibre $L$. Let $\pi_*^{-1}(0)\into \MCdg(A)$ be the fibre of $\pi_*$ above $0\in \MCdg(B)$. Then:
    \begin{enumerate}
        \item The objects of $\pi_*^{-1}(0)$ are in bijection with $Z^1(L)$.
        \item The complex of morphisms $\pi_*^{-1}(0)(l,m)$ between two such elements is $\ground\oplus L$ equipped with the differential $(\lambda,x)\mapsto (0,dx+\lambda(m-l))$. In particular we have $\pi_*^{-1}(0)(l,l)\cong \ground\oplus L$.
        \item $H^0(\pi^{-1}(0))$ is a groupoid whose isoclasses of objects are in bijection with $H^1(L)$.
    \end{enumerate}
\end{lem}
\begin{proof}
Claim (1) follows easily from the proof of \ref{obst1}. For claim (2), we first compute the morphism complex $\MCdg(A)(l,m)$ between two $l,m \in Z^1(L)$ to be $A^\sharp=B^\sharp\oplus L^\sharp$ equipped with the upper-triangular differential $\begin{psmallmatrix} d_B & \partial + \phi_{l,m} \\ 0 & d_L\end{psmallmatrix}$, where $\phi_{l,m}$ sends $a$ to $ma-\tilde a l$ and we recall that by definition $\partial(b) = d_A(b)-d_B(b)\in L$. Taking the fibre product yields the desired description. For claim (3), first observe that a map of the form $1+x$, with $x\in L^0$, defines an isomorphism in $H^0(\pi^{-1}(0))$ from $l$ to $l+dx$, with inverse given by $1-x$. On the other hand, by (2), the $0$-cycles in the mapping complex $\pi_*^{-1}(0)(l,m)$ are given by the pairs $(\lambda, x) \in \ground\oplus L^0$ such that $dx=\lambda(m-l)$, so that if $l$ and $m$ are not cohomologous in $L$ then there is no morphism between them in $H^0(\pi^{-1}(0))$. It follows that $H^0(\pi^{-1}(0))$ is a groupoid, with isoclasses of objects given by $Z^1(L)/B^1(L) = H^1(L)$, as desired.
\end{proof}
\begin{rem}
Since the differential in the hom-spaces $\MCdg(A)(l,m)$ is upper-triangular, any isomorphism $f:l\to m$ in $H^0(\MCdg(A))$ can be written in the form $f=b(1+x)$, where $b=\pi_*(f)$ is an automorphism of $0 \in H^0\MCdg(B)$ and $1+x$ with $x \in L$ is an isomorphism that lifts $\id_0$. If $R$ is a dg algebra then the homotopy automorphisms of an MC element $r\in R$ are the same thing as the homotopy gauge equivalences from $r$ to itself. One can now easily observe that there is a natural group isomorphism $\mathrm{Aut}_{H^0\MCdg(B)}(0)\cong H^0(B)^\times$. It now follows from the above that we have a bijection $$\frac{\{x\in H^0\MCdg(A)\text{ such that }\pi(x)=0 \}}{\text{isomorphism}}\quad\cong\quad H^1(L)/H^0(B)^\times.$$
\end{rem}

\subsection{Twisted modules}
Our goal in this subsection is to prove the following result, linking properties of $\MCdg$ to properties of the larger category $\Twfg$:

\begin{prop}\label{MCtoTWprop}Let $f:A \to B$ be a map of curved algebras. Consider the following statements:
\begin{enumerate}
    \item For every finite dimensional vector space $V$, the induced functor $$f^V_*: \MCdg(A\otimes \mathrm{End}(V)) \ \to\  \MCdg(B\otimes \mathrm{End}(V))$$is a quasi-equivalence.
    \item The induced functor $f_*: \Twfg(A) \to \Twfg(B)$ is a quasi-equivalence.
    \item There exists a map $g:B\to A$ such that the pair of functors $f_*:\Twfg(A)\to \Twfg(B)$ and $g_*:\Twfg(B)\to \Twfg(A)$ are mutually inverse quasi-equivalences.
    \item For every finite dimensional vector space $V$, the induced functor $$f^V_*: \MCdg(A\otimes \mathrm{End}(V)) \ \to\  \MCdg(B\otimes \mathrm{End}(V))$$is a fibration.
    \item The functor $f_*: \Twfg(A) \to \Twfg(B)$ is a fibration.
    \end{enumerate}
    Then there are implications $(3)\implies(1)\implies(2)$ and $(4)\iff (5)$.

\end{prop}
\begin{proof}
Suppose first that (3) holds. First observe that each $\MCdg(A\otimes \mathrm{End}(V))$ sits inside $\Twfg(A)$ as a full dg subcategory: if $x,y$ are two MC elements of $A\otimes \mathrm{End}(V)$ then we have natural isomorphisms $\MCdg(A)(x,y)\cong {}^{[y]}[A\otimes \mathrm{End}(V)]^{[x]} \cong \Twfg(A)([A\otimes V]^{[x]},[A\otimes V]^{[y]})$. Moreover, these inclusions are compatible with the maps $f^V_*$, and it now follows that the $f^V_*$ are quasi-fully faithful. To see that each $f_*^V$ is quasi-essentially surjective, simply observe that $f^V_*g_*(x)=f_*g_*(x)\simeq x$. Hence (1) holds.

Assuming now that (1) holds, observe that, as $V$ varies, every finitely generated twisted $B$-module is in the image of some $f_*^V$. Hence $f_*$ is quasi-essentially surjective. To see that it is in addition quasi-fully faithful, let $M,N$ be finitely generated twisted $A$-modules corresponding to MC elements $u$ of $A\otimes \End(U)$ and $v$ of $A\otimes \End(V)$ respectively. We then have an isomorphism $\Twfg(A)(M,N)\cong {}^{[v]}[A\otimes \Hom(U,V)]{}^{[u]}$. Putting $W=U\oplus V$, we obtain MC elements $u'=\begin{psmallmatrix} u & 0 \\ 0 & 0\end{psmallmatrix}$ and $v'=\begin{psmallmatrix} 0 & 0 \\ 0 & v\end{psmallmatrix}$ of $A\otimes \End(W)$. Since morphisms between $u'$ and $v'$ decompose into block matrices, it follows that $\Twfg(A)(M,N)$ is an iterated retract of $\MCdg(A\otimes \End(W))(u',v') \simeq \MCdg(B\otimes \End(W))(f_*u',f_*v')$. But by the same argument, $\Twfg(B)(f_*M,f_*N)$ is the same iterated retract of $\MCdg(B\otimes \End(W))(f_*u',f_*v')$. It now follows that $f_*$ is quasi-fully faithful, and hence $(2)$ holds as desired.

For the remaining equivalence, if $f_*$ is a fibration, then so is its pullback along the inclusion $\MCdg(B\otimes \End(V)) \into \Twfg(B)$ for any $V$. But this pullback is precisely $f_*^V$, and hence $(5)$ implies $(4)$. To see that the converse holds, proceed as above for the isomorphism lifting, and use that surjectivity is preserved under retracts.
\end{proof}

\begin{cor}\label{SquareZeroFibBest}
   Let $\pi:A \to B$ be a square zero extension of curved algebras. Then $\Twfg(\pi)$ is a fibration of dg categories.
\end{cor}
\begin{proof}
If $V$ is a finite dimensional vector space, then $A\otimes \mathrm{End}(V) \to B \otimes \mathrm{End}(V)$ is a square zero extension. Hence by \ref{squarezero1}, it induces a fibration on $\MCdg$. Hence by \ref{MCtoTWprop}, we are done.
\end{proof}

\section{Structure theory for curved coalgebras}\label{section:structure}

In this section, our aim is to give some structure theory for curved coalgebras. More specifically, we show in \ref{CogStructThm} that an arbitrary injection of curved coalgebras is a relative cell complex for maps of the following form:
\begin{enumerate}
    \item Cosquare zero extensions of finite dimensional curved coalgebras.
    \item Injections between finite dimensional curved cosemisimple coalgebras.
\end{enumerate}
In the next subsections, we explain the meaning of the above terms and give proofs. Before we begin, we recall the definition of a relative cell complex.

If $\alpha$ is an ordinal, an $\alpha$-sequence in a category $\mathcal{C}$ is a cocontinuous functor $X:\alpha \to \mathcal{C}$, i.e.\ a collection of elements $X_\beta$ for every $\beta\in\alpha$, together with successor maps $X_\beta\to X_{\beta+1}$, and such that if $\beta=\varinjlim_{\gamma\in\beta}$ is a limit ordinal there is a natural isomorphism $X_\beta\cong \varinjlim_{\gamma\in\beta}X_\gamma$. The transfinite composition of an $\alpha$-sequence $X$ is the natural morphism $X_0 \to \varinjlim_{\beta\in\alpha}X_\beta$. 

Recall that if $\mathcal{K}$ is a class of morphisms in a category, then a {$\mathcal{K}$-relative cell complex} is a transfinite composition of the form $X_0 \to \varinjlim_{\alpha\in\lambda} X_\alpha$ for some ordinal $\lambda$ and $\lambda$-sequence $X_\alpha$, where each $X_{\alpha+1}$ is a pushout of a span of the form $X_\alpha \leftarrow Y_\alpha \xrightarrow{f_\alpha}  Y_\alpha' $, where each $f_\alpha$ is a coproduct of morphisms from $\mathcal{K}$. We denote the class of all $\mathcal{K}$-relative cell complexes by $\mathrm{Cell}(\mathcal{K})$.
\subsection{Conilpotent extensions}

A {conilpotent extension} of curved coalgebras is an injective morphism $i:C' \into C$ such that the quotient $C/C'$ is conilpotent. 

\begin{prop}\label{injgen}
    Let $i$ be an injection of curved coalgebras. Then $i$ is a relative cell complex for injections of finite dimensional curved coalgebras. If $i$ is a conilpotent extension, then it is a relative cell complex for conilpotent extensions of finite dimensional curved coalgebras.
\end{prop}

\begin{proof}
	Let $i:C' \into C$ be an injection. Pick $x\in C\setminus C'$ and let $D\subseteq C$ be a finite dimensional curved coalgebra containing $x$. Consider the pushout $C'\sqcup_{C\cap D}D$, which is a subcoalgebra of $C$ containing $x$. The map $C'\cap D \into D$ is clearly an injection of finite dimensional coalgebras. Continuing this process transfinitely gives the required cell decomposition in the first case. The second case is the same; one only needs to observe that the extension $C'\cap D \into D$ is in addition a conilpotent extension.
\end{proof}

An extension of curved coalgebras  $i:C' \into C$ is {cosquare zero} if the reduced comultiplication $\bar\Delta$ on the quotient $C/'C$ satisfies $\bar\Delta^2=0$. Clearly a cosquare zero extension is conilpotent. Note that $C' \into C$ is cosquare zero if and only if the extension $C^* \to C'^*$ of curved pseudocompact algebras is square zero.
\begin{lem}\label{cnlem2}
	A conilpotent extension between finite dimensional curved coalgebras factors as a finite composition of cosquare zero extensions.
\end{lem}
\begin{proof}
	We reduce to the case of algebras, where this is a standard argument. Let $C'\into C$ be a conilpotent extension. Put $A\coloneqq C^*$ and $A'\coloneqq C'^*$. This yields a nilpotent extension $A \twoheadrightarrow A'$ of finite dimensional curved algebras; let $I$ be its kernel, which is a nilpotent curved ideal of $A$. The powers of $I$ yield a filtration $\cdots\into I^3 \into I^2 \into I^1=I$ of $I$ by ideals. Since $I$ is a nilpotent ideal in a finite dimensional algebra, we must have $I^N=0$ for some $N$. Consider the resulting tower of extensions $$A=A/I^N  \twoheadrightarrow A/I^{N-1} \twoheadrightarrow \cdots  \twoheadrightarrow A/I^2  \twoheadrightarrow A/I^1=A'$$which factorises $A \twoheadrightarrow A'$. It is easy to see that each extension in the tower is square zero, and hence dualising the tower gives the desired factorisation of $C'\into C$.
\end{proof}
\begin{rem}
	One could instead argue directly using the coradical filtration on $C$.
\end{rem}
\begin{cor}
    A conilpotent extension of curved coalgebras is a relative cell complex for cosquare zero extensions of finite dimensional curved coalgebras.
\end{cor}

\subsection{Curved semisimple algebras}

We develop some structure theory for finite dimensional curved algebras, in particular a notion of curved semisimplicity. Related results were obtained by Orlov in the dg setting \cite{Orlov19}, and we use the same terminology of internal and external radicals.

The following lemma is useful and we will implicitly make use of it several times. Recall that a graded algebra is graded simple if it has no nontrivial graded ideals.
\begin{lem}
    Let $A$ be a finite dimensional graded algebra. Then $A$ is graded simple if and only if $A$ is a simple algebra equipped with a grading.
\end{lem}
\begin{proof}
    Certainly if $A$ is simple then it is graded simple. For the converse, note that $Z(A)$ is a finite graded field extension of $k$ and hence itself a field. The result now follows from a theorem of Jespers \cite{jespers}.
\end{proof}

Let $A$ be a finite dimensional curved algebra and let $J$ be the radical of $A^\#$. Recall that the maximal semisimple quotient of $A^\#$ is the semisimple graded algebra $A^\#/J$. Define the internal curved radical to be $J_-\coloneqq \{x\in J: dx\in J\}$. It is easy to see that $J_-$ is closed under $d$ and is hence a curved ideal. Since $A$ was finite dimensional it is a nilpotent ideal. The quotient $B\coloneqq A/J_-$ is a finite dimensional curved algebra, and the surjection $A^\#\to B^\#$ of finite dimensional graded algebras induces an isomorphism between their maximal semisimple quotients.

\begin{prop}\label{cssdef}
	Let $B$ be a finite dimensional curved algebra. The following are equivalent:
	\begin{enumerate}
		\item The internal curved radical of $B$ vanishes.
		\item $B$ is a quotient $A/J_-$ with $A$ finite dimensional and $J_-$ its internal curved radical.
	\end{enumerate}
	If either of these holds, we say that $B$ is curved semisimple.
\end{prop}
\begin{proof}
	Clearly $(1)$ implies $(2)$ since we can take $A=B$. For the converse, take $A$ finite dimensional, with radical $J$, and let $B\coloneqq A/J_-$. Let $J'_-$ be the internal curved radical of $B$. An element $x\in J'_-$ is represented by an element $j\in J$ with $dj\in J$. But then $j\in J_-$ and hence $x=0$, and so $J'_-$ vanishes.
\end{proof}

Recall that we chose $\ground$ to be a perfect field; the following is where we use that hypothesis.

\begin{prop}\label{structuretheoremCSS}
	Let $B$ be a finite dimensional curved semisimple algebra. Then $B$ is curved isomorphic to a finite product of algebras of the following two types, which we call curved simple:
	\begin{enumerate}
		\item A dg algebra $R$ whose underlying graded algebra is simple.
		\item A dg algebra $R$ of the form $S\otimes K$, where $S$ is a simple graded algebra with zero differential and $K$ is the acyclic dg algebra $K={\ground[x]}/{x^2}$ with $x$ in cohomological degree $-1$ and $dx=1$.
	\end{enumerate}
\end{prop}

\begin{proof}
	
	Let $B$ be a curved semisimple algebra and let $J$ be the radical of $B^\#$. The external curved radical of $B$ is $J_+\coloneqq J+dJ$. Since $d^2J=[h,J]\subseteq J$, this is a curved ideal. Note that $J\cap dJ=dJ_-$, so the sum is direct. Moreover, since the kernel of $d$ on $J$ is a subspace of $J_-$, the map $d:J \to dJ$ must be an isomorphism. In particular if $B$ is a dg algebra then $J_+$ is acyclic.

	Let $C\coloneqq B/J_+$ be the quotient, which is a curved semisimple algebra with the additional property that $C^\#$ is semisimple. Note that the maximal semisimple quotient of $C^\#$ may be smaller than that of $B^\#$. Since $\ground$ is perfect, the Wedderburn--Malcev theorem gives a linear splitting $B^\#\cong J^\#\oplus (B/J)^\#$, where the second summand is a graded subalgebra. Since $C^\#$ is a subalgebra of $(B/J)^\#$, it follows that $(dJ)^\#$ is a semisimple subalgebra of $B^\sharp$, disjoint from $C^\#$, and moreover we have a linear splitting $B^\#\cong C^\#\oplus(J_+)^\#$. Both summands are graded subalgebras. Note that $C^\#$ and $(dJ)^\#$ are orthogonal, since their sum is a semisimple algebra. Hence if $c\in C$ and $j\in J$, we have $d(cj)=d(c)j$ by the Leibniz rule. The left hand side of this equality is in $dJ$ and the right hand side is in $J$, so both are zero. Since $d$ is an injection on $J$, it follows that $cj=0$. Hence we have a splitting of graded algebras $B^\#\cong C^\#\oplus(J_+)^\#$.
	
	If $h$ is the curvature element of $B$, write $h=h_C+h_{J}$ with $h_C\in C$ and $h_J\in J_+$. Since $C$ and $J_+$ are orthogonal, we have $[h,c]=[h_C,c]$ for all $c\in C$ and similarly for $J_+$. The subalgebra $J_+$ is certainly closed under $d$, and the curvature element $h_J$ makes it into a curved algebra. The quotient map $B \to J_+$ is a morphism of curved algebras.
	
	Since $C$ is separable, it is projective as a $C$-bimodule. Hence if $M$ is a $C$-bimodule then $HH^1(C,M)$ vanishes, and in particular all derivations $C \to M$ are inner (i.e.\ of the form $[m,-]$ for some $m \in M$). In particular the derivation $d:C \to B$ is inner and hence given by $dc=[b,c]$ for some $b\in B^1$. Since $C$ is orthogonal to $J_+$ we may as well take $b\in C$. It is then clear that $C$ is closed under the differential, and $h_C$ makes it into a curved algebra. The quotient map $B \to C$ is a morphism of curved algebras. This exhibits the curved algebra $B$ as the product $C\times J_+$. What remains is to analyse the simple summands of $C^\#$ and $(dJ)^\#$. The curved simple algebras of type (1) will come from $C$ and those of type (2) will come from $dJ$.
	
	If $S\subseteq C^\#$ is a simple summand, then as before the differential $d:S\to C$ is given by a commutator $ds=[b,s]$ for some $b\in C$. Since $S$ is orthogonal to all the other summands of the semisimple algebra $C^\#$, this shows that $d$ restricts to a differential on $S$. Similarly the curvature element $h_C$ restricts to a curvature element $h_S$ and it is clear that $C \to S$ is a morphism of curved algebras. Using that $d^2s=[h_S,s]$, one can easily check that $b^2-h_S$ is a central element. If $R$ is any ring then the centre of $M_n(R)$ consists of the diagonal matrices over the centre of $R$. Since the centre of a division ring is a field, and a finite graded field extension of $\ground$ must be in degree zero, the centre of a graded semisimple algebra must be in degree zero. Hence $b^2-h_S$ must be zero for degree reasons, and so $b^2=h$. Using this one can check that $-b$ is an MC element, and twisting by this MC element yields a curved isomorphism between $S$ and an algebra of type (1). Hence $C$ is (curved isomorphic to) a finite product of algebras of type (1).
	
	The analysis of $J_+$ is a little more complicated. First note that $J$ is an ideal of $J_+^\#$ and the quotient $J_+^\#/J$ is precisely $(dJ)^\#$.
	If $x,y\in J$, then consider $d(xy)=d(x)y\pm x dy$. The left hand side is in $dJ$ and the right hand side is in $J$. Hence both sides are zero and hence $xy=0$. Hence the sequence $J \to J_+^\# \to (dJ)^\#$ is a square zero extension. Let $u,v\in J$. We have $d(udv)=dudv\pm u[h_J,v]$. Since the first two terms are in $dJ$ and the third is in $J$, we must have $u[h_J,v]=0$ and $d(udv)=dudv$. In other words, $d:J \to dJ$ is a right $dJ$-module map. Similarly it is also a left module map. Let $x\in J$ be the element with $dx=1$. We see that if $u\in dJ$ then $d(ux)=u$, and so the action of $dJ$ on $J$ is inverse to the isomorphism $d:J \to dJ$. This gives us an an isomorphism $J^\#_+\cong (dJ)^\#[x]/x^2$ of graded algebras, where $x$ has cohomological degree $-1$. If $u\in dJ$ then we compute $du=d^2(ux)=[h,ux]=[h,u]x$. 
	
 One can easily check that the element $-hx$ is an MC element of $J_+$, and hence twisting by it we obtain a dg algebra $W$. The underlying graded algebra of $W$ is $W'[x]/x^2$, where $W'$ is semisimple. If $w\in W'$ then the differential satisfies $dw=0$ and $d(wx)=w$. Hence $W$ is isomorphic to the tensor product $W'\otimes K$, where $W'$ is regarded as a dg algebra with zero differential and $K$ is the acyclic dg algebra $K={\ground[x]}/{x^2}$ with $x$ in cohomological degree $-1$ and $dx=1$.
	
	If $S$ is a simple summand of $W'^\#$, let $S'\subseteq W'x$ be the subspace such that $dS'=S$. Clearly $S'\oplus S$ is a dg subspace of $W$. If $u,v\in S'$ then we have $d(udv)=dudv\in S$ and hence $udv\in S'$. Hence $S'$ is a right $S$-module, and by similar arguments it is an $S$-bimodule. In particular, $S' \oplus S$ is also a square zero extension $S[x]/x^2$, where $x$ now denotes the element with $dx=1_S$. As before, $S'\oplus S$ is isomorphic to $S\otimes K$. There is a natural quotient map $W\to S\otimes K$ exhibiting $W$ as (curved isomorphic to) the finite product of algebras of type (2).
\end{proof}
\begin{cor}\label{CSSsection}
	A surjection of curved semisimple algebras admits a section.
\end{cor}
\begin{proof}
	Let $\pi:A\twoheadrightarrow B$ be a surjection between curved semisimple algebras. It will suffice to show that the ideal $I\coloneqq \ker(\pi)$ is a product of curved simple subalgebras, since then $\pi$ will restrict to an isomorphism $A' \to B$ whose inverse will be the desired section. For a curved simple subalgebra $R$ of $A$, the space $I\cap R$ is a curved ideal of $R$. We wish to show that it is either $0$ or $R$. Clearly if $R$ is of type (1) then this holds since $I^\#$ is an ideal of $R^\#$. So we may assume that $R$ is of the form $S[x]/x^2$ with $S^\#$ simple and differential as in a type (2) algebra. Let $J$ be a nonzero curved ideal of $R$. If $J$ contains a nonzero element $s\in S$ then it must contain all of $S$, and hence all of $Sx$, and hence $J=R$. If not, then it must contain an element of the form $sx$ with $s\neq 0$, but since it is a curved ideal it must contain $d(sx)=s$ and as before we have $J=R$. So we are done.
\end{proof}

\begin{rem}
If one works with a $\Z/2$-grading (e.g.\ if one is interested in matrix factorisations) then in \ref{structuretheoremCSS} the algebras of type (1) must be replaced with the more general class of algebras
\begin{itemize}
	\item[(1')] a curved algebra $R$ whose underlying graded algebra is simple.
	\end{itemize}
The reason is that in the proof, we used that $b^2-h$ was a central element of cohomological degree $2$, and hence zero. In the $\Z/2$-graded setting this does not hold.
	\end{rem}

\begin{rem}
	Call a finite dimensional dg algebra dg semisimple if, when regarded as a curved algebra with zero curvature, it is curved semisimple. Orlov proves in \cite{Orlov19} that a dg semisimple algebra has a semisimple derived category. In fact, a curved semisimple algebra also has a semisimple derived category of the second kind. Curved simple algebras of type (1) are curved isomorphic to simple algebras with zero differential, and such an algebra has semisimple derived category of the second kind. Curved simple algebras of type (2) have vanishing derived category of the second kind, since every module is homotopy equivalent to zero.
\end{rem}

\subsection{Curved cosemisimple coalgebras}

Now we will transfer the above to the setting of coalgebras. If $C$ is a finite dimensional curved coalgebra, say that $C$ is curved cosemisimple if its linear dual $C^*$ is a curved semisimple algebra. A finite dimensional curved coalgebra $C$ has a maximal curved cosemisimple subcoalgebra $R$, given as the linear dual of the quotient of $C^*$ by its internal curved radical. We refer to this subcoalgebra as the curved coradical. Note that since the internal curved radical is nilpotent, the map $R\into C$ is a conilpotent extension. If $R'\into C$ is the inclusion of another curved cosemisimple coalgebra, then it must factor through an inclusion $R' \into R$.

\begin{theorem}\label{CogStructThm}
Let $i:C \into C'$ be an injection of curved coalgebras. Then $i$ is a relative cell complex for maps of the following form:
\begin{enumerate}
    \item Cosquare zero extensions of finite dimensional curved coalgebras.
    \item Injections between finite dimensional curved cosemisimple coalgebras.
\end{enumerate}

\end{theorem}
\begin{proof}
 By \ref{injgen}, it is enough to show that an injection of finite dimensional curved coalgebras can be presented as a finite relative cell complex for maps of the given form. So assume that $i:C\into C'$ is an injection between finite dimensional curved coalgebras. Let $R$ be the curved coradical of $C$, and similarly for $R'$. The composition $R\into C'$ factors through $R'$. Let $D$ be the pushout of the span $C \leftarrow R \to R'$, which is naturally a curved subcoalgebra of $C'$. Since $D$ contains $R'$, the inclusion $D\into C'$ is a conilpotent extension, and hence a finite composition of cosquare zero extensions by \ref{cnlem2}. So the composition $C \to D \to C'$ is a finite relative cell complex of the given form: $C\to D$ is a pushout along a map from $(2)$, whereas $D \to C'$ is a composition of maps from (1).
\end{proof}

\section{Convolution algebras and {\normalfont II}-Morita equivalences}\label{section:morita}

We introduce the notion of {\normalfont II}-Morita equivalence between curved algebras. The main result of this section is that convolution with an arbitrary curved coalgebra preserves certain kinds of well-behaved {\normalfont II}-Morita equivalences. Besides being of independent interest, this is a key technical result.

\subsection{{\normalfont II}-Morita equivalences}

We begin with two lemmas on pretriangulated dg categories.

 \begin{lem}\label{pretrqff}
 	Let $F:\mathcal{A} \to \mathcal{B}$ be a dg functor between pretriangulated dg categories. Then $F$ is a quasi-equivalence if and only if $H^0(F)$ is a triangle equivalence.
 \end{lem}	
 \begin{proof}
 	The forward implication is clear, so assume that $H^0(F)$ is a triangle equivalence. It is clear that $F$ is quasi-essentially surjective, so we just need to check that it is quasi-fully faithful; i.e.\ that for every $x,y \in \mathcal{A}$ the map $F_{xy}:\mathcal{A}(x,y)\to \mathcal{B}(Fx,Fy)$ is a quasi-isomorphism. But because $\mathcal{A}$ and $\mathcal{B}$ are pretriangulated, $H^iF_{xy}$ can be regarded as a map $$H^0(\mathcal{A})(x,y[i])\to H^0(\mathcal{B})(Fx,Fy[i])$$which by assumption is an isomorphism.
 \end{proof}

 \begin{lem}\label{cpctlem}
    Let $F:\mathcal{A} \to \mathcal{B}$ be a functor of pretriangulated dg categories. Suppose that:
    \begin{itemize}
    \item The triangulated categories $H^0(\mathcal{A})$ and $H^0(\mathcal{B})$ are compactly generated.
        \item $H^0(F)$ preserves direct sums.
    \end{itemize}
    Then $F$ is a quasi-equivalence if and only if it induces a triangle equivalence $H^0(\mathcal{A})^c \to H^0(\mathcal{B})^c$.
\end{lem}
\begin{proof}
The forward implication is clear, so we need only consider the backwards implication. By \ref{pretrqff} we need only show that if $H^0(F)$ restricts to an equivalence on compacts then it is an equivalence. Since $H^0(\mathcal{A})$ is compactly generated, it is the smallest triangulated subcategory of itself containing $H^0(\mathcal{A})^c$ and closed under direct sums \cite[Lemma 2.2.1]{SSstable}. Since the same holds for $H^0(\mathcal{B})$, and $F$ preserves direct sums and is an equivalence on compact objects, the claim follows.
\end{proof}

    Say that a morphism $f:A \to A'$ of curved algebras is a {\normalfont II}-Morita equivalence if the induced map $f_*:\Perfcoderiveddgc(A) \to \Perfcoderiveddgc(A')$ of dg categories is a quasi-equivalence. The below proposition follows by a repeated application of \ref{pretrqff} and \ref{cpctlem}:
\begin{prop}
    Let $f:A \to A'$ be a morphism of curved algebras. Then the following are equivalent:
    \begin{enumerate}
        \item $f$ is a {\normalfont II}-Morita equivalence.
        \item $f_*:\Perfcoderivedc(A) \to \Perfcoderivedc(A')$ is a triangle equivalence.
        \item $f_*:\Dcoderivedc(A) \to \Dcoderivedc(A')$ is a triangle equivalence.
    \end{enumerate}
\end{prop}
Clearly a curved isomorphism is a II-Morita equivalence. 

	We will need to know that dg algebras represent all II-Morita equivalence classes of curved algebras. The following construction appears in \cite{hochschildII} and is due to Julian Holstein.
	
	\begin{prop}\label{dgIIMorita}
		Let $A$ be a curved algebra. Then there exists a dg algebra $B$ together with a {\normalfont II}-Morita equivalence $A\to B$. If $C$ is a curved coalgebra then the induced morphism $\Hom(C,A) \to \Hom(C,B)$ is a {\normalfont II}-Morita equivalence.  The assignment $A\mapsto B$ is functorial with respect to uncurved maps $A\to A'$. 
		\end{prop}
		\begin{proof}
			Let $M$ be $A\oplus A[1]$, with differential given by the square matrix $x=\bigl(\begin{smallmatrix}0 & 1 \\ h & 0\end{smallmatrix}\bigr)$. Since the differential squares to $h$, we see that $M$ is a finitely generated twisted $A$-module (which is homotopy equivalent to $0$) given by the MC element $x\in A\otimes \End(V)$, where $V$ is $\ground\oplus\ground[1]$. There is an obvious II-Morita equivalence $A\to A\otimes \End(V)$. Twisting $A\otimes \End(V)$ by $x$ yields a dg algebra $B$ and a curved isomorphism $A\otimes \End(V) \to B$. By composition we obtain a II-Morita equivalence $A\to B$, as desired. To check the claim about convolution algebras, let $C$ be a curved coalgebra. Since $\Hom(C,A\otimes \End(V))$ is naturally isomorphic to $\Hom(C,A)\otimes \End(V)$, we see as before that the natural map $\Hom(C,A)\to \Hom(C,A\otimes \End(V))$ is a II-Morita equivalence. Since the natural map $\Hom(C,A\otimes\End(V)) \to\Hom(C,B)$ is a curved isomorphism, the induced map $\Hom(C,A)\to \Hom(C,B)$ is a II-Morita equivalence, as desired. Finally, if $f:A\to A'$ is an uncurved morphism, then the induced morphism $A\otimes \End(V) \to A'\otimes \End(V)$ preserves the curvature element, and in particular the MC elements constructed above. Hence one obtains a natural map $B \to B'$ of dg algebras. 
		\end{proof}

\begin{rem}\label{iimorremark}
 In this paper, we will not axiomatise any notion of {\normalfont{II}}-Morita equivalence involving bimodules, nor will we consider any notion of {\normalfont{II}}-Morita equivalence for curved or dg categories. Rather, we only treat the case of when a morphism of curved algebras is a {\normalfont{II}}-Morita equivalence. This notion is too fine to give a well-behaved {\normalfont{II}}-Morita theory, since $\Perfcoderivedc(\ground)$ and $\Perfcoderivedc(M_n(\ground))$ are equivalent, but any such equivalence is not realised by a morphism $\ground \to M_n(\ground)$. In general, one hopes for the existence of a conjectural `Morita model structure of the second kind' on the category of (curved or) dg categories, whose homotopy category agrees with a {\normalfont II}-Morita homotopy category defined using bimodules, just as in the case of derived Morita theory for dg categories - recall that if two dg categories are derived Morita equivalent, then there is a zigzag of morphisms between them, all of which are derived Morita equivalences. 
\end{rem}

\subsection{Convolution algebras and colimits}

\begin{lem}\label{convlimits}
	Let $X$ be a curved algebra. The functor $C\mapsto \Hom( C, X)$, viewed as a functor $\ccogp \to \calgp$, sends colimits to limits.
\end{lem}
\begin{proof}
 If $X=0$ then the functor in question is the constant functor $C \mapsto 0$, which clearly sends colimits to limits. So we may assume that $X$ is nonzero. Let $F:J \to \ccogp$ be a diagram. If $*$ is in the image of $F$, then by the characterisation of colimits as universal cocones, we must have $\colim F \cong *$. Since $X\neq 0$, we have $\Hom(*,X)=\varnothing$, and by the same logic we must have $\lim \Hom(F-,X)\cong \varnothing$. So we may assume that $*$ is not in the image of $F$; i.e.\ $F$ is actually a diagram in $\ccog$.

It is enough to check that $\Hom(-, X)$ sends coproducts to products and coequalisers to equalisers. Recall that coproducts in curved coalgebras are created by the underlying graded coalgebras, by \ref{limscolims}. In particular, it is not hard to see that if $C,D$ are curved coalgebras then there is an isomorphism $\Hom(C\sqcup D,A)^\#\cong \Hom(C,A)^\#\times \Hom(D,A)^\#$. But products of curved algebras are also created underlying, by \ref{limscolims} again, and it follows that there is a natural isomorphism $$\Hom(C\sqcup D,A)\cong \Hom(C,A)\times \Hom(D,A)$$of curved algebras. 

Finally we check the statement about (co)equalisers, which is harder since they may not be created by the underlying vector spaces. To do this, we dualise. Let $(A,d_A,h_A)$ and $(B,d_B, h_B)$ be two pseudocompact curved algebras, and let $(f,u)$ and $(g,v)$ be two morphisms $A \to B$. If $E$ denotes their equaliser, we wish to show that the diagram $$E\hat\otimes X \to A\hat\otimes X\rightrightarrows B\hat\otimes X$$is an equaliser diagram of curved algebras.

Suppose that there exists $z \in A^1$ such that $u+fz=v+gz$. In this case, following the proof of \cite[3.30]{HL2020}, we see that $E$ has the following description. The underlying pseudocompact graded algebra of $E$ is the equaliser of the diagram $A^\# \rightrightarrows B^\#$ of pseudocompact graded algebras. The differential on $E$ is $d_E=d_A+[z,-]$. The curvature element of $E$ is $h_E=h_A+dz+z^2$. Consider $z'\coloneqq z\otimes 1 \in (A\hat\otimes X)^1$. It is clear that $u\otimes 1 + (f\otimes 1)z' = v\otimes 1 + (g\otimes1 )z'$, so one may compute the equaliser of $A\hat\otimes X\rightrightarrows B\hat\otimes X$ in exactly the same manner. Since the completed tensor product commutes with taking underlying pseudocompact graded algebras, one can now directly verify that $E\hat\otimes X \to A\hat\otimes X\rightrightarrows B\hat\otimes X$ is an equaliser diagram.

So we may assume that no such $z$ exists. It follows that $E$ must be $\varnothing$. In this case, $E\hat \otimes X$ is also $\varnothing$. If there existed a $z' \in  (A\hat\otimes X)^1$ such that $u\otimes 1 + (f\otimes 1)z' = v\otimes 1 + (g\otimes1 )z'$, then taking a basis for $X$ and comparing coefficients we obtain a $z$ with $u+fz=v+gz$, a contradiction. So no such $z'$ exists and we see that $\varnothing$ must be the equaliser of $A\hat\otimes X\rightrightarrows B\hat\otimes X$, as desired.
\end{proof}

\begin{cor}\label{mclimscor}
	The functor $\MCdg(-,A):\ccogp\to \dgcat'$ sends colimits to limits.
	\end{cor}
	\begin{proof}
		$\MCdg(-,A)$ is the composition of the two functors $\Hom(-,A)$ and $\MCdg$. The former sends colimits to limits by \ref{convlimits} and the latter preserves limits by \ref{mcdgadj}.
		\end{proof}

\subsection{Convolution and {\normalfont II}-Morita equivalences}

\begin{defi}
Let $A,B$ be curved algebras. A good {\normalfont II}-Morita equivalence is a pair of morphisms  $f:A \to B$ and $g:B\to A$ which induce mutually inverse quasi-equivalences on $\Perfcoderiveddgc$.
\end{defi}
Note that this condition already appears in \ref{MCtoTWprop}(3). Clearly if $(f,g)$ is a good {\normalfont II}-Morita equivalence then both $f$ and $g$ are {\normalfont II}-Morita equivalences. Our goal for the rest of this section is to prove the following theorem:
\begin{theorem}\label{prop:convolution}
    Let $f:A \to B$, $g:B\to A$ be a good {\normalfont II}-Morita equivalence. If $C$ is a curved coalgebra, then $\Hom(C,f)$ is also a {\normalfont II}-Morita equivalence.
\end{theorem}
To do this, we will use \ref{MCtoTWprop} to reduce to considering only the induced maps on $\MCdg$, and \ref{CogStructThm} to reduce to checking statements about cosemisimple coalgebras, square zero extensions, and a limiting step. We begin with the special case when $C$ is finite dimensional curved cosemisimple.
\begin{lem}\label{IIMEsimplelem}
 Let $A \to A'$ be a {\normalfont II}-Morita equivalence of curved algebras. If $C$ is a finite dimensional curved cosemisimple coalgebra, then $\Hom(C,A) \to \Hom(C,A')$ is a {\normalfont II}-Morita equivalence.
\end{lem}
\begin{proof}
    Factoring $A \to A'$ into a curved isomorphism and an uncurved morphism, we may assume that it is uncurved. By \ref{dgIIMorita} we may now in addition assume that both $A$ and $A'$ are dg algebras. In what follows write $R$ for $C^*$, so that $\Hom(C,-)\cong R\otimes-$. By the classification result \ref{structuretheoremCSS}, $R$ is a finite product of curved simple algebras. Since a finite product of {\normalfont II}-Morita equivalences is a {\normalfont II}-Morita equivalence, we may assume that $R$ is a curved simple algebra. If $R$ is of type (2) in the classification of \ref{structuretheoremCSS} then both $\Perfcoderivedc(R\otimes A)$ and $\Perfcoderivedc(R\otimes A')$ vanish, and in particular are equivalent, as required. So we may assume that $R$ is of type (1); i.e.\ $R$ is a dg algebra whose underlying graded algebra is simple. The Artin--Wedderburn theorem now tells us that $R^\sharp$ is of the form $M_n(K)$, where $K$ is a finite dimensional division algebra over $\ground$, and hence concentrated in degree zero. Moreover, the proof of \ref{structuretheoremCSS} shows that the differential on $R$ is given by a commutator, and hence we have an isomorphism $R\simeq \End_K(V)$ for some dg-$K$-vector space $V$. Consider the commutative diagram
    $$\begin{tikzcd}
        \Perfcoderivedc(A) \ar[r,"-\otimes_A A'"]\ar[d,"-\otimes_K V"]&  \Perfcoderivedc(A')\ar[d,"-\otimes_K V"] \\
          \Perfcoderivedc(R\otimes A) \ar[r,"-\otimes_A A'"]&  \Perfcoderivedc(R\otimes A') 
    \end{tikzcd}$$
    By assumption, the upper horizontal map is an equivalence, so to show that the lower horizontal map is an equivalence, it suffices to show that the functor $-\otimes_KV: \Twfg(A)\to \Twfg(R\otimes A)$ is an equivalence for arbitrary $A$. It is easily observed to be fully faithful. Essential surjectivity follows from the fact that every finitely generated twisted $R\otimes A$-module can be viewed as a finitely generated twisted $A$-module.
\end{proof}

\begin{lem}\label{IIMEsquarezerolem}
Let $f:A \to A'$ be a map of curved algebras and let $i:C \to C'$ be a cosquare zero extension of finite dimensional curved coalgebras. If $\MCdg(C,f)$ is a quasi-equivalence, then so is $\MCdg(C',f)$.
\end{lem}
\begin{proof}

Replacing $f$ by $\Hom(C,f)$ in the notation yields a commutative diagram of curved algebras
$$\begin{tikzcd}
    B\ar[d, two heads,"\pi"] \ar[r,"g"] & B'\ar[d, two heads,"\pi'"]\\
    A \ar[r,"f"] & A'
\end{tikzcd}$$
where both $\pi$ and $\pi'$ are square zero extensions, with fibres $L$ and $L'$, say. Moreover, we may now assume that $\MCdg(f)$ is a quasi-equivalence, and we wish to show that $\MCdg(g)$ is a quasi-equivalence. If $B$ has no MC elements then neither does $B'$ and we are done. So we may assume that $B$ has an MC element and hence, by twisting, that the above diagram is a diagram of dg algebras. In particular $f$ is a quasi-isomorphism, and it follows that the induced map $L \to L'$ is also a quasi-isomorphism. By \ref{squarezero1} we obtain a diagram of dg categories 
$$\begin{tikzcd}
    \MCdg(B)\ar[d, two heads,"\pi_*"] \ar[r,"g_*"] & \MCdg(B')\ar[d, two heads,"\pi_*'"]\\
    \MCdg(A )\ar[r,"f_*"] & \MCdg(A')
\end{tikzcd}$$with vertical maps fibrations and $f_*$ a quasi-equivalence. So to see that $g_*$ is a quasi-equivalence it suffices to show that the induced maps on the fibres above every $a\in \MC(A)$ are quasi-equivalences. By twisting again, it suffices to prove this for the MC element $0\in A$. But this follows from \ref{fibrelem}.
\end{proof}

\begin{prop}\label{IIMEintprop}
    Let $f: A \to B$, $g:B\to A$ be a good {\normalfont II}-Morita equivalence of curved algebras. Then for every curved coalgebra $C$, the induced map $\MCdg(C,f)$ is a quasi-equivalence.
\end{prop}
\begin{proof}
By \ref{CogStructThm}, the map $0 \to C$ is a relative cell complex for two types of morphism:
\begin{enumerate}
    \item Cosquare zero extensions of finite dimensional curved coalgebras.
    \item Injections between finite dimensional curved cosemisimple coalgebras.
\end{enumerate}
By \ref{squarezero1}, if $i$ is a morphism of the first type then $\MCdg(i,A)$ is a fibration. By \ref{CSSsection}, if $i$ is a morphism of the second type then $\MCdg(i,A)$ admits a retract and is hence a fibration. Since $\MCdg(-,A)$ sends colimits to limits by \ref{mclimscor}, and fibrations of dg categories are preserved under pullbacks, it follows that $\MCdg(C,A)$ is the limit of a Reedy fibrant diagram of dg categories of the form $\MCdg(C_\alpha,A)$ for $\alpha \in \lambda^\mathrm{op}$ for some ordinal $\lambda$. Hence $\MCdg(C,A)$ is also the homotopy limit of this diagram. The same holds for $\MCdg(C,B)$ and we may moreover factor $\MCdg(C,f)$ as a homotopy limit of morphisms of the form $\MCdg(C_\alpha,f)$. Since $\dgcat'$ is right proper, each of these morphisms is a quasi-equivalence by \ref{IIMEsimplelem} and \ref{MCtoTWprop} in case (1) and \ref{IIMEsquarezerolem} in case (2). Since a homotopy limit of quasi-equivalences is a quasi-equivalence we are done.
\end{proof}

\begin{proof}[Proof of \ref{prop:convolution}]
Recall that we are given a good II-Morita equivalence $f:A \to B$, $g:B\to A$, and we want to show that for every curved coalgebra $C$ that $\Hom(C,f)$ is a II-Morita equivalence. By \ref{MCtoTWprop}, it suffices to show that if $V$ is a finite dimensional vector space, the map $\MCdg(C,f\otimes\mathrm{End}(V))$ is a quasi-equivalence. By the hom-tensor adjunction for convolution algebras, this map agrees with $\MCdg(C\otimes E,f)$ where $E$ is the linear dual of the finite dimensional algebra $\mathrm{End}(V)$. But this latter map is a quasi-equivalence by \ref{IIMEintprop}.
\end{proof}

\begin{rem}In the situation of \ref{prop:convolution}, it follows by symmetry that both $\Hom(C,f)$ and $\Hom(C,g)$ are {\normalfont II}-Morita equivalences. In fact, following the proof, one sees that the pair $(\Hom(C,f),\Hom(C,g))$ is a good {\normalfont II}-Morita equivalence. We may hence interpret \ref{prop:convolution} as the statement that $\Hom(C,-)$ preserves good {\normalfont II}-Morita equivalences.
\end{rem}

\section{Maurer--Cartan equivalences}\label{section:MC2}
In this section we introduce the notion of Maurer--Cartan equivalence of curved algebras and coalgebras. For conilpotent dg coalgebras, it reduces to the usual notion of weak equivalence underpinning the conilpotent Koszul duality of \cite{Positselski11}, c.f.\ \ref{mcconil}. For \emph{cofibrant} dg algebras it reduces to the notion of quasi-isomorphism (\ref{mcconil}); for non-cofibrant ones it is generally finer than quasi-isomorphism.

\subsection{Maurer-Cartan equivalences}

If $A$ is a curved algebra and $C$ is a curved coalgebra, we denote by $\MCmod(C,A)$ the set of isoclasses of objects of the category $H^0\MCdg(C,A)$. We call this set the the MC moduli set. Note that if $[X,Y]_3$ denotes the set of $3$-homotopy classes of maps $X\to Y$, then we have isomorphisms $$[\Omega C,A]_3 \cong \MCmod(C,A)\cong [C,\check B A]_3.$$

\begin{defi} \hfill \phantom{a}
	\begin{enumerate}
		\item
		A map $C\to C'$of curved coalgebras is called a Maurer--Cartan equivalence if for any curved algebra $A$ the induced map $\MCmod(C',{A})\to\MCmod(C,{A})$ is a bijection.
		\item
		A map $A\to A'$ of curved algebras is called a Maurer--Cartan equivalence if for any curved coalgebra $C$ the induced map $\MCmod(C,A)\to\MCmod(C,A')$ is a bijection.
	\end{enumerate}
\end{defi}
We abbreviate Maurer--Cartan equivalence by MC equivalence. It is clear that a 3-homotopy equivalence (of either curved algebras or curved coalgebras) is automatically an MC equivalence. It is easy to see that MC equivalences satisfy the two-out-of-three property. The notion of an MC equivalence admits several equivalent characterisations.
\begin{prop}\label{prop:MCcharacterization1} Let  $g:C\to C'$ be a map of curved coalgebras. The following are equivalent:
	\begin{enumerate}
		\item The map $g$ is an MC equivalence.
		\item For any curved algebra $A$ the induced map
		\[
		[C',\check{B}A]_3\to [C, \check{B}A]_3
		\] 
		is a bijection.
		\item For any curved algebra $X$ the induced map
		\[
		[\Omega(C'),A]_3\to [\Omega(C),A]_3
		\] 
		is a bijection.
		\item
		The induced map of curved algebras $\Omega(g):\Omega C\to \Omega C'$ is a 3-homotopy equivalence.
		\item The induced map of curved algebras $\Omega(g):\Omega C\to \Omega C'$ is an MC equivalence.
	\end{enumerate}
\end{prop}
\begin{proof}
	The equivalence of (1), (2), and (3) is clear. Clearly (4) implies (3). To see that (3) implies (4), apply the Yoneda lemma to the category whose objects are those of the form $\Omega(D)$ where $D$ is a curved coalgebra, and whose morphisms are given by 3-homotopy classes of curved algebra maps. To see the equivalence of (4) and (5), note that the map $\MCmod(D,\Omega(C))\to \MCmod(D,\Omega(C'))$ is an isomorphism if and only if the map $[\Omega(D),\Omega(C)]^*_3 \to [\Omega(D),\Omega(C')]^*_3 $ is, and the latter map being an isomorphism is equivalent to (4) by the same Yoneda lemma argument. 
\end{proof}

Exactly the same (or dual) arguments give the following result.

\begin{prop}\label{prop:MCcharacterization2} Let  $f:A\to A'$ be a map of curved algebras. Then the following conditions are equivalent.
	\begin{enumerate}
		\item The map $f$ is an MC equivalence.
		\item For any curved coalgebra $C$ the induced map 
		\[
		[\Omega(C),A]_3\to  [\Omega(C), A']^*_3
		\]  is a bijection.
		\item For any curved coalgebra $C$ the induced map 
		\[
		[C,\check{B}(A)]_3\to  [C, \check{B}(A')]_3
		\]  is a bijection.
		\item
		The induced map of curved coalgebras $\check{B}A\to \check{B}A'$ is a 3-homotopy equivalence.
		\item The induced map curved coalgebras $\check{B}A\to \check{B}A'$ is an MC equivalence. 
	\end{enumerate}
\end{prop}

So the functors $\check{B}$ and $\Omega$ both preserve and reflect MC equivalences. The following result shows that the unit and counit of the bar-cobar adjunction are both MC equivalences.

\begin{theorem}\label{thm:barcobarequiv}
	Let $A$ be a curved algebra and let $C$ be a curved coalgebra. Then:
	\begin{enumerate}
		\item 
		The unit of the adjunction $\eta:C\to \check{B}\Omega(C)$ is an MC equivalence.
		\item
		The counit of the adjunction $\epsilon:\Omega\check{B}(A)\to A$ is an MC equivalence.
	\end{enumerate}
\end{theorem}
Before we prove this, we first give a preliminary lemma.
\begin{lem}\label{lem:MCequiv}
	Let $A,B$ be curved algebras and $f:A\to B$, $g:B\to A$ a good {\normalfont II}-Morita equivalence. Then both $f$ and $g$ are MC equivalences.
\end{lem}	
\begin{proof}By symmetry we may consider the statement for $f$ only. Letting $C$ be any curved coalgebra, Proposition \ref{IIMEintprop} shows that $\MCdg(C,f)$ is a quasi-equivalence. Passing to $H^0$ we obtain a triangle equivalence $H^0\MCdg(\Hom(C,A)) \to H^0\MCdg(\Hom(C,A'))$, and thus in particular a bijection on isomorphism classes $\MCmod({C}, A)\cong\MCmod({C},A')$. Since $C$ was arbitrary, $A$ and $A'$ are hence MC equivalent.
\end{proof}

\begin{proof}[Proof of Theorem \ref{thm:barcobarequiv}]
	Let us prove (1).
	According to Proposition \ref{prop:MCcharacterization1}(5) it suffices to show that 
	$\Omega(\eta):\Omega(C)\to\Omega\check{B}\Omega(C)$ is an MC equivalence. Recall that the map $\eta:C\to\check{B}\Omega(C)$ induces an equivalence of triangulated categories 
	$\Dcoderivedc(C)\to \Dcoderivedc(\check{B}\Omega(C))$, cf. \cite{GL20}. Therefore $\Omega(\eta): \Omega(C)\to\Omega\check{B}\Omega(C)$ induces an equivalence 
	$\Dcoderivedc(\Omega(C))\to\Dcoderivedc(\Omega\check{B}(\Omega(C)))$. Moreover, by the zigzag identities for adjunctions, the map $\Omega(\eta)$ admits a one-sided inverse, so that there is a pair of maps 
	\[
	\Omega(C)\rightleftarrows\Omega\check{B}\Omega(C)
	\]
	that gives rise to mutually quasi-inverse functors on $\Dcoderivedc$. Hence this pair is a good {\normalfont II}-Morita equivalence. Applying Lemma \ref{lem:MCequiv}, we conclude that $\Omega(C)$ and $\Omega\check{B}\Omega(C)$ are MC equivalent via the map $\Omega(\eta)$, as required.

	Claim (2) is a formal consequence of (1). Indeed, it suffices, by Proposition \ref{prop:MCcharacterization2}(5), to show that  $\check{B}(\epsilon):\check{B}\Omega\check{B}(A)\to\check{B}(A)$ is an MC equivalence. By the zigzag identities again, combined with two-out-of-three for MC equivalences, $\check B (\epsilon)$ is an MC equivalence if and only if the map $\eta\circ \check B: \check B A \to \check{B}\Omega\check{B}(A)$ is an MC equivalence. This holds by part (1).
\end{proof}	

\begin{cor}\label{MCisDII}\hfill
	\begin{enumerate}
		\item If $f:C \to C'$ is an MC equivalence between two curved coalgebras then $\Dcoderivedc(f)$ is a triangle equivalence.
		\item If $f:A \to A'$ is an MC equivalence between two curved algebras then $\Dcoderivedc(f)$ is a triangle equivalence.
	\end{enumerate}
\end{cor}
\begin{proof}
	We begin with (1). Applying the natural transformation $\eta:\id\to \check B \Omega$ to $f$, followed by the functor $\Dcoderivedc$, yields a commutative square of triangulated categories 
	$$\begin{tikzcd} \Dcoderivedc(C) \ar[d,"\Dcoderivedc (f)", swap]\ar[r,"\Dcoderivedc (\eta_C)"]& \Dcoderivedc(\check B \Omega C)	\ar[d,"\Dcoderivedc(\check B \Omega f)"]\\
		\Dcoderivedc(C') \ar[r,"\Dcoderivedc (\eta_{C'})"]& \Dcoderivedc(\check B \Omega C').	\end{tikzcd}$$As in the proof of \ref{thm:barcobarequiv}, the maps running horizontally are triangle equivalences. Because $\check B \Omega f$ is a 3-homotopy equivalence by \ref{prop:MCcharacterization1} and \ref{prop:MCcharacterization2}, it follows from \ref{cor:3cohomotopy} that $\Dcoderivedc(\check B \Omega f)$ is a triangle equivalence. Hence $\Dcoderivedc(f)$ is a triangle equivalence, as desired. The proof of (2) is dual and uses \ref{cor:3homotopy} instead.
\end{proof}

\begin{rem}\label{mfrmk}
	Let $R$ be a commutative $\ground$-algebra and let $w\in R$. As in \cite{caldararutu, tu, becker} the triangulated category $\mathrm{MF}(R,w)$ of matrix factorisations of $w$ can be identified as the category of finite rank twisted modules over the $\Z/2$-graded curved algebra $R_w$, which is given by $R$ placed in even degree with zero differential and curvature element $w$. When $R$ is a noetherian regular complete local $\ground$-algebra, then $\mathrm{MF}(R,w)$ is equivalent to the singularity category of the hypersurface singularity $R/w$. Suppose that $R'$ is another commutative ring and $w'\in R'$. The above result shows that if the curved algebras $R_w$ and $R'_{w'}$ are MC equivalent, then $\mathrm{MF}(R,w)$ and $\mathrm{MF}(R',w')$ are triangle equivalent. This can be enhanced to an equivalence of $\Z/2$-graded dg categories. 
	\end{rem}

\subsection{Maurer--Cartan equivalences and dg categories}
Recall that the MC moduli set $\MCmod(C,A)$ is the set of isoclasses in the homotopy category of the dg category $\MCdg(C,A)$. In particular, if a map $f$ of curved algebras induces a quasi-equivalence on $\MCdg(C,f)$ for all curved coalgebras $C$, then $f$ is an MC equivalence, and similarly for coalgebras. We show that the converses of both these statements are also true.

\begin{lem}\label{mccat1}
Let $f:X \to X'$ be an MC equivalence of curved algebras. Then the induced dg functor $f_*:\MCdg(X) \to \MCdg(X')$ is a quasi-equivalence.
	\end{lem}
 \begin{proof}
 Since there is a natural isomorphism $X\cong \Hom(\ground,X)$ of curved algebras, we have a natural isomorphism $\MCdg(X)\cong \MCdg(\ground,X)$ of dg categories. Because $f$ was an MC equivalence, it follows that $H^0(f_*)$ is a bijection on isoclasses. In particular, $H^0(f_*)$ must be essentially surjective, and so $f_*$ is quasi-essentially surjective. By \ref{MCisDII}, the natural map $f_*:\Dcoderivedc(X) \to \Dcoderivedc(X')$ is a triangle equivalence, which by taking compact objects restricts to a triangle equivalence $f_*:\Perfcoderivedc(X) \to \Perfcoderivedc(X')$. By \ref{pretrqff}, the natural dg functor $f_*:\Perfcoderiveddgc(X) \to \Perfcoderiveddgc(X')$ is hence a quasi-equivalence, and in particular quasi-fully faithful. Hence its restriction $f_*:\MCdg(X) \to \MCdg(X')$ is quasi-fully faithful, and the claim follows.
 	\end{proof}

\begin{lem}\label{mcequivtensor}
	Let $C$ and $C'$ be curved coalgebras and $A$ and $A'$ be curved algebras. \begin{enumerate}
		\item If $A\to A'$ is an MC equivalence then the induced map $$\Hom(C,A) \to \Hom(C,A')$$ is an MC equivalence of algebras, for any $C$.
	\item If $C\to C'$ is an MC equivalence then the induced map $$\Hom(C',A) \to \Hom(C,A)$$ is an MC equivalence of algebras, for any $A$.
	\end{enumerate}
	\end{lem}
\begin{proof}
	We begin with (1). Letting $D$ be a test curved coalgebra, we need to show that the induced map $\MCmod(D,\Hom(C,A)) \to \MCmod(D,\Hom(C,A')) $ is a bijection. But $\Hom(D,\Hom(C,A))$ is functorially isomorphic to $\Hom(D\otimes C, A)$ by the hom-tensor adjunction for convolution algebras, and hence $\MCmod(D,\Hom(C,A))$ is functorially isomorphic to $\MCmod(D\otimes C, A)$. So the claim is equivalent to the statement that the induced map $\MCmod(D\otimes C,A) \to \MCmod(D\otimes C,A') $ is a bijection, which holds because $A \to A'$ is an MC equivalence. The proof of (2) is very similar: because there is a natural isomorphism $D\otimes C \cong C\otimes D$ of coalgebras, there is a natural isomorphism $\Hom(D,\Hom(C,A))\cong \Hom(C,\Hom(D,A))$ of algebras (it is perhaps easier to see this when one thinks of $\Hom(C,A)$ as the completed tensor product $C^*\hat\otimes A$). In particular, there is a functorial isomorphism $\MCmod(D,\Hom(C,A))\cong \MCmod(C,\Hom(D,A))$ and the claim follows as before.
	\end{proof}

\begin{prop}\label{mccat2}
	Let $C\to C'$ be a morphism of curved coalgebras and let $A\to A'$ be a morphism of curved algebras.
	\begin{enumerate}
		\item The following are equivalent: 
		\begin{enumerate}
			\item $C \to C'$ is an MC equivalence.
			\item For all curved algebras $E$, the induced map $\MCdg(C',E) \to \MCdg(C,E)$ is a quasi-equivalence.
		\end{enumerate}
			\item The following are equivalent: 
			\begin{enumerate}
				\item $A \to A'$ is an MC equivalence.
				\item For all curved coalgebras $D$, the induced map $\MCdg(D,A) \to \MCdg(D,A')$ is a quasi-equivalence.
			\end{enumerate}
	\end{enumerate}
	\end{prop}
\begin{proof}
	We begin with (1). Assume that (a) holds and pick a test curved algebra $E$. By \ref{mcequivtensor}, the induced map $g:\Hom(C',E) \to \Hom(C,E)$ is an MC equivalence of curved algebras. We conclude that (b) holds by applying \ref{mccat1} to $g$. It is easy to see that (b) implies (a), by taking isoclasses in the homotopy category. The proof of (2) is similar.
	\end{proof}

\begin{cor}\label{mccogtensor}
	Let $f:C\to C'$ be an MC equivalence of curved coalgebras and let $X$ be any curved coalgebra. Then $f\otimes X:C\otimes X \to C' \otimes X$ is an MC equivalence. 
	\end{cor}
\begin{proof}
	Let $A$ be a test curved algebra. We have an isomorphism $\Hom(C\otimes X,A)\cong \Hom(C,\Hom(X,A))$ which identifies the map $\MCdg(f\otimes X,A)$ with $\MCdg(f,\Hom(X,A))$. By hypothesis this latter map is a quasi-equivalence and hence $f\otimes X$ is an MC equivalence, as desired.
	\end{proof}
The class of MC equivalences sits in between the {\normalfont II}-Morita equivalences and the good {\normalfont II}-Morita equivalences:
    \begin{prop}
        Let $f: A \to B$ be a morphism of curved algebras. If $f$ is an MC equivalence then it is a {\normalfont II}-Morita equivalence. If $f$ is part of a good {\normalfont II}-Morita equivalence then it is an MC equivalence.\footnote{In an earlier version of this paper, we erroneously claimed that a morphism is an MC equivalence precisely when it is a {\normalfont II}-Morita equivalence. We would like to thank Patrick Antweiler, Julian Holstein, and Kristoffer Rasmussen for pointing out this mistake.}
    \end{prop}
    \begin{proof}
       Combine \ref{MCisDII} with \ref{lem:MCequiv}.
    \end{proof}

  The following lemma is useful.
\begin{lem}\label{inftyhtpyequiv}
	Let $C$ be a curved coalgebra and $A$ be a curved algebra. Let $3\leq n \leq \infty$. Then: \begin{enumerate}
		\item Elementary $n$-homotopy is an equivalence relation on $\Hom(\Omega C,A)$.
		\item Elementary $n$-homotopy is an equivalence relation on $\Hom(C,\check BA)$.
	\end{enumerate}
\end{lem}
\begin{proof}
	We prove (1); the proof of (2) is similar. By \ref{prebcb}(1), two morphisms $\Omega C \to A$ are $n$-homotopic if and only if their associated MC elements are $n$-homotopic. But if two MC elements are homotopic then their corresponding morphisms are elementary $n$-homotopic.
\end{proof}

\begin{lem}\label{endpointslem}For $3\leq n \leq \infty$, the endpoint inclusions $i_0,i_1:C \into C\otimes I_n$ are MC equivalences.
\end{lem}
\begin{proof}
	Observe that as in the proof of \ref{prebcb} we have an isomorphism $\MC\Hom(C\otimes I_n,A)\cong \MC(\Hom(C,A)\hat\otimes I^n)$. Denoting the convolution algebra by $E$, we hence wish to show that $\MC(E\hat\otimes I^n) \to \MC(E)$ is an isomorphism. Passing to the associated dg categories, observe that $\MCdg(E\hat\otimes I^n)$ is the category whose objects are pairs $x,y$ of MC elements from $E$ together with an $n$-homotopy $h:x\simeq y$ between them. In particular, the isoclasses of $H^0\MCdg(E\hat\otimes I^n)$ are in bijection with the isoclasses of $H^0\MCdg(E)$, as required. We remark that if $n=\infty$ then $I_n$ is $\infty$-contractible by \ref{prop:homotopycoalg}(6), and hence the $i_j$ are in fact $\infty$-homotopy equivalences.
\end{proof}

\subsection{Maurer--Cartan equivalences for dg algebras and coalgebras} 
We show that if a morphism of dg algebras is an MC equivalence, then it is a quasi-isomorphism. This does not necessarily imply that two MC equivalent dg algebras are quasi-isomorphic, since they may be MC equivalent via curved morphisms. Firstly, since a dg algebra admits a canonical MC element $0$, we can deduce the following:
\begin{lem}\label{mcisqipre}
	Suppose that $f:X \to X'$ is a morphism of dg algebras which induces a quasi-fully faithful dg functor $f_*:\MCdg(X) \to \MCdg(X')$. Then $f$ is a quasi-isomorphism.
	\end{lem}
\begin{proof}
	For clarity let $\zeta \in X$ denote the MC element $0\in X^1$. Similarly let $\zeta'\in X'$ denote the MC element $0\in X'^1$. Clearly $f_*$ takes $\zeta$ to $\zeta'$, and by hypothesis the induced morphism of endomorphism dg algebras $$\MCdg(X)(\zeta,\zeta) \to \MCdg(X')(\zeta',\zeta')$$is a quasi-isomorphism. But by definition $\MCdg(X)(\zeta,\zeta)$ is simply the dg algebra $X$, and the induced morphism on endomorphism dg algebras is simply $f_*$, so we are done.
	\end{proof}

\begin{prop}\label{mcisqi}\hfill
		\begin{enumerate}
		\item If a morphism of dg algebras is an MC equivalence, then it is a quasi-isomorphism.
			\item If a morphism of dg coalgebras is an MC equivalence, then it is a quasi-isomorphism.
	\end{enumerate}
	\end{prop}
\begin{proof}
	We begin with (1). Suppose $A \to A'$ is an MC equivalence. By \ref{mccat2}, for all curved coalgebras $C$, the induced map $\MCdg(\Hom(C,A))\to \MCdg(\Hom(C,A'))$ is a quasi-equivalence. In particular, taking $C$ to be the dg coalgebra $\ground$, we have $\Hom(\ground, A) \cong A$ as dg algebras. Hence $\MCdg(A)\to \MCdg(A')$ is a quasi-equivalence, and the claim follows by an application of \ref{mcisqipre}. The proof of (2) is similar: we can conclude that an MC equivalence $g:C \to C'$ of dg coalgebras induces a quasi-isomorphism $g^*:C'^* \to C^*$; since the linear dual is exact it follows that $g$ must also be a quasi-isomorphism.
	\end{proof}

Recall from \cite[\S9.3]{Positselski11} the existence of a model category structure on  $\mathbf{cuCog}^\mathrm{conil}$, the category of conilpotent curved coalgebras, with the following properties. The cofibrations are the injections, and the weak equivalences are the morphisms $f$ such that $\Omega f$ is a quasi-isomorphism of dg algebras. The bar-cobar adjunction $\Omega:\mathbf{cuCog}^\mathrm{conil}\leftrightarrow \mathbf{Alg}_\mathrm{q.i.}:B$ is a Quillen equivalence. Recall that a dg algebra is cofibrant in the usual model structure if it is the cobar construction on a conilpotent curved coalgebra.
\begin{cor}\label{mcconil}\hfill
	\begin{enumerate}
		\item A map $f:A \to A'$ between cofibrant dg algebras is an MC equivalence if and only if it is a quasi-isomorphism.
		\item A map $g:C \to C'$ between conilpotent curved coalgebras is an MC equivalence if and only if it is a weak equivalence.
		\end{enumerate}
	\end{cor}
\begin{proof}
	Beginning with (1), the forwards direction is \ref{mcisqi}(1). For the backwards direction, just observe that two quasi-isomorphic cofibrant algebras are necessarily $n$-homotopy equivalent for all $n$, in particular $3$-homotopy equivalent, and in particular MC equivalent. For (2), note that $g$ is an MC equivalence if and only if $\Omega g$ is, by \ref{prop:MCcharacterization1}. But by (1), $\Omega g$ is an MC equivalence if and only if it is a quasi-isomorphism.
	\end{proof}

\section{Strong cofibrations}\label{section:strongcof}

In this section we define the notion of strong cofibration of coalgebras, which is a morphism which induces a fibration on $\MCdg(-,A)$ for all $A$. These will be our model-theoretic cofibrations for the MC model structure. In the main theorem of this section, we show that every injection of curved coalgebras is a strong cofibration (and vice versa). We will first prove that strong cofibrations are saturated, which will allow us to reduce to the finite dimensional case. The desired result will then follow from the structure theory developed in Section \ref{section:structure}.

Let $A$ be a curved algebra. We regard $\MCdg(-,A)$ as a contravariant functor from curved coalgebras to dg categories. 

\begin{defi}
	Let $f:C \to C'$ be a map of curved coalgebras. Say that $f$ is a {strong cofibration} if $f^*:\MCdg(C',A)\to \MCdg(C,A)$ is a fibration for all curved algebras $A$. Say that $f$ is an {acyclic strong cofibration} if $f^*$ is an acyclic fibration for all $A$. 
	\end{defi}

\begin{rem}\label{scdefrem}
 If $A=\varnothing$ is the initial object of $\calgp$, then $\MCdg(C,A)$ is the empty dg category unless $C=0$, in which case $\MCdg(0,\varnothing)$ is the zero dg category. In particular, it is not hard to check that all maps $f:C \to C'$ of curved coalgebras induce a fibration $f^*:\MCdg(C',\varnothing)\to \MCdg(C,\varnothing)$ of dg categories; it is even acyclic unless $C=0$ and $C'\neq 0$ (such a map is never an MC equivalence). Hence there is no difference in the above definition whether one chooses to test against objects of $\calg$ or of the slightly larger category $\calgp$. Dually, if $C=*$ is the final object of $\ccogp$, then every map of curved algebras $g:A \to A'$ induces a fibration $g_*:\MCdg(C,A)\to \MCdg(C,A')$, which is acyclic unless $A\neq 0$ and $A'=0$. 
	\end{rem}
Note that the dg category $\MCdg(C, \Omega C)$ is nonempty, since it contains the MC element of $\Hom(C,\Omega C)$ corresponding to the unit $C \to\check B \Omega C$ of the bar-cobar adjunction. In particular, if $f:C \to C'$ is any map of curved coalgebras, there is at least one curved algebra $A$ (namely $\check B \Omega C'$) such that $f^*:\MCdg(C',A)\to \MCdg(C,A)$ is a dg functor between nonempty dg categories.

\begin{prop}\label{scofibdict}
	Let $f:C \to C'$ be a a morphism of curved coalgebras.
	\begin{enumerate}
		\item The following are equivalent:
		\begin{enumerate}
			\item $f$ is a strong cofibration.
			\item $f$ is an injection, and every homotopy commutative diagram $$\begin{tikzcd} C \ar[r,"g"]\ar[d,"f"]& \check BA \\
				C'\ar[ur,"h",swap] & \end{tikzcd}$$ rectifies to a genuinely commutative diagram $$\begin{tikzcd} C \ar[r,"g"]\ar[d,"f"]& \check BA \\
				C'\ar[ur,"h' ",swap] & \end{tikzcd}$$ with $h$ homotopy equivalent to $h'$.
		\end{enumerate}
		\item The following are equivalent:
	\begin{enumerate}
		\item $f$ is an acyclic strong cofibration.
		\item $f$ is a strong cofibration and an MC equivalence.
		\item $f$ is an injection, and every solid diagram $$\begin{tikzcd} C \ar[r]\ar[d]& \check BA \\
			C'\ar[ur, dashed] & \end{tikzcd}$$admits an extension, unique up to homotopy.
		\item $f$ is an injective MC equivalence, and every solid diagram $$\begin{tikzcd} C \ar[r]\ar[d]& \check BA \\
			C'\ar[ur, dashed] & \end{tikzcd}$$admits an extension.
	\end{enumerate}
	\end{enumerate}
\end{prop}
\begin{proof}
We begin with (1). Recall that a map $F$ of dg categories is a fibration if and only if it is both surjective on hom-complexes and $H^0F$ is an isofibration. We first show that $f$ is an injection if and only if $\MCdg(f,A)$ is surjective on hom-complexes for all curved algebras $A$. If $f$ is an injection then $\Hom(f,A)$ is a surjection, and since the hom-complexes in $\MCdg(C,A)$ are all twists of $\Hom(C,A)$, the forwards implication holds. To see the backwards implication, take $A=\Omega C'$. We know that $\MCdg(C',A)$ is a nonempty dg category, so picking an object of it we get a morphism on hom-complexes which - upon restricting to the underlying graded vector spaces -  is a surjection $\Hom(C',A) \to \Hom(C,A)$. This can only be a surjection if $f$ was an injection. Secondly, one can see that $H^0\MCdg(f,A)$ is an isofibration if and only if the rectification property holds: this follows by unwinding the definition of what it means for $H^0\MCdg(f,A)$ to be an isofibration. So we are done.

For (2), to see that (a) is equivalent to (b) just use \ref{mccat2}: $f$ is an MC equivalence if and only if $\MCdg(f,A)$ is a quasi-equivalence for all $A$. To see that (a) is equivalent to (c), note that a map $F$ of dg categories is both a quasi-equivalence and an isofibration on $H^0$ if and only if $H^0(F)$ is an equivalence that is surjective on objects. So as before it suffices to show that $H^0\MCdg(f,A)$ is both an equivalence and surjective on objects if and only if the unique extension property holds. The surjectivity part corresponds to the existence of the extension and the equivalence part corresponds to the homotopy uniqueness, by \ref{prop:MCcharacterization1}(2). The equivalence of (c) and (d) is now clear.
\end{proof}

Note that by \ref{prop:MCcharacterization1}(4) and \ref{cor:Barcobar}(1), a morphism $f$ of curved coalgebras is an MC equivalence if and only if $\Omega f$ is an $\infty$-homotopy equivalence. 

\begin{lem}\label{scretract}
	Let $f:C \to C'$ be an acyclic strong cofibration.
	\begin{enumerate}
		\item $C\to \check B\Omega C$ factors through $f$.
		\item $\Omega f$ admits a retract which is an $\infty$-homotopy inverse.

		\end{enumerate}
	\end{lem}
\begin{proof}
	For (1), consider the commutative diagram  $$\begin{tikzcd} C \ar[r]\ar[d]& \check B\Omega C \\
		C'\ar[ur, dashed] & \end{tikzcd}$$which admits a lift because $f$ was an acyclic strong cofibration. For (2), note that the lift in the above diagram is adjunct to a lift in the diagram $$\begin{tikzcd} \Omega C \ar[r,"\id"]\ar[d]& \Omega C \\
		\Omega C'\ar[ur, dashed,"g", swap] & \end{tikzcd}$$which provides the desired retract $g$. Since $f$ was an MC equivalence, $\Omega f$ is necessarily an $\infty$-homotopy equivalence. Since $\Omega f$ is an isomorphism in the homotopy category $g$ must also be an isomorphism in the homotopy category; i.e.\ an $\infty$-homotopy inverse of $\Omega f$.
	\end{proof}

Say that a class of morphisms in a category is saturated if it is closed under pushouts, retracts, and transfinite composition. In a model category, the classes of cofibrations and acyclic cofibrations are saturated, since they are defined via lifting properties. Dually, say that a class of morphisms in a category is cosaturated if it is closed under pullbacks, retracts, and transfinite cocomposition (i.e.~ transfinite composition in the opposite category). In a model category, the classes of fibrations and acyclic fibrations are cosaturated.

	\begin{prop}\label{scsatd}
	Strong cofibrations and acyclic strong cofibrations are saturated.
	\end{prop}
\begin{proof}
	If $J$ is a pushout diagram of curved coalgebras then for every curved algebra $A$, the diagram $\MCdg(J,A)$ is a pullback diagram in $\dgcat'$ by \ref{mclimscor}. Fibrations in any model category are closed under pullbacks, since they are characterised by a right lifting property. It now follows that strong cofibrations are closed under pushout. Closure under transfinite composition follows from the fact that fibrations are closed under transfinite cocomposition, and closure under retracts follows from the closure of fibrations under retracts. Hence the class of strong cofibrations is saturated. The analogous result follows for acyclic strong cofibrations, since the class of acyclic fibrations in any model category is similarly cosaturated.
	\end{proof}
    
    Now we can show, using the structure theory results developed in Section \ref{section:structure} together with the MC element lifting results of Section \ref{section:mclifting}, that the class of strong cofibrations coincides with the class of injections:

\begin{theorem}\label{thm:strongcof}
	The injections of curved coalgebras are precisely the strong cofibrations.
\end{theorem}
\begin{proof}
A strong cofibration is an injection by \ref{scofibdict}, so we just need to prove that the converse holds. By \ref{CogStructThm}, an injection of curved coalgebras is a relative cell complex for maps of the following form:
\begin{enumerate}
    \item Cosquare zero extensions of finite dimensional curved coalgebras.
    \item Injections between finite dimensional curved cosemisimple coalgebras.
\end{enumerate}
Since strong cofibrations are saturated by \ref{scsatd}, it suffices to prove that both of these classes of maps are strong cofibrations. For case (1), let $C' \to C$ be a cosquare zero extension of finite dimensional curved coalgebras and let $X$ be a curved algebra. Since $C$ is finite dimensional we have $\Hom(C,X)\cong C^*\otimes X$. Moreover, $C^*\otimes X \to C'^*\otimes X$ is a square zero extension of curved algebras. So it will suffice to prove that a square zero extension $\pi:A \twoheadrightarrow B$ of curved algebras induces a fibration on $\MCdg$. But this is precisely \ref{squarezero1}. For case (2), observe that all such maps have retracts by \ref{CSSsection}, and are hence strong cofibrations.
\end{proof}

\section{Strong fibrations}\label{section:strongfib}
Dual to the notion of strong cofibration of coalgebras is the notion of strong fibration of algebras, which we develop in this section. Although a strong fibration is a surjection, we will see that the converse is not true. We obtain a characterisation of (acyclic) strong fibrations in terms of a lifting property against $\Omega$ of (acyclic) injections. We use these results to show that if $\mathcal{K}$ denotes the class of injective MC equivalences, every $\Omega(\mathcal{K})$-relative cell complex is an $\infty$-homotopy equivalence. This will be a key result in our construction of the MC model structures.

\begin{defi}
	Let $f$ be a morphism of curved algebras. Say that $f$ is a {strong fibration} if for all curved coalgebras $C$, the induced map $f_*:\MCdg(C,A) \to \MCdg(C,A')$ is a fibration of dg categories. Say that $f$ is an {acyclic strong fibration} if for all curved coalgebras $C$, the induced map $f_*:\MCdg(C,A) \to \MCdg(C,A')$ is an acyclic fibration of dg categories.
\end{defi}

As in \ref{scdefrem}, in the above definition it does not matter if we choose to test against $C\in \ccog$ or $C\in \ccogp$. We may also extend this definition to $\calgp$, and it is not hard to see that $\varnothing \to A$ is a strong fibration for any curved algebra $A$. Taking $C=\ground$ we see that if $A \to A'$ is a strong fibration then $\MCdg(A) \to \MCdg(A')$ is a fibration of dg categories.

\begin{lem}\label{sfchar}
	Let $f:A \to A'$ be a morphism of curved algebras. Then $f$ is a strong fibration if and only if it is a surjection, and every homotopy commutative diagram $$\begin{tikzcd} & A\ar[d,"f"] \\
		\Omega C\ar[ur,"h'"] \ar[r,"g"] & A'\end{tikzcd}$$ rectifies to a genuinely commutative diagram $$\begin{tikzcd} & A\ar[d,"f"] \\
		\Omega C\ar[ur,"h"] \ar[r,"g"] & A'\end{tikzcd}$$ with $h$ homotopy equivalent to $h'$.
\end{lem}
\begin{proof}
	Dual to the proof of \ref{scofibdict}(1).
\end{proof}
\begin{cor}\label{psfib}
	Let $A$ be a curved algebra and $3\leq n \leq \infty$. Let $\pi_0,\pi_1:A\hat\otimes I^n \to A$ be the endpoint projections. Then the map $\pi\coloneqq\pi_0\times \pi_1:A\hat\otimes I^n \to A \times A$ is a strong fibration.
\end{cor}
\begin{proof}
	We use \ref{sfchar}. The map $\pi$ is clearly a surjection, so we need to verify the rectification property. Suppose given a homotopy commutative diagram $$\begin{tikzcd} & A\hat\otimes I^n\ar[d,"\pi"] \\
		\Omega C\ar[ur,"h"] \ar[r,"g"] & A\times A\end{tikzcd}$$and let $H: h\pi \to g$ be an elementary $n$-homotopy witnessing this commutativity, which exists by \ref{inftyhtpyequiv}. Let $E\coloneqq \Hom(C,A)$ be the convolution algebra. From the above data we may recover\begin{itemize}
	\item A pair of MC elements $g_0,g_1 \in E$.
	\item An $n$-homotopy $h: g_0' \to g_1'$ between two MC elements of $E$.
	\item A pair of $n$-homotopies $H_i:g'_i \to g_i$.
\end{itemize}
By composition we hence obtain an $n$-homotopy $h':g_1 \to g_2$ between MC elements, which corresponds to the desired map rectifying the above commutative diagram. We remark that $\pi$ need not be an MC equivalence.
\end{proof}

Every strong fibration is a surjection, but the converse is false.

\begin{example}\label{sfcounter}
	Let $A$ be the graded algebra $$A\coloneqq \frac{\ground[x,y]\langle f\rangle}{(x^2,y^2,xy, fx-yf)}$$
	where $x,y$ are placed in cohomological degree $1$ and $f$ is placed in degree $0$. We regard $A$ as a curved algebra with zero differential and zero curvature. Let $B$ be the quotient $B\coloneqq A/(f^2-1)$, so that we have an obvious surjection $\pi:A \twoheadrightarrow B$. We claim that $\pi$ is not a strong fibration. It is clear that $x,y$ are two MC elements in $B$ and that $f$ is a gauge between them. Moreover $x\in A$ is an MC element that lifts $x\in B$. An element $\tilde y \in A$ is a lift of $y$ if and only if it is of the form $$\tilde y = \sum_{m,n\in \mathbb{N}}\lambda_{mn}f^{2m}yf^{2n} \quad\text{such that}\quad \sum_{m,n\in \mathbb{N}}\lambda_{mn}=1 $$where only finitely many of the $\lambda_{mn}$ are nonzero. But the only such lift of $y$ which is an MC element of $A$ is $y$ itself. Since $A$ has no elements in cohomological degree $-1$, its MC elements $x,y$ are gauge homotopy equivalent if and only if there is some invertible element $g \in A^0$ which is a gauge between them. But $A^0\cong k[f]$ has $k^\times$ as its group of units. Since $x$ and $y$ are not scalar multiples of each other, they cannot be gauge homotopy equivalent, and hence $\pi$ is not a strong fibration. 
\end{example}

\begin{lem}\label{magsqdg}
Let $$\begin{tikzcd}
X \ar[r,"u"]\ar[d,"v"] & Y \ar[d,"v'"]\\
Z \ar[r,"u'"] & W
\end{tikzcd}$$be a commutative square of curved algebras and let $P\coloneqq Y\times_W Z$ be the pullback, which admits a natural map $\psi: X \to P$. Note that $P\neq \varnothing$ since it admits a map from $X\neq \varnothing$. Suppose that:
\begin{enumerate}
	\item $\MCdg(v)$ and $ \MCdg(v')$ are fibrations.
	\item $\MCdg(u)$ and $\MCdg(u')$ are acyclic fibrations.
	\item The natural map $\psi: X \to P$ is a surjection.
	\end{enumerate}
Then the map $\MCdg(\psi): \MCdg(X)\to \MCdg(P)$ is an acyclic fibration.
	\end{lem}
\begin{proof}
Write $\Psi\coloneqq \MCdg(\psi)$. If $X=0$ then $P=0$ and $\Psi$ is an isomorphism. If $Z=0$ then $P=Y$ and $\Psi=\MCdg(u)$, which is an acyclic fibration by assumption. By \ref{mcdgadj}, the diagram $$\begin{tikzcd}
		\MCdg(P) \ar[r]\ar[d]& \MCdg(Y)\ar[d, two heads]
		\\ \MCdg(Z) \ar[r, "\simeq"]& \MCdg(W)
	\end{tikzcd}$$is a pullback diagram in $\dgcat'$. The right-hand vertical map is a fibration by (1) and so this square is a homotopy pullback square in $\dgcat'$. The lower horizontal map is a quasi-equivalence by (2) and hence $	\MCdg(P) \to \MCdg(Y)$ is a quasi-equivalence. Hence $\MCdg(u)$ factors as $\Psi$ followed by a quasi-equivalence. By (2) it now follows that $\Psi$ is itself a quasi-equivalence. If $Y=0$ then $P=Z$ and $\Psi=\MCdg(v)$, which by assumption is a fibration and so we are done. So we may assume that $X,Y,Z$ are all nonzero, in which case the above diagram is a pullback diagram in $\dgcat$. By assumption $\Psi$ is surjective on hom-complexes, so to show that it is an acyclic fibration it will suffice to show that it is surjective on objects (since an equivalence of ordinary categories is an isofibration if and only if it is surjective on objects). To do this we will examine the fibres of $\Psi$; this is why we need to know that the above pullback diagram is actually a pullback diagram in $\dgcat$ and not $\dgcat'$ (where pullbacks are computed differently if some of the terms are $0$). Consider the commutative diagram of dg categories $$\begin{tikzcd}
	\MCdg(X) \ar[r,"\Psi"]\ar[d, "F", two heads]& \MCdg(P)\ar[r,"\simeq"] \ar[d, "\Phi", two heads]& \MCdg(Y)\ar[d, "F'", two heads] \\
	\MCdg(Z) \ar[r,equals]& \MCdg(Z) \ar[r,"G"]& \MCdg(W).
\end{tikzcd}$$Pick $z \in \MCdg(Z)$ and put $z'\coloneqq Gz$. We get induced morphisms between fibres $$F^{-1}(z) \to \Phi^{-1}(z) \to F'^{-1}(z')$$(note that these fibres also compute the homotopy fibres). Because the right-hand square is a pullback diagram, the map $\Phi^{-1}(z) \to F'^{-1}(z')$ is an isomorphism. Because the map $F^{-1}(z) \to F'^{-1}(z')$ is an acyclic fibration by (2), the map $F^{-1}(z) \to \Phi^{-1}(z) $ must be an acyclic fibration and in particular surjective on objects. In particular, take an element $p$ of $\MCdg(P)$ and put $z \coloneqq \Phi (p)$. Regarding $p$ as an element of $\Phi^{-1}(z) $ we see that there exists $x \in F^{-1}(z) $ with $\Psi (x) = p$. Hence $\Psi$ must be surjective on objects, as required.
	\end{proof}

\begin{cor}\label{magsqcor}
	Let $f:C \into C'$ be an injection of curved coalgebras and $g:A \to A'$ be a strong fibration of curved algebras. Suppose that at least one of $f$ or $g$ is an MC equivalence. Put $P\coloneqq \Hom(C',A') \times_{\Hom(C,A')} \Hom(C,A)$. Then the natural map $\MCdg(C',A) \to \MCdg(P)$ is an acyclic fibration.
	\end{cor}
\begin{proof}
	We obtain a commutative square $$\begin{tikzcd}
		\Hom(C',A) \ar[r]\ar[d]& \Hom(C,A)\ar[d]
		\\ \Hom(C',A') \ar[r]& \Hom(C,A')
	\end{tikzcd}$$ of convolution algebras. Note that in the corresponding square of MC dg categories, all maps are fibrations. By the assumption on $f$ and $g$, either both horizontal maps or both vertical maps are acyclic fibrations. Hence, applying \ref{magsqdg}, we only need to check that the natural map $\psi:\Hom(C',A)\to P$ is a surjection. To do this, let $K$ be the kernel of $\Hom(C',A)\to\Hom(C,A)$ and $K'$ be the kernel of $P \to \Hom(C,A)$, so that we get the following commutative diagram with exact rows:
$$\begin{tikzcd}
0 \ar[r] \ar[d, equal]& K\ar[r]\ar[d] & \Hom(C',A)\ar[r]\ar[d,"\psi"] &\Hom(C,A)\ar[d,equal] 
\\0 \ar[r]& K'\ar[r] & P\ar[r] &\Hom(C,A)
\end{tikzcd}$$One can check that $K'$ is isomorphic to the kernel of the map $\Hom(C',A') \to\Hom(C,A')$, and hence setting $W\coloneqq \coker(C \to C')$ we have by exactness of Hom an isomorphism $K'\cong \Hom(W,A')$. Similarly by exactness we have an isomorphism $K\cong \Hom(W,A)$, and since $A \to A'$ was a surjection it follows that $K\to K'$ is a surjection. By the Four Lemma it now follows that $\psi$ is a surjection.
	\end{proof}

\begin{prop}\label{strongfibsareLfibs}\hfill
	\begin{enumerate}
		\item
	Strong fibrations of curved algebras have the right lifting property with respect to maps of the form $\Omega(f)$, where $f$ is an injective MC equivalence of curved coalgebras.
	\item Acyclic strong fibrations of curved algebras have the right lifting property with respect to maps of the form $\Omega(f)$, where $f$ is an injection of curved coalgebras.
	\end{enumerate}
\end{prop}
\begin{proof}
		Let $$\begin{tikzcd}
		\Omega(C) \ar[r,"u"]\ar[d,"f"]& A\ar[d,"g"]
		\\ \Omega(C') \ar[r,"v"]& A'
	\end{tikzcd}$$ be a commutative square with $g$ a strong fibration and $f=\Omega f'$ with $f'$ an injection. Assuming that at least one of $f$ or $g$ is acyclic, we wish to prove that a lift $\Omega C' \to A$ exists. As before put $P\coloneqq \Hom(C',A') \times_{\Hom(C,A')} \Hom(C,A)$. Note that an object of $\MCdg(P)$ is a commutative square of the form $$\begin{tikzcd}
	\Omega(C) \ar[r]\ar[d,"f"]& A\ar[d,"g"]
	\\ \Omega(C') \ar[r]& A'
\end{tikzcd}$$ where the horizontal maps are not fixed. The natural map $\Psi:\MCdg(C',A) \to \MCdg(P)$ sends a morphism $\ell:\Omega C' \to A$ to the commutative diagram $$\begin{tikzcd}
\Omega(C) \ar[r,"\ell f"]\ar[d,"f"]& A\ar[d,"g"]
\\ \Omega(C') \ar[r,"g\ell"]& A'.
\end{tikzcd}$$By \ref{magsqcor}, $\Psi$ is surjective on objects, and hence a lift $\Omega C' \to A$ exists in the original diagram.
\end{proof}

\begin{lem}\label{llem}
	Let $f$ be a map of curved coalgebras and let $g$ be a map of curved algebras. Let $C$ be a curved coalgebra. Then $g$ lifts on the right against $\Omega(f\otimes C)$ if and only if $\Hom(C,g)$ lifts on the right against $\Omega f$.
\end{lem}

\begin{proof}
This is a straightforward application of \ref{htconv}.
	\end{proof}

\begin{prop}\label{Lfibsarestrongfibs}
	Let $g$ be a morphism of curved algebras with the right lifting property with respect to maps of the form $\Omega(f)$, where $f$ is an injective MC equivalence of curved coalgebras. Then $g$ is a strong fibration.
\end{prop}
\begin{proof}
	Let $i_0:\ground \to I_3$ denote the inclusion at the $0$ vertex. Fix a curved coalgebra $C$. By	\ref{endpointslem}, every map of the form $C\otimes i_0: C \to C\otimes I_3$ is an injective MC equivalence, and in particular $g$ lifts against $\Omega(C\otimes i_0)$. By \ref{llem}, the map $\Hom(C,g)$ lifts on the right against $\Omega i_0$. It follows that the morphism of dg categories $\MCdg(C,g)$ lifts on the right against $\MCdg(\Omega i_0)$: indeed $\MCdg(\Omega\ground) \cong \ground$, so a morphism $\MCdg(\Omega\ground) \to \MCdg(C,A)$ is the same as a morphism $\Omega C\to A$. Similarly a morphism $\MCdg(\Omega (I_3))\to \MCdg(C,A')$ is the same as a pair of morphisms $\Omega C \to A'$ and a $3$-homotopy between them. It follows that any commutative diagram of the form $$\begin{tikzcd}
		\MCdg(\Omega\ground) \ar[r]\ar[d,"\MCdg(\Omega i_0)", swap]& \MCdg(C,A) \ar[d,"\MCdg(C{,}g)"]
		\\ \MCdg(\Omega I_3) \ar[r]&  \MCdg(C,A')
	\end{tikzcd}$$ is obtained by applying $\MCdg$ to a commutative diagram of the form $$\begin{tikzcd}
		\Omega\ground \ar[r]\ar[d,"\Omega i_0", swap]& \Hom(C,A) \ar[d,"\Hom(C{,}g)"]
		\\ \Omega I_3 \ar[r]&  \Hom(C,A')
	\end{tikzcd}$$and hence if $\Hom(C,g)$ lifts on the right against $\Omega i_0$ then $\MCdg(C,g)$ lifts on the right against $\MCdg(\Omega i_0)$. Recalling the categorical cobar construction $\OCat$ from \cite{HL2020}, together with the computation of $\OCat(I_3)$ from \ref{prop:coderivedinfinity}, it is not hard to check that the morphism $\MCdg(\Omega i_0)$ is actually isomorphic to $\OCat(i_0)$. Recall from \cite{Tabuada04} that a morphism of dg categories is a fibration if and only if it lifts against a certain morphism $\ground \to \mathcal{K}$, where $\mathcal{K}$ is an explicitly defined dg category with two objects. Again using the explicit computation of $\OCat(I_3)$ from \ref{prop:coderivedinfinity}, one can easily see that $\ground \to \mathcal{K}$ is isomorphic to $\OCat(i_0)$. Hence, $\MCdg(C,g)$ is a fibration, which is the desired statement.
\end{proof}

\begin{lem}\label{sfibheq}
	Let $f:X \to X'$ be a map of curved coalgebras that has the right lifting property with respect to injections. Then $f$ is a $3$-homotopy equivalence.
\end{lem}
\begin{proof}
	The left-hand vertical map in the commutative diagram $$\begin{tikzcd}
		0 \ar[r]\ar[d]& X\ar[d,"f"]
		\\ X' \ar[r]& X'
	\end{tikzcd}$$is an injection and hence we get a lift $g:X' \to X$ exhibiting $X'$ as a retract of $X$. It will suffice to show that $gf$ is $3$-homotopic to $\id_{X}$. Consider the commutative diagram $$\begin{tikzcd}
		X\sqcup X \ar[r,"u"]\ar[d,"i"]& X\ar[d,"f"]
		\\ X\otimes I_3\ar[r,"v"]& X'
	\end{tikzcd}$$where $u=\id_X \sqcup gf$ and $v$ is the constant homotopy on $f$. This diagram admits a lift $H$ because $i$ is an injection, and $H$ is the desired homotopy.	
\end{proof}

\begin{prop}\label{asfalg}
	Let $f$ be a morphism of curved algebras.
\begin{enumerate}
	\item The following are equivalent:	\begin{enumerate}
		\item $f$ is a strong fibration.
		\item $f$ has the right lifting property with respect to morphisms of the form $\Omega g$, where $g$ is an injective MC equivalence of curved coalgebras.
	\end{enumerate}
	\item The following are equivalent:	\begin{enumerate}
	\item $f$ is an acyclic strong fibration.
\item $f$ has the right lifting property with respect to morphisms of the form $\Omega g$, where $g$ is an injection of curved coalgebras.
\end{enumerate}
		\end{enumerate}
	\end{prop}
\begin{proof}
We begin with claim (1). The implication (a)$\implies$(b) is \ref{strongfibsareLfibs}(1) and the implication (b)$\implies$(a) is \ref{Lfibsarestrongfibs}. For claim (2), the implication (a)$\implies$(b) is \ref{strongfibsareLfibs}(2). If (b) holds, then it follows by adjunction that $\check B f$ lifts against all injections. So by \ref{sfibheq} the map $\check Bf$ is a $3$-homotopy equivalence, and so $f$ is an MC equivalence. It is a strong fibration by claim (1), and hence an acyclic strong fibration.
	\end{proof}

We can use the previous results to obtain a strong version of \ref{scretract}. Recall from \ref{scretract}(2) that if $f:C \to C'$ is an injective MC equivalence of curved coalgebras, then $\Omega f$ admits a retract $g:\Omega C \to \Omega C'$ that is an $\infty$-homotopy inverse.
\begin{lem}\label{scahext}
	Let $f:C \to C'$ be an injective MC equivalence of curved coalgebras and $g$ a retract of $\Omega f$. Then for $3\leq n \leq \infty$, there exists an $n$-homotopy $H:(\Omega f)g\simeq \id_{\Omega C'}$ which restricts to the constant homotopy on $\id_{\Omega C}$.
\end{lem}
\begin{proof}For brevity put $A\coloneqq \Omega C$ and $A'\coloneqq \Omega C'$. Consider the commutative diagram 
	$$\begin{tikzcd} 
		A \ar[d,"\Omega f", swap]\ar[r]& A'\hat\otimes I^n \ar[d,"\pi"]\\
		A' \ar[r,"(\Omega f)g \times \id"]& A'\times A'  
	\end{tikzcd}$$where the unlabelled arrow is the composition of the constant homotopy $A \to A\hat\otimes I^n$ with the map $\Omega(f)\hat\otimes I^n$. By \ref{psfib}, the map $\pi$ is a strong fibration, and so by \ref{strongfibsareLfibs}(1) we obtain a lift $H:A' \to A'\hat\otimes I^n$ in the above diagram. Commutativity of the lower right triangle says that $H$ is a homotopy $\Omega(f)g \simeq \id_{A'}$. Commutativity of the upper left triangle says that $H$ restricts to the constant homotopy on $\id_A$, as desired.
\end{proof}

\begin{prop}\label{cellargument}
	If $\mathcal{K}$ denotes the class of injective MC equivalences of curved coalgebras, then every morphism in $\mathrm{Cell}(\Omega\mathcal{K})$ is an $\infty$-homotopy equivalence.
\end{prop}
\begin{proof}
	Before we begin the proof proper we give a quick sketch of the idea in order to orient the reader. Given a cell complex $c:X\to Y$ we induct up the (possibly transfinite) tower of morphisms $X=X_0 \to  X_1 \to X_2 \to \cdots \to Y$ to show that each morphism in the tower is an $\infty$-homotopy equivalence. The induction step is not terribly hard at successor ordinals, but limit ordinals are more difficult: if $\alpha$ is a limit ordinal, then in order to assemble a collection of homotopies $X_\beta \to X_\beta \hat\otimes I^\infty \to X_\alpha \hat\otimes I^\infty$ into a homotopy $X_\alpha \to X_\alpha \hat\otimes I^\infty$ we need to know that the individual homotopies are compatible with the transition maps in the cell complex. For this reason, we will carefully construct each homotopy inductively rather than choosing an arbitrary one at each level. In more detail, at the successor ordinal step we will glue together two known homotopies: one will be obtained from the induction hypothesis and one will be obtained from \ref{scahext}. In order to glue them we will need a compatibility condition, which we will also assume in the induction hypothesis.
	
	Now let us start the proof in earnest. Let $c:X_0 \to \varinjlim_{\alpha \in  \lambda} X_\alpha$ be an $\Omega\mathcal{K}$-relative $\lambda$-cell complex, with transition maps $u_\alpha:X_\alpha \to X_{\alpha+1}$. We will denote by $c_\alpha:X_0 \to X_\alpha$ the natural composition map. It is enough to check that each $c_\alpha$ is an $\infty$-homotopy equivalence, since it will then follow that $c$ is an $\infty$-homotopy equivalence.
	
	Note that if $f_i:C_i \to C_i'$ is a collection of morphisms of coalgebras indexed by a set $I$, then $\sqcup_i \Omega f_i \cong \Omega(\sqcup_i f_i)$ because $\Omega$ is a left adjoint. Moreover, if each $f_i$ is an injective MC equivalence then so is $\sqcup_i f_i$, by \ref{scsatd}(2). Hence for each $\alpha \in \lambda$ we have a pushout diagram  $$\begin{tikzcd}
		X_\alpha \ar[r,"u_\alpha"]& X_{\alpha+1}
		\\ \Omega C_\alpha\ar[u]\ar[r,"\Omega f_\alpha"]& \Omega C'_\alpha\ar[u]
	\end{tikzcd}$$where each $f_\alpha$ is an injective MC equivalence of coalgebras. For brevity we will write $E_\alpha\coloneqq \Omega C_\alpha$, and similarly for $E_\alpha'$. Fix an integer $m\geq 3$. By \ref{scretract}, each $\Omega f_\alpha$ admits a retract $g_\alpha$ which is an $m$-homotopy inverse. 	Taking the pushout of $g_\alpha$ yields a retract $v_\alpha$ of $u_\alpha$. 
 
 We define, by transfinite induction, morphisms $w_\alpha:X_\alpha \to X_0$. First set $w_0\coloneqq\id_{X_0}$. For successor ordinals, set  $w_{\alpha+1}\coloneqq w_\alpha \circ v_\alpha$. Finally if $\alpha = \varinjlim_{\beta\in \alpha}\beta$ is a limit ordinal, observe that each of the $w_\beta$ for $\beta\in \alpha$ assemble to give a morphism $X_\alpha =\varinjlim_{\beta \in \alpha} X_\beta \to X_0$. It is not hard to see that $w_\alpha$ is a retract of the transfinite composition map $c_\alpha$. We will now prove the following statement by transfinite induction on $\alpha$:
 
 \textbf{For each $\alpha \in \lambda$, there is a homotopy $H_\alpha: X_\alpha \to X_\alpha \otimes I^m$ from $c_\alpha w_\alpha$ to $\id_{X_\alpha}$ such that for each $\beta \in \alpha$, the restriction of $H_\alpha$ to $E_\beta$ is the constant homotopy on $\id_{E_\beta}$. }
 
To start the induction we take $H_0$ to be the constant homotopy on $\id_{X_0}$. 
	
	Suppose that $\alpha+1$ is a successor ordinal and that we have a homotopy $H_\alpha$ satisfying the desired properties. By \ref{scahext}, we have an $m$-homotopy $H'_{\alpha}: E_{\alpha}' \to E_{\alpha}' \otimes I^m$ from $\Omega(f_{\alpha})g_\alpha$ to $\id_{E_{\alpha}'}$, and moreover we may assume that $H'_{\alpha}$ restricts to the constant homotopy on $\id_{E_{\alpha}}$. Recall that $X_{\alpha+1}$ is the pushout $X_{\alpha}\sqcup_{E_{\alpha}}E'_{\alpha}$. Let $h_1: X_{\alpha} \to X_{\alpha+1} \otimes I^m$ be the composition of the homotopy $H_{\alpha}$ with the inclusion $X_{\alpha} \otimes I^m \to X_{\alpha+1} \otimes I^m$. Let $h_2:E'_{\alpha} \to  X_{\alpha+1} \otimes I^m$ be the composition of the homotopy $H'_{\alpha}: E_{\alpha}' \to E_{\alpha}' \otimes I^m$ with the map $E_{\alpha}' \otimes I^m \to X_{\alpha+1} \otimes I^m$. By construction, $h_1$ and $h_2$ agree on $E_{\alpha}$ so they give a morphism $H_{\alpha+1}:X_{\alpha+1} \to X_{\alpha+1}\otimes I^m$, which is a homotopy from $c_{\alpha}w_{\alpha}$ to $\id_{X_{\alpha+1}}$. Clearly $H_{\alpha+1}$ restricts to the constant homotopy on $\id_{E_{\alpha}}$. Moreover because $H_\alpha$ restricts to the constant homotopy on $\id_{E_{\beta}}$ for all $\beta \in \alpha$, we see that $H_{\alpha+1}$ restricts to the constant homotopy on $\id_{E_{\beta}}$ for all $\beta \in \alpha+1$, as required.
	
	Suppose finally that $\alpha$ is a limit ordinal. We have by the induction hypothesis a collection of homotopies $H_{\beta}:X_{\beta} \to X_{\beta} \otimes I^m$, one for each $\beta \in\alpha$. By composition with the inclusion $X_\beta \to X_\alpha$ this yields a system $X_{\beta} \to X_\alpha \otimes I^m$. By construction, these homotopies are all compatible with the maps $u_{\alpha}$, and hence give a morphism $H_\alpha: X_\alpha \to X_\alpha \otimes I^m$, which is the desired $m$-homotopy from $c_\alpha w_\alpha$ to $\id_{X_\alpha}$. It is easy to see that $H_\alpha$ satisfies the restriction property.
	
	To lift these $m$-homotopy equivalences to an $\infty$-homotopy equivalence, let's suppose we began the induction by taking $\infty$-homotopies $H'^\infty_{\alpha}: E_{\alpha}' \to E_{\alpha}' \hat\otimes I^\infty$ from $\Omega(f_{\alpha})g_\alpha$ to $\id_{E_{\alpha}'}$ such that $H'^\infty_{\alpha}$ restricts to the constant homotopy on $\id_{E_{\alpha}}$. By composition with the projection maps $E_{\alpha}' \hat\otimes I^\infty \to E_{\alpha}' \otimes I^m $ we obtain a coherent system of $m$-homotopies $H'^m_{\alpha}:E_{\alpha}'\to E_{\alpha}' \otimes I^m $, one for each $m$. Following the above induction, for each $\alpha$ we obtain a coherent system of $m$-homotopies $H_\alpha^m: X_\alpha \to X_\alpha \hat\otimes I^m$ which assemble, by taking the inverse limit, into the desired $\infty$-homotopy $H_\alpha^\infty: X_\alpha \to X_\alpha \hat\otimes I^\infty$. Since $w_\alpha$ is a retract of $c_\alpha$, it now follows that $c_\alpha$ is an $\infty$-homotopy equivalence, as desired.
\end{proof}

\section{The MC model structures}\label{section:MCmodel}

In this section we prove the main result. We first show that the categories $\calgp$ and $\ccogp$ admit MC model structures, for which the bar-cobar adjunction is a Quillen equivalence. We then show that $\ccogp$ is a monoidal model category and that $\calgp$ is a $\ccogp$-enriched model category. We use this to exhibit some small sets of generating cofibrations for $\calgp$. Finally we compare our MC model structures with the usual model structures for conilpotent Koszul duality, and the model structures for categorical Koszul duality from \cite{HL2020}.

\subsection{Proof of the main result}

\begin{lem}\label{pushoutlem}
	Let $C \to C'$ be an injection of curved coalgebras and let $P\coloneqq C'\sqcup_C (C\otimes I_3)$ be the pushout. The natural map $\theta:P \to C'\otimes I_3$ is an injective MC equivalence.
	
\end{lem}
\begin{proof}
	First observe that $P$ is not the final object $*$, since it has a natural description as a sub-curved coalgebra of $C' \otimes I_3$, via the map $\theta$. In particular $\theta$ is injective. The map $C \to C\otimes I_3$ is an injective MC equivalence by \ref{endpointslem}, and hence so is its pushout $C' \to P$ by \ref{scsatd}. Moreover, $C' \to C'\otimes I_3$ is an MC equivalence by the same logic, so by 2-out-of-3 for MC equivalences, we conclude that $P \to C'\otimes I_3$ is an MC equivalence, as desired.
\end{proof}

\begin{rem}
	In fact, $\theta$ has a retract: picking a linear complement to $C$ inside $C'$ gives a linear retract of $\theta$, and one can check that this is a coalgebra morphism.
\end{rem}

\begin{prop}\label{maypontoish}
	Let $p:A \to A'$ be an MC equivalence of curved algebras. If $\check Bp$ has the right lifting property with respect to all injective MC equivalences between finite dimensional curved coalgebras, then it has the right lifting property with respect to all injections.
	
\end{prop}

\begin{proof}
	Since injections between finite dimensional curved coalgebras generate all injections, it is enough to lift $\check Bp$ against injections of finite dimensional curved coalgebras. Fix a commutative square $$\begin{tikzcd}
		C\ar[r,"u"]\ar[d,"g", swap]& \check BA\ar[d,"\check Bp"]
		\\ C' \ar[r, swap,"v"] &\check B A' 
	\end{tikzcd}$$with $g$ an injection and $C,C'$ finite dimensional. We wish to show that the above square admits a lift. Because $p$ is an MC equivalence, it follows that $\check Bp$ is a $3$-homotopy equivalence, so admits a homotopy inverse $\check BA' \to \check BA$. Composing this with the map $C' \to \check B A'$ yields a \textit{homotopy} commutative diagram $$\begin{tikzcd}
		C\ar[r,"u"]\ar[d,"g", swap]& \check BA\ar[d,"\check Bp"]
		\\ C' \ar[r,"v", swap] \ar[ur,"h"]&\check B A'
	\end{tikzcd}$$ of coalgebras, where by $\ref{scofibdict}(1)(b)$ we may assume without loss of generality that the upper triangle is strictly commutative; i.e.\ $hg=u$.  By \ref{inftyhtpyequiv} we may choose a $3$-homotopy $W:C' \otimes I_3 \to BA'$ from $v$ to $(\check B p )h$ witnessing the homotopy commutativity of the lower triangle. Let $P$ be the pushout of the span $C\otimes I_3 \xleftarrow{i_i} C \xrightarrow{g} C'$. The constant homotopy on $hg=u$ fits into a commutative square $$\begin{tikzcd}
		C \ar[r,"i_1"]\ar[d,"g"]& C\otimes I_3\ar[d]
		\\ C'\ar[r,"h"]& \check BA
	\end{tikzcd}$$which defines a map $P \to \check B A$, which in turn fits into a commutative square $$\begin{tikzcd}
		P \ar[r]\ar[d,"\theta"]& \check BA\ar[d,"\check Bp"]
		\\ C'\otimes I_3\ar[r,"W"]& \check BA'
	\end{tikzcd}$$where $\theta$ is the natural map. By \ref{pushoutlem}, the map $\theta$ is an injective MC equivalence, and clearly both $P$ and $C'\otimes I_3$ are finite dimensional. Hence the above square admits a lift $H:C'\otimes I_3 \to \check BA$. One can check that $H\circ i_0$ is the desired lift in the original diagram.
\end{proof}
\begin{rem}
	In the above proof, it is important to choose a lift that extends both $g$ and the constant homotopy on $u$ by using the pushout. One can use that $C'\to C'\otimes I_3$ is an acyclic very strong cofibration to lift $W$ to a map $W': C'\otimes I_3 \to \check B A$ under $C'$, but neither end of this cylinder will produce a lift in the original diagram.
\end{rem}

We next recall some facts about lifting properties. Let $\mathcal{K}$ be any class of morphisms in a category. We denote by $\mathcal{K}^\boxslash$ the set of morphisms that lift on the right against $\mathcal{K}$, and similarly ${}^\boxslash\mathcal{K}$ those which lift on the left. The {orthogonal closure} of $\mathcal{K}$ is $\bar{\mathcal{K}}\coloneqq {}^\boxslash (\mathcal{K}^\boxslash)$. The {saturated closure} of $\mathcal{K}$, denoted $\mathrm{sat}(\mathcal{K})$, is the closure of $\mathcal K$ under pushouts, transfinite composition, and retracts. It's not hard to see that $\mathrm{sat}(\mathcal{K}) \subseteq \bar{\mathcal{K}}$. If $L$ is a cocontinuous functor and $\mathcal{K}$ any class of morphisms, we have $L(\mathrm{sat}(\mathcal{K}))\subseteq\mathrm{sat}(L(\mathcal{K}))$. This inclusion is almost never an equality: although $L(\mathrm{sat}(\mathcal{K}))$ is closed under transfinite composition it need not be closed under pushouts or retracts, unless $L$ is full and essentially surjective. Even if $Lf \in \mathrm{sat}(L(\mathcal{K}))$, it need not be the case that $f \in \mathrm{sat}(\mathcal{K})$, unless $L$ is fully faithful. 

If $L\dashv R$ is an adjunction, then $Lf \boxslash g$ if and only if $f \boxslash Rg$ and hence $R$ gives a bijection between $(L\mathcal{K})^\boxslash$ and $\mathcal{K}^\boxslash \cap \im(R)$. Similarly $L$ gives a bijection between $^\boxslash(R\mathcal{K})$ and $^\boxslash\mathcal{K}\cap \im(L)$.

If $\mathcal{K}$ is a set and the domains of every morphism in $\mathcal{K}$ are small, then by Quillen's small object argument we have an equality $\mathrm{sat}(\mathcal{K}) = \bar{\mathcal{K}}$, and it follows that $\mathcal{K}^\boxslash = \bar{\mathcal{K}}^\boxslash$, which we will freely make use of. In particular, every curved algebra and every curved coalgebra is small, and we will implicitly use this fact in the following.

Our model structures will be cofibrantly generated, and we first define our sets of generating cofibrations.
\begin{defi}\hfill
	\begin{itemize}
		\item Let $\mathrm{inj}$ denote the set of injections between finite dimensional curved coalgebras. 
		\item Let $\mathrm{Inj}$ denote the class of injections of curved coalgebras.
		\item Let $\mathcal{W}$ denote the class of MC equivalences of curved coalgebras. We will abusively also use $\mathcal{W}$ to denote the class of MC equivalences of curved algebras; it will be clear from context what is meant.
		\item We write $\mathcal{W}\mathrm{inj}\coloneqq \mathcal{W}\cap\mathrm{inj}$ and $\mathcal{W}\mathrm{Inj}\coloneqq \mathcal{W}\cap\mathrm{Inj}$ for the class of injective MC equivalences.
	\end{itemize}
\end{defi}
We have $\overline{\mathrm{inj}}=\mathrm{Inj}$; clearly an $\mathrm{inj}$-cell complex is an injection, and the converse holds by \ref{injgen}. We also have $\overline{\mathcal{W}\mathrm{inj}}\subseteq \mathcal{W}\mathrm{Inj}$ since injective MC equivalences are a saturated class.

\begin{theorem}\label{mcmostralg}
	The category $\calgp$ of initialised curved algebras admits a combinatorial model structure, the {MC model structure}, where the weak equivalences are the MC equivalences, the generating cofibrations are $\Omega(\mathrm{inj})$, and the generating acyclic cofibrations are $\Omega(\mathcal{W}\mathrm{inj})$. Every algebra of the form $\Omega C$ is cofibrant.
\end{theorem}
\begin{proof}
	For brevity put $\mathcal{I}\coloneqq\Omega(\mathrm{inj})$ and $\mathcal{J}\coloneqq\Omega(\mathcal{W}\mathrm{inj})$. We will apply \cite[2.1.19]{Hov99}. It is clear that MC equivalences are closed under retracts and satisfy two-out-of-three. Moreover, all algebras are small, so we are left to verify the following two conditions:
	\begin{itemize}
		\item $\mathrm{Cell}(\mathcal{J})\subseteq \mathcal{W} \cap \mathrm{Cof}(\mathcal{I})$. \quad To see this, first observe that since $\mathcal{J}\subseteq \mathcal{I}$ we certainly have $\mathrm{Cell}(\mathcal{J})\subseteq  \mathrm{Cof}(\mathcal{I})$. Hence we just need to prove that  $\mathrm{Cell}(\mathcal{J})\subseteq \mathcal{W} $, which is a straightforward application of \ref{cellargument}.
		\item $\mathcal{I}^\boxslash = \mathcal{W} \cap \mathcal{J}^\boxslash$. \quad The fact that $\mathcal{W} \cap \mathcal{J}^\boxslash \subseteq \mathcal{I}^\boxslash$ follows from \ref{maypontoish}, since a morphism $f$ of algebras lifts against $\Omega g$ if and only if $\check B f$ lifts against $g$. For the converse inclusion, certainly $\mathcal{J}\subseteq \mathcal{I}$ so $\mathcal{I}^\boxslash \subseteq \mathcal{J}^\boxslash$, so it suffices to check that $\mathcal{I}^\boxslash \subseteq \mathcal{W} $. But this is \ref{sfibheq}.
	\end{itemize}
	Hence \cite[2.1.19]{Hov99} yields a cofibrantly generated model structure. The category of curved algebras is locally presentable, hence this model structure is combinatorial. If $C$ is a curved coalgebra, the natural morphism $\varnothing \to \Omega C$ is the image of the natural morphism $0\to C$ under $\Omega$. Since this map is an injection, $\varnothing \to \Omega C$ is a cofibration.
\end{proof}
\begin{lem}\label{rightproperlem}\hfill 
	\begin{enumerate}
		\item A map $f:A \to A'$ of curved algebras is an acyclic strong fibration if and only if it is an acyclic fibration in the MC model structure.
\item If $f:A \to A'$ is a strong fibration of curved algebras, then it is a fibration in the MC model structure.
\item Every curved algebra is a fibrant object in the MC model structure.
\item The MC model structure on curved algebras is right proper.
	\end{enumerate}
	\end{lem}
\begin{proof}
We keep the notation used in the proof of \ref{mcmostralg}. For (1), since $\overline{\mathrm{inj}}=\mathrm{Inj}$, the acyclic fibrations in the MC model structure are $\mathcal{I}^\boxslash \coloneqq\Omega(\mathrm{inj})^\boxslash = \overline{\Omega(\mathrm{inj})}^\boxslash = \overline{\Omega(\mathrm{Inj})}^\boxslash =\Omega(\mathrm{Inj})^\boxslash$. Now \ref{asfalg}(2) tells us that $\Omega(\mathrm{Inj})^\boxslash$ is exactly the class of acyclic strong fibrations. The proof of (2) is similar since $\overline{\mathcal{W}\mathrm{inj}}\subseteq \mathcal{W}\mathrm{Inj}$ which gives us an inclusion $\Omega(\mathcal{W}\mathrm{Inj})^\boxslash \subseteq \mathcal{J}^\boxslash$. For (3), to see that every curved algebra is fibrant, we need to check that $A\to 0$ is a fibration. In fact it is a strong fibration, since $\MCdg(C,0)\cong 0$ is the final dg category for every curved coalgebra $C$, and all dg categories are fibrant. Claim (4) follows immediately from (3).
	\end{proof}

Observe that the class of MC equivalences of curved algebras is accessible, since it is the class of weak equivalences of a combinatorial model category. We use this fact, together with Jeff Smith's theorem, to show that the category of curved coalgebras admits a similar model structure.

\begin{theorem}\label{thm:coalgebrasmodel}
	The category $\ccogp$ of finalised curved coalgebras admits a left proper combinatorial model structure, the {MC model structure}, where the weak equivalences are the MC equivalences and the cofibrations are the injective maps. Every coalgebra is cofibrant. Every coalgebra of the form $\check B A$ is fibrant.
\end{theorem}

\begin{proof}
For the existence of the model structure, we apply Jeff Smith's Theorem \cite{Beke00,Barwick10}. Our set of generating cofibrations will be $\mathrm{inj}$, and our weak equivalences will be be $\mathcal{W}$. The category of curved coalgebras is locally presentable, and $\mathcal{W}$ clearly satisfies two-out-of-three, so it remains to check the following three conditions:
	\begin{itemize}
		\item $\mathcal{W}$ is an accessible and accessibly embedded subcategory of the arrow category. \quad To see this, first observe that if (for clarity) $\mathcal{W}'$ denotes the class of MC equivalences of curved algebras, then we have $\mathcal{W}=\Omega^{-1}(\mathcal{W}')$ by 
		\ref{prop:MCcharacterization1}(4). Since $\mathcal{W}'$ is the class of weak equivalences of a combinatorial model category by \ref{mcmostralg}, it is a theorem of Smith \cite[2.5]{Barwick10} that it is accessible and accessibly embedded. Moreover, $\Omega$ is certainly an accessible functor, since it is cocontinuous, and now the claim follows from \cite[1.18]{Beke00}.
		\item $\mathrm{inj}^\boxslash \subseteq\mathcal{W}$. \quad This follows from \ref{sfibheq}, since $\mathrm{inj}^\boxslash=\mathrm{Inj}^\boxslash$.
		\item The class $\mathrm{cof}(\mathrm{inj})\cap \mathcal{W}$ is closed under pushouts and transfinite composition. \quad Since $\mathrm{cof}(\mathrm{inj})=\mathrm{Inj}$, the class in question is the class of injective MC equivalences, which is saturated by \ref{scsatd}.
	\end{itemize}
	Hence Jeff Smith's Theorem provides us with the structure of a combinatorial model category where the weak equivalences are the MC equivalences and the cofibrations are injections. Every curved coalgebra is clearly cofibrant, and hence the model structure is left proper. If $A$ is a curved algebra, lifting $\check B A \to *$ against acyclic cofibrations is equivalent to lifting $ A \to 0$ against morphisms from $\Omega(\mathcal{W}\mathrm{Inj})$, which holds by \ref{asfalg}(1) since $A \to 0$ is a strong fibration. Hence all coalgebras of the form $\check B A$ are fibrant.
\end{proof}

\begin{theorem}\label{thm:barcobarQequivalence}
	The bar-cobar adjunction is a Quillen equivalence.
\end{theorem}
\begin{proof}
	The functor $\Omega$ preserves and reflects weak equivalences by \ref{prop:MCcharacterization1}(4). It sends generating cofibrations to (generating) cofibrations, and hence preserves all cofibrations and so is left Quillen. It is homotopy essentially surjective by \ref{thm:barcobarequiv}(2), since for any algebra $A$ the counit map $\Omega \check BA \to A$ is an MC equivalence. It is homotopy fully faithful, since if $C,C'$ are curved coalgebras we have natural isomorphisms of sets $$[\Omega C, \Omega C']\cong [C, \check B \Omega C'] \cong [C,C']$$using this time that the unit $C' \to \check B \Omega C'$ is an MC equivalence by \ref{thm:barcobarequiv}(1).
\end{proof}

\subsection{Monoidal properties}\label{monoidalsection}
Recall from \cite{AnelJoyal} that the category of dg coalgebras is closed symmetric monoidal, and the category of dg algebras is enriched, tensored, and cotensored over dg coalgebras. The cotensor of $A$ by $C$ is the convolution algebra $\Hom(C,A)$ and the tensor of $A$ by $C$ is an algebra $C\sw A$ called the Sweedler product of $C$ and $A$. In \cite{HL2022} a model-categorical version of this was proved: the (finalised) category of pointed curved coalgebras $\pcog_*$ is a closed symmetric monoidal model category, and $\dgcat'$ is a $\pcog_*$-enriched model category. In this section we give an analogue of this in our setting: we show that $\ccogp$ is closed symmetric monoidal, and that $\calgp$ is a $\ccogp$-enriched model category. For the existence of the enrichment we will adapt the proofs given in \cite{HL2022}. We begin with the monoidal structure on coalgebras. We will employ a particularly useful adjoint functor theorem which we record here for future reference.

\begin{prop}\label{saft}
	Suppose that $\mathcal{C}$ is a locally presentable category and that $F:\mathcal{C} \to \mathcal{D}$ is a functor that preserves coproducts and coequalisers. Then $F$ is a left adjoint.
\end{prop}
\begin{proof}
	The category $\mathcal{C}$ is cowellpowered by \cite[1.58]{AdamekRosicky}. Since coproducts and coequalisers generate all colimits, $F$ preserves all colimits, and now the claim follows from the Special Adjoint Functor Theorem.
\end{proof}

\begin{theorem}\label{cogmonoidal}
	The category $\ccogp$, equipped with the MC model structure and the tensor product, is a monoidal model category.
\end{theorem}
\begin{proof}
	It is not hard to see that $(\ccogp,\otimes,\ground)$ is a symmetric monoidal category (recall that the zero coalgebra is an absorbing element and $C\otimes * \cong *$ for $C\neq 0$). To see that this monoidal structure is closed, by \ref{saft} it suffices to show that $C\otimes-$ preserves coproducts and coequalisers. But this can be proved in the exact same manner as \ref{convlimits}. Since every object in $\ccogp$ is cofibrant, the unit axiom is satisfied. We verify that the pushout-product axiom holds. Take a pair of injections $X\into X'$ and $Y\into Y'$ and let $P$ be the pushout of the span $X\otimes Y'\leftarrow X\otimes Y\to X' \otimes Y$, so that we obtain a commutative diagram
	$$\begin{tikzcd}X\otimes Y \ar[r,hook]\ar[d,hook]& X'\otimes Y \ar[d,hook]\ar[ddr,hook, bend left=30]& 
		\\X\otimes Y'\ar[r,hook]\ar[drr,hook, bend right=20] & P\ar[dr,hook] &
		\\&& X'\otimes Y'.
	\end{tikzcd}$$The natural map $P\to X'\otimes Y'$ is an injection, so we need to verify the acyclicity part of the axiom. Suppose that $X \into X'$ was acyclic. By \ref{mccogtensor}, the natural map $X\otimes Y \into X'\otimes Y$ is also acyclic, and by left properness so is its pushout $X\otimes Y' \to P$. By \ref{mccogtensor} again, $X\otimes Y' \into X'\otimes Y'$ is also acyclic, and hence by two-out-of-three $P \to X'\otimes Y'$ is also acyclic.
\end{proof}
We next turn our attention to the enrichment of $\calgp$ over $\ccogp$. Before we begin we will need some recollections on enriched categories. Suppose that $\mathcal{V}$ is a monoidal category and $\mathcal{C}$ is a $\mathcal{V}$-enriched category, with enrichment $\{-,-\}:\mathcal{C}^\mathrm{op}\times \mathcal{C} \to \mathcal{V}$. Recall that a tensoring of $\mathcal{C}$ over $\mathcal{V}$ is a functor $\sw:\mathcal{V}\times \mathcal{C} \to \mathcal{C}$ such that there is a natural isomorphism $\mathcal{V}(v,\{c,c'\})\cong \mathcal{C}(v\sw c, c')$. Similarly a cotensoring of $\mathcal{C}$ over $\mathcal{V}$ is a functor $[-,-]:\mathcal{V}^\mathrm{op}\times \mathcal{C} \to \mathcal{C}$ such that there is a natural isomorphism $\mathcal{V}(v,\{c,c'\})\cong \mathcal{C}(c, [v,c'])$. Recall from \cite[Chapter 4]{Hov99} the concept of an adjunction of two variables and a module over a monoidal category.
\begin{prop}\label{adjmoduleenrich}
	Let $(\mathcal{V},\otimes,I)$ be a monoidal category. Suppose that there is an adjunction of two variables $(\sw, [-,-],\{-,-\}):\mathcal{V}\times \mathcal{C} \to \mathcal{C}$ such that $\sw$ makes $\mathcal{C}$ into a right $\mathcal{V}$-module. Then $\mathcal{C}$ is enriched, tensored, and cotensored over $\mathcal{V}$.
\end{prop}
\begin{proof}
	The enrichment will be given by $\{-,-\}$. The adjunction of two variables property will ensure that $\mathcal{C}$ is tensored and cotensored, so we need only check that $\{-,-\}$ admits an associative and unital composition morphism. Recall that a functor $\sw$ makes $\mathcal{C}$ into a right $\mathcal{V}$-module if there are natural isomorphisms
	\begin{itemize}\item $a:(v\otimes v')\sw c \to v\sw(v'\sw c)$
		\item $r:I\sw c \to c$
	\end{itemize}
	satisfying fourfold associativity for $a$ and compatibility of $r$ with the unit isomorphisms in $\mathcal{V}$. Observe that for all $c,c'$ in $\mathcal{C}$ we have a universal map $e:\{c,c'\}\sw c\to c'$ obtained as the adjunct of $\id_{\{c,c'\}}$. By composition this gives us a morphism $$\left(\{c',c''\}\otimes \{c,c'\}\right)\sw c \xrightarrow{a}\{c',c''\}\sw \left(\{c,c'\}\sw c\right) \xrightarrow{\id\sw e} \{c',c''\}\sw c' \xrightarrow{e} c''$$which will be adjunct to our desired composition morphism $$\mu:\{c',c''\}\otimes \{c,c'\} \to \{c,c''\}.$$The associativity condition on $a$ translates precisely into associativity for $\mu$. As for units, we obtain a natural map $i:I \to \{c,c\}$ obtained as the adjunct of $r:I\sw c \to c$. Unitality of $\mu$ is then ensured by the compatibility of $r$ with the unit isomorphisms in $\mathcal{V}$.
\end{proof}
\begin{prop}\label{enrtenscotens}
	The category $\calgp$ is enriched, tensored, and cotensored over $\ccogp$.
\end{prop}
To prove this, our strategy will be to apply \ref{adjmoduleenrich}, for which we will need to produce tensors and cotensors. The cotensoring of an algebra by a coalgebra will be given by the convolution algebra. The tensoring and enrichment will be first defined on algebras of the form $\Omega C$ and then Kan extended to all algebras. In order to prove that the various adjunction properties hold, we will need to know that algebras of the form $\Omega C$ generate all algebras, in the following sense. Recall that a functor $i:\mathcal{C} \to \mathcal{D}$ is dense if there is a natural isomorphism $\mathrm{Lan}_ii \cong \id_\mathcal{D}$. We denote by $\Omega(\ccogp)$ the full subcategory of $\calgp$ spanned by the objects of the form $\Omega C$, for $C\in \ccogp$. Equivalently, $\Omega(\ccogp)$ is the category of curved coalgebras with morphisms the $\infty$-morphisms. Let $i:\Omega(\ccogp) \to \calgp$ be the fully faithful inclusion functor. 

\begin{lem}\label{denselem}
	The functor $i:\Omega(\ccogp) \to \calgp$ is dense.
\end{lem}
\begin{proof}
	Let $\calg'$ denote the non-full subcategory of $\calg$ with the same objects and whose morphisms are strict morphisms. As in \cite[2.4]{HL2022}, there is a monadic free-forgetful adjunction $H:\mathbf{grAlg} \leftrightarrow \calg':V$ where $V$ is the functor which forgets	the curvature and differential. Similarly, as in \cite[2.3]{HL2022}, there is a monadic free-forgetful adjunction $\mathbf{grVect}\leftrightarrow\mathbf{grAlg}$. The composition yields an adjunction $\mathbf{grVect}\leftrightarrow \calg'$ which - as in the proof of \cite[2.5]{HL2022} - is monadic because $V$ preserves coequalisers. This exhibits every curved algebra as the absolute coequaliser of a diagram $A \rightrightarrows A'$ where both $A$ and $A'$ are free curved algebras, i.e. curved algebras of the form $HT(U)$ where $U$ is a graded vector space and $T$ is the tensor algebra functor. As in \cite[2.6]{HL2022}, the algebra $HT(U)$ is the cobar construction on a dg coalgebra whose underlying graded coalgebra is a cosquare zero extension of $\ground$. Since $\Omega(0)\cong \varnothing$, we see that $\calgp$ is the closure of $\Omega(\ccogp)$ by absolute coequalisers. Since $i$ is fully faithful by definition, it now follows from \cite[5.19]{kelly} that $i$ is dense.
\end{proof}
\begin{rem}
	A similar proof, along the lines of that of \cite[2.7]{HL2022}, shows that the inclusion $j:\check B(\calgp) \to \ccogp$ is a codense functor, i.e.\ the right Kan extension $\mathrm{Ran}_{j}j$ is naturally isomorphic to the identity.
\end{rem}
\begin{rem}
	The functor $\Omega$ is not dense. Indeed $\mathrm{Lan}_\Omega\Omega$ can be computed as  $\Omega\mathrm{Lan}_\Omega\id$, and there is a natural isomorphism $\mathrm{Lan}_\Omega\id\cong \check B$. Hence the density comonad of $\Omega$ is precisely the cobar-bar resolution monad. Similarly, $\mathrm{Ran}_{\check B}\check B$ is $\check B \Omega$.
\end{rem}
Fix a curved coalgebra $C$. If $D$ is another curved coalgebra, observe that the hom-tensor adjunction for convolution algebras gives a universal algebra morphism \mbox{$\Omega D \to \Hom(C,\Omega(C\otimes D))$}. An algebra morphism $\Omega D' \to \Omega D$ hence gives a morphism \mbox{$\Omega D' \to \Hom(C,\Omega(C\otimes D))$}, which corresponds in the same way to a morphism $\Omega(C\otimes D')\to \Omega(C\otimes D)$. In other words, the assignment $D \mapsto \Omega(C\otimes D)$ defines a functor $\Omega(\ccogp) \to \ccogp$. Define a functor $C\sw-:\calgp \to \ccogp$ as the left Kan extension of the functor $D \mapsto \Omega(C\otimes D)$ along $i$. Via the pointwise description of left Kan extensions as colimits, we obtain an isomorphism $$C\sw A \cong\colim_{\Omega D \to A}\Omega(C\otimes D)$$which shows that $\sw$ is a functor in both variables.

Similarly, fix a curved algebra $A$. If $D$ is a curved coalgebra then the hom-tensor adjunction yields a universal algebra morphism $\Omega(D) \to \Hom(\check B \Hom(D,A),A)$. If $\Omega(D') \to \Omega(D)$ is an algebra morphism we hence obtain an algebra morphism $\Omega(D')\to \Hom(\check B \Hom(D,A),A)$, which corresponds to a morphism $\check B \Hom(D,A) \to \check B \Hom(D',A)$. Hence the assignment $D \mapsto \check B\Hom(D,A)$ defines a functor $\Omega(\ccogp) \to \ccogp^\mathrm{op}$. We define a functor $\{-,A\}:\calgp \to \ccogp^\mathrm{op}$ as the left Kan extension of the functor $D \mapsto \check B \Hom(D,A)$ along $i$. As before, the pointwise description of left Kan extensions gives an isomorphism $$\{A',A\}\cong \lim_{\Omega D \to A'}\check B \Hom(D,A)$$which shows that $\{-,-\}$ is a bifunctor.

\begin{proof}[Proof of \ref{enrtenscotens}]
	The enrichment will be given by $\{-,-\}$, the tensor by $\sw$, and the cotensor by the convolution algebra. If $C,D$ are curved coalgebras and $A$ is a curved algebra then we have natural isomorphisms 
	\begin{align*}
		\ccogp(C,\check B \Hom(D,A))&\cong \MC\Hom(C,\Hom(D,A))\\
		&\cong \MC\Hom(C\otimes D,A)\\
		&\cong \calgp(\Omega(C\otimes D),A)
	\end{align*}where in the middle we use the hom-tensor adjunction for convolution algebras. In particular, if $A'$ is a curved algebra we have natural isomorphisms \begin{align*}
		\ccogp(C,\{A',A\})&\cong \lim_{\Omega D \to A'}\ccogp(C,\check B \Hom(D,A))\\
		&\cong \lim_{\Omega D \to A'}\calgp(\Omega(C\otimes D),A)\\
		&\cong \calgp(C\sw A',A)
	\end{align*}
	which proves that $\sw$ is a tensor. Similarly, since $\otimes$ is symmetric monoidal, we have natural isomorphisms \begin{align*}
		\ccogp(C,\check B \Hom(D,A))&\cong \ccogp(D,\check B \Hom(C,A))\\
		&\cong \calgp(\Omega(D),\Hom(C,A))
	\end{align*}
	which give natural isomorphisms \begin{align*}
		\ccogp(C,\{A',A\})&\cong \lim_{\Omega D \to A'}\ccogp(C,\check B \Hom(D,A))\\
		&\cong \lim_{\Omega D \to A'}\calgp(\Omega(D),\Hom(C,A))\\
		&\cong\calgp(\colim_{\Omega D \to A'}\Omega(D),\Hom(C,A)).
	\end{align*}
	Note that $\colim_{\Omega D \to A'}\Omega(D)$ is the value of the left Kan extension $\mathrm{Lan}_ii$ on the object $A'$. Since $\Omega$ is dense by \ref{denselem}, this is naturally isomorphic to $A'$ and we hence have a natural isomorphism  \begin{align*}
		\ccogp(C,\{A',A\})&\cong\calgp(A',\Hom(C,A))
	\end{align*}which proves that the convolution algebra is a cotensor. So we obtain an adjunction of two variables, and hence by \ref{adjmoduleenrich} it is enough to prove that $\sw$ makes $\calgp$ into a $\ccogp$-module. To do this we will use the convolution algebra. Fix curved algebras $A,A'$ and curved coalgebras $C,C'$. By adjunction we have natural isomorphisms 
	\begin{align*}
		\calgp((C\otimes C')\sw A,A')&\cong	\calgp(A,\Hom(C\otimes C',A'))\\
		&\cong	\calgp(A,\Hom(C',\Hom(C,A')))\\
		&\cong	\calgp(C'\sw A,\Hom(C,A'))\\
		&\cong\calgp(C\sw(C'\sw A),A')
	\end{align*}
	and hence the Yoneda lemma gives us a natural isomorphism $(C\otimes C')\sw A\cong C\sw(C'\sw A)$. Using the associativity of $\otimes$ one can see that this satisfies the associativity property. For the unit isomorphisms, just observe that we have $\ground \sw A\cong \colim_{\Omega D \to A}\Omega(\ground \otimes D)\cong \mathrm{Lan}_ii (A)\cong A$. One can also check using the convolution algebra that this satisfies the unitality axioms for a $\ccogp$-module.
\end{proof}

\begin{rem}\label{coginternal}
	As in \cite{HL2022}, via similar reasoning one can show that the internal hom functor $[C,-]$ of $\ccogp$ is given by the right Kan extension of $A\mapsto \check B\Hom(C,A)$ along the inclusion $\check B(\calgp) \to \ccogp$. It is easy to check using this description that the bar-cobar adjunction is a $\ccogp$-enriched adjunction.
\end{rem}

Let $\mathcal{V}$ be a monoidal model category. Recall that a model category $\mathcal{C}$ is said to be a $\mathcal{V}$-enriched model category if it is enriched, tensored, and cotensored over $\mathcal{V}$, and any of the following equivalent conditions holds:
\begin{itemize}
	\item  $\mathcal{C}$ satisfies the pullback-power axiom.
	\item The tensor is a left Quillen bifunctor.
	\item The cotensor is a right Quillen bifunctor.
\end{itemize}
(for the equivalence, see \cite[4.2.2]{Hov99}).

\begin{theorem}\label{modelenrich}
	$\calgp$ is a $\ccogp$-enriched model category.
\end{theorem}
\begin{proof}
	We check that the Sweedler product is a left Quillen bifunctor. To do this it suffices to check on generating cofibrations. So let $C\into C'$ and $D\into D'$ be injections between finite dimensional curved coalgebras and let $P$ be the pushout of the span $C\sw \Omega D'\leftarrow C\sw \Omega D\to C' \sw \Omega D$ of algebras. By definition of the Sweedler product, $P$ is naturally isomorphic to the pushout of the span $\Omega(C\otimes D')\leftarrow \Omega(C\otimes D)\to \Omega(C' \otimes D)$. Since $\Omega$ is a left adjoint, $P$ is naturally isomorphic to $\Omega(P')$, where $P'$ is the pushout of the coalgebra span $C\otimes D'\leftarrow C\otimes D\to C' \otimes D$. By \ref{cogmonoidal}, the natural map $P'\to C'\otimes D'$ is a cofibration, and hence the natural map $P\to \Omega(C'\otimes D')$ is a cofibration. But this is the natural map $P\to C'\sw \Omega D'$ by the above arguments. Since $D\into D'$ is acyclic if and only if $\Omega D \to \Omega D'$ is, the acyclicity part of the axiom for $\sw$ to be a Quillen bifunctor also holds.
\end{proof}

\subsection{Fibrations and generating sets}
We use the enrichment of $\calgp$ over $\ccogp$, together with properties of $\MCdg$, to deduce that every model-theoretic fibration of algebras is a strong fibration. Similar ideas allow us to produce small generating sets for the MC model structure on $\calgp$.
	\begin{prop}\label{mcdgquillen}
		The functor $\MCdg:\calgp \to \dgcat'$ is right Quillen.
	\end{prop}
	\begin{proof}
		$\MCdg$ is a right adjoint by \ref{mcdgadj}. If $f$ is a fibration of curved algebras, observe that it must lift against the generating acyclic cofibration $\Omega(\ground) \to \Omega(I_3)$. As in the proof of \ref{Lfibsarestrongfibs}, it follows that $\MCdg(f)$ is a fibration of dg categories. If $f$ is an MC equivalence of curved algebras, then $\MCdg(f)\cong \MCdg(\ground ,f)$ is a quasi-equivalence by \ref{mccat2}(2). 
	\end{proof}
	
	Since the composition of a right Quillen functor with a right Quillen bifunctor is again a right Quillen bifunctor, we can deduce the following.
	\begin{cor}\label{mcdgquillenbi}
	$\MCdg:\ccogp^\mathrm{op} \times \calgp \to \dgcat'$ is a right Quillen bifunctor.
		\end{cor}
	\begin{proof}
		By \ref{modelenrich}, the convolution algebra $\Hom(C,A)$ is a right Quillen bifunctor, and hence by \ref{mcdgquillen} the composition $\MCdg(C,A)\coloneqq (\MCdg \circ \Hom)(C,A)$ is also a right Quillen bifunctor.
		\end{proof}
\begin{prop}\label{modeltheoreticfibs}
	Let $f:A \to A'$ be a morphism of curved algebras. The following are equivalent:
	\begin{enumerate}
		\item $f$ is a fibration in the MC model structure.
		\item $f$ is a strong fibration.
	\end{enumerate}
\end{prop}
\begin{proof}
	We have already observed in \ref{rightproperlem} that $(2)\implies (1)$, so we just need to prove that the converse holds. Take a model-theoretic fibration $f:A \to A'$; we wish to prove that it is a strong fibration. To do this, take an arbitrary curved coalgebra $C$. By \ref{injgen} we may write $C$ as a direct limit $\varinjlim_\alpha C_\alpha$, where $C_0 \cong 0$ and each $C_\alpha \to C_{\alpha+1}$ is an injection. Because $\MCdg(-,A)$ sends colimits to limits by \ref{mclimscor}, we may view $\MCdg(C,A)$ as the inverse limit of the tower $\MCdg(C_\alpha,A)$. Moreover, since this description is functorial, this exhibits $\MCdg(C,f)$ as the limit of the map of towers $\MCdg(C_\alpha,f)$. Hence to prove that $\MCdg(C,f)$ is a fibration, it suffices to show that the map $\MCdg(C_\alpha,f)$ is a fibration in the Reedy model structure on towers. Certainly $\MCdg(C_0,f)$ is a fibration, since it is the identity map on the zero dg category. So we just need to check that for all $\alpha$, the natural map $$\MCdg(C_{\alpha+1},A) \to \MCdg(C_{\alpha},A)\times_{\MCdg(C_{\alpha},A')}\MCdg(C_{\alpha},A')$$is a fibration. But this is precisely the pullback-power axiom for $\MCdg$, which holds since $\MCdg$ is a right Quillen bifunctor by \ref{mcdgquillenbi}.
	\end{proof}
\begin{rem}\label{mcdgl}
	In the above proof, since $\MCdg$ is Quillen in its coalgebra argument each map $\MCdg(C_{\alpha+1},A) \to \MCdg(C_{\alpha},A)$ is a fibration. In particular $\MCdg(C,A)$ is the homotopy inverse limit of the tower $\MCdg(C_\alpha,A)$.
\end{rem}
\begin{rem}
	Since strong fibrations are precisely those maps that lift against $\Omega(\mathcal{W}\mathrm{Inj})$, the above shows that we have an equality $\Omega(\mathcal{W}\mathrm{inj})^\boxslash = \Omega(\mathcal{W}\mathrm{Inj})^\boxslash$. In fact this should already hold on the level of coalgebras: if $i:C\into C'$ is an MC equivalence of coalgebras that is also a conilpotent extension, then it is a filtered quasi-isomorphism and the corresponding filtration exhibits $i$ as a $\mathcal{W}\mathrm{inj}$-cell complex. The general case of an injection reduces to the conilpotent case by an argument similar to the proof of \ref{thm:strongcof}. One should also be able to prove a `Goldman-Millson type theorem' along the lines of \cite{GoldmanMillson} stating that filtered quasi-isomorphisms between appropriate curved co/algebras are MC equivalences.
	\end{rem}

We now turn to our small generating sets, which will be defined as follows:
\begin{defi}\hfill
	\begin{itemize}
		\item Let $\mathcal{J'}$ be the set of morphisms of coalgebras of the form $i_0:C \into C \otimes I_3$ where $C$ is finite dimensional, and let $\mathcal{J}\coloneqq \Omega(\mathcal{J'})$.
		\item Let $\mathcal{I'}$ be the set of morphisms of coalgebras of the form $0 \to C$ or $C\sqcup C \to C\otimes I_3$, where $C$ is finite dimensional. Let $\mathcal{I}\coloneqq \Omega(\mathcal{I'}) \cup \mathcal{J} =\Omega(\mathcal{I'}\cup \mathcal{J}' )$.
	\end{itemize}
\end{defi}

\begin{prop}\label{sffinlem}
	Let $f:A \to A'$ be a morphism of algebras. The following are equivalent:
	\begin{enumerate}
		\item $f$ is a fibration in the MC model structure.
		\item $f$ has the right lifting property with respect to all maps from $\mathcal{J}$.
		\item For all finite dimensional curved coalgebras $C$, the dg functor $\MCdg(C,f)$ is a fibration.
		\end{enumerate}
	\end{prop}
\begin{proof}
	To show that $(1)\implies (2)$, just observe that $\mathcal{J}$ is a subset of the acyclic generating cofibrations for the MC model structure, and in particular fibrations must lift against $\mathcal{J}$. The proof of \ref{Lfibsarestrongfibs} (cf.\ also the proof of \ref{mcdgquillen}) shows that $(2)\implies (3)$. The proof of \ref{strongfibsareLfibs}(1) shows that if $f$ satisfies $(3)$, then it lifts against all generating acyclic cofibrations, and hence satisfies $(1)$.
	\end{proof}
\begin{prop}
	Let $f:A \to A'$ be a morphism of algebras. The following are equivalent:
	\begin{enumerate}
		\item $f$ is an acyclic fibration in the MC model structure.
		\item $f$ has the right lifting property with respect to all maps from $\mathcal{I}$.
		\item For all finite dimensional curved coalgebras $C$, the dg functor $\MCdg(C,f)$ is an acyclic fibration.
	\end{enumerate}
\end{prop}
\begin{proof}
	To see that $(1)\implies (2)$, just observe that $\mathcal{I}$ is a subset of the generating cofibrations for the MC model structure, and in particular acyclic fibrations lift against it. 	To see that $(2)\implies (3)$, suppose that $f$ lifts against $\mathcal{I}$ and take a finite dimensional curved coalgebra $C$. Since $\mathcal{J}\subseteq \mathcal{I}$, it follows from \ref{sffinlem} that $\MCdg(C,f)$ is a fibration of dg categories, so we just need to show that it is an acyclic fibration. This is similar to \ref{sfibheq}. Since $0\to C$ is in $\mathcal{I'}$, it follows that any morphism $\Omega C \to A'$ extends to a morphism $\Omega C \to A$. In particular, if $\MCdg(C,A)$ is empty then so is $\MCdg(C,A')$. So without loss of generality we can assume that both $\MCdg(C,A)$ and $\MCdg(C,A')$ are nonempty; it follows that $\Hom(C,A) \to \Hom(C,A')$ is surjective. We just need to check that any commutative diagram of the form $$\begin{tikzcd} & A\ar[d,"f"] \\
		\Omega C \ar[r,"g"] & A'\end{tikzcd}$$ admits a homotopy unique lift. We have already proved existence, so let $h,h':\Omega C \to A$ be two different lifts of $g$. These fit into a commutative diagram $$\begin{tikzcd} \Omega(C\sqcup C) \ar[r,"h\sqcup h'"] \ar[d]& A\ar[d,"f"] \\
		\Omega (C\otimes I_3) \ar[r,"\mathrm{const}_g"] & A'\end{tikzcd}$$where the left hand morphism is in $\mathcal{I}$. A lift in the above diagram corresponds to a $3$-homotopy $h \simeq h'$. Hence $\MCdg(C,f)$ is an acyclic fibration, as required. Finally we need to show that $(3)\implies (1)$ holds. Suppose that $f$ is a morphism satisfying $(3)$. It follows from \ref{sffinlem} that $f$ is a fibration in the MC model structure, so we need only show that it is an MC equivalence. To do this, take an arbitrary curved coalgebra $C$. Following the proof of \ref{modeltheoreticfibs}, write $C\cong\varinjlim_\alpha C_\alpha$, where $C_0 \cong 0$ and each $C_\alpha \to C_{\alpha+1}$ is a pushout of a strong cofibration $D_\alpha \to D'_\alpha$ between finite dimensional coalgebras. As in \ref{mcdgl}, we have a quasi-equivalence $\MCdg(C,f)\simeq \hoilim_\alpha\MCdg(C_\alpha,f)$. Since $\MCdg(C_\alpha,f)$ is a pullback of the acyclic fibration $\MCdg(D_\alpha,f)$, it is an acyclic fibration. In particular $\MCdg(C,f)$ is a homotopy limit of quasi-equivalences, and hence itself a quasi-equivalence.
\end{proof}

\subsection{Sliced model structures}
We can obtain MC model structures on dg coalgebras by slicing, since the category of dg coalgebras is the overcategory $(\ccogp)_{/\ground}$, and similarly for algebras. We start by recording the necessary facts we will need about sliced model structures.

\begin{theorem}\label{slicethm}
	Let $\mathcal{C}$ be a left (resp. right) proper combinatorial model category and $c\in \mathcal{C}$ an object. Then the slice categories $\mathcal{C}_{/c}$ and $\mathcal{C}_{c/}$ are left (resp. right) proper combinatorial model categories, with (co)fibrations and weak equivalences created by the forgetful functor to $\mathcal{C}$. The projection functors $\mathcal{C}_{/c} \to \mathcal{C}$ and $\mathcal{C}_{c/} \to \mathcal{C}$ are left (resp. right) Quillen. Moreover, if $$L:\mathcal{C} \longleftrightarrow \mathcal{D}:R$$ is a Quillen equivalence and $c\in \mathcal{C}$ is cofibrant and $d \in \mathcal{D}$ is fibrant, then the sliced adjunctions $$L:\mathcal{C}_{c/} \longleftrightarrow \mathcal{D}_{Lc/}:R$$
	$$L:\mathcal{C}_{/Rd} \longleftrightarrow \mathcal{D}_{/d}:R$$are Quillen equivalences.
	\end{theorem}

\begin{proof}The fact that the slice categories are cofibrantly generated proper model categories is \cite[15.3.6]{MayPontoMore}. Slice categories of locally presentable categories are again locally presentable and hence the slice model categories are combinatorial. It is easy to see that the projection functors are Quillen. The claim about Quillen equivalences is \cite[3.1]{LiSlice}.
	\end{proof}
 
 \begin{cor}\label{sliceadjns}
 	The categories $\accog$, $\cog$, and $\acog$ all admit left proper combinatorial model structures defined via the forgetful functor to $\ccogp$. The categories $\acalg$, $\alg$, and $\aalg$ all admit right proper combinatorial model structures defined via the forgetful functor to $\calgp$. There is a diagram of Quillen adjunctions
 	\end{cor}

 $$
 \begin{tikzcd}	[row sep={2em,between origins},]
 	\acog \ar[rr,"\Omega", bend left=8]\ar[dd,bend right=15]&\bot& \hphantom{f} \aalg\hphantom{f}  \ar[ll,"\check B", bend left=6] \ar[dd,bend right=15]\\
 		\dashv && \dashv \\
 	 	\accog \ar[rr,"\Omega", bend left=8]\ar[uu,bend right=15]\ar[dd,bend left=15]&\bot& \hphantom{fff}\alg\hphantom{fff} \ar[ll,"\check B", bend left=6] \ar[uu,bend right=15]\ar[dd,bend left=15]\\
 	\dashv && \dashv \\
 	\ccogp \ar[rr,"\Omega", bend left=8]\ar[dd,bend left=15]\ar[uu,bend left=15]&\bot& \calgp \ar[ll,"\check B", bend left=6] \ar[dd,bend left=15]\ar[uu,bend left=15]\\
 	\dashv && \dashv \\
 	\hphantom{fi} \cog\hphantom{fi}  \ar[rr,"\Omega", bend left=8]\ar[uu,bend left=15]\ar[dd,bend right=15]&\bot& \acalg \ar[ll,"\check B", bend left=6] \ar[uu,bend left=15]\ar[dd,bend right=15]\\
 		\dashv && \dashv \\
 		\acog \ar[rr,"\Omega", bend left=8]\ar[uu,bend right=15]&\bot& \hphantom{f} \aalg\hphantom{f}  \ar[ll,"\check B", bend left=6] \ar[uu,bend right=15]
 	\end{tikzcd}
 $$
where the diagrams of left (resp. right) adjoints commute and the bar-cobar adjunctions are Quillen equivalences.
\begin{proof}
	Recall that there are equivalences $\accog\cong (\ccogp)_{\ground/}$, $\cog\cong (\ccogp)_{/\ground}$, and $\acog\cong (\ccogp)_{\ground/\ground}$. Hence starting with the model structure on $\ccogp$ and repeatedly applying \ref{slicethm} gives us the existence of the model structures on the categories of coalgebras together with the Quillen adjunctions on the left hand side of the diagram. A similar argument works for algebras. The fact that the sliced bar-cobar adjunctions remain Quillen equivalences follows from \ref{slicethm}, since $\Omega(\ground)\cong \ground$ and $\check B(\ground)\cong \ground$, and every curved coalgebra (resp. curved algebra) is cofibrant (resp. fibrant). The topmost square of left adjoints commutes because the upper pair of bar and cobar functors are defined through the forgetful functors. Hence the topmost square of left adjoints also commutes. A similar argument works to show the commutativity of the other squares of adjoints and hence the whole diagram.
	\end{proof}

\begin{rem}
	If $(\mathcal{C},\otimes,I)$ is a monoidal category and $X\in \mathcal{C}$ is a monoid, then the slice category $\mathcal{C}_{/X}$ admits a monoidal structure given by the composition
	$$(C \to X)\otimes (C'\to X) \coloneqq (C\otimes C' \to X\otimes X \xrightarrow{\mu} X)$$and dually, if $Y$ is a comonoid then $\mathcal{C}_{Y/}$ is monoidal. This is a well-known folk theorem in category theory; cf.\ \cite{campbell}. If $Y$ is a comonoid and $X$ is a monoid, the set $\Hom(Y,X)$ becomes a monoid under convolution. If $f:Y \to X$ is such that $f^2=f$ in the convolution monoid, then one can check that the double slice category $\mathcal{C}_{Y/X}$ is also monoidal, with product inherited from $\mathcal{C}$. In particular if $Y=X=I$ and $f=\id$ then these conditions are satisfied and so all of $\mathcal{C}_{I/}$, $\mathcal{C}_{/I}$ and $\mathcal{C}_{I/I}$ are monoidal categories. Hence all of the categories appearing in \ref{sliceadjns} are monoidal categories with respect to the tensor product. 
	\end{rem}

We next examine how the sliced MC model structures interact with the usual model structures for conilpotent Koszul duality. If $C$ is a coaugmented curved coalgebra, then it has a maximal conilpotent subcoalgebra $\mathrm{nil}C$. Since the image of a conilpotent coalgebra is again conilpotent, $\mathrm{nil}$ is functorial: given a morphism $f:C \to C'$, the image of $\mathrm{nil}C$ necessarily is a subcoalgebra of $\mathrm{nil}C'$, which defines the desired map $\mathrm{nil}f$. It is easy to see that the $\mathrm{nil}$ functor is right adjoint to the inclusion functor $\iota: \mathbf{cuCog}^\mathrm{conil} \into \accog $. Alternately one can abstractly deduce the existence of a right adjoint to $\iota$ by appealing to \ref{saft}. Using \ref{mcconil} it is easy to see that $\iota$ is left Quillen.

Let $\algqim$ denote the category of dg algebras, equipped with the usual model structure where weak equivalences are quasi-isomorphisms and fibrations are surjections. Let $\algmcm$ denote the category of dg algebras, equipped with our MC model structure.
\begin{prop}	
The identity functor $\algmcm \to \algqim$ is right Quillen. 
	\end{prop}
\begin{proof}
If $f:A\to A'$ is any dg algebra morphism, note that $f$ is identified with the component of the dg functor $\MCdg(f)$ at the pair of MC elements $(0,0)$. If $f$ is a fibration then $\MCdg(f)$ is a fibration by \ref{mcdgquillen} and hence $f$ is a surjection, as desired.
	\end{proof}

Hence we have a diagram of Quillen adjunctions
$$\begin{tikzcd}	
		\mathbf{cuCog}^\mathrm{conil} \ar[rr,"\iota", bend left = 10] \ar[rr, phantom, "\perp"]\ar[dd,"\Omega", bend right=12, swap]\ar[dd, phantom, "\dashv"]&& \accog \ar[ll,"\mathrm{nil}", bend left = 10]\ar[dd,"\Omega", bend right=12, swap]\ar[dd, phantom, "\dashv"]
	 \\ &&\\
	\algqim \ar[rr,"\id", bend left = 10] \ar[rr, phantom, "\perp"] \ar[uu,"B",bend right=12, swap]&& \algmcm \ar[ll,"\id", bend left = 10]\ar[uu,"\check B",bend right=12, swap]
	\end{tikzcd}$$
The diagram of left adjoints is clearly commutative, and it hence follows that the diagram of right adjoints is commutative, i.e.\ there is a natural isomorphism $\mathrm{nil}\check B A \cong BA$.

Recall that a Quillen adjunction $L\dashv R$ with total derived adjunction $\mathbb{L}\dashv \mathbb{R}$  is called a {Quillen coreflection} if the derived unit $\id \to \mathbb{R}\mathbb{L}$ is an objectwise weak equivalence. This is equivalent to $\mathbb{L}$ being the inclusion of a coreflective subcategory, with coreflection $\mathbb{R}$. If $R$ is a right Bousfield localisation, then $L\dashv R$ is a Quillen coreflection, but the converse is not true (we will see a counterexample shortly, in \ref{bousfieldrem}).

\begin{prop}\label{qiMC}\hfill
\begin{enumerate}
	\item The identity adjunction $\algqim \leftrightarrow \algmcm$ is a Quillen coreflection.
	\item The inclusion adjunction $\iota: \mathbf{cuCog}^\mathrm{conil} \leftrightarrow\accog $ is a Quillen coreflection.
	\end{enumerate}

	\end{prop}

\begin{proof}
	We first recall that the derived unit of a Quillen adjunction $L\dashv R$ is modelled at an object $x$ by a morphism of the form $Qx \to RPLQx$, where $Q$ is a cofibrant replacement functor and $P$ is a fibrant replacement functor. In the case of (1), both $L$ and $R$ are the identity. Moreover, every algebra is fibrant in the MC model structure, so the derived unit at $A$ is simply the identity map $QA \to QA$, where $Q$ is a cofibrant replacement functor in the usual model structure. This is certainly a quasi-isomorphism. For (2), because every conilpotent coalgebra is cofibrant, the derived unit at a conilpotent coalgebra $C$ is of the form $C \to \mathrm{nil}PC$ where $P$ is a fibrant replacement in the category of coaugmented coalgebras. We may take $P$ to be the extended cobar-bar resolution $\check B \Omega$, and the derived unit becomes the natural map $C \to \mathrm{nil}\check B \Omega C\cong B\Omega C$ which is a weak equivalence in the category of conilpotent coalgebras.
	\end{proof}

Hence we have a diagram of homotopy categories $$\begin{tikzcd}	
	\Ho(\mathbf{cuCog}^\mathrm{conil}) \ar[rr,"\iota", bend left = 10, hook'] \ar[rr, phantom, "\perp"]\ar[dd,"\Omega", bend right=14, swap]\ar[dd, phantom, "\simeq"]&& \Ho(\accog) \ar[ll,"\mathbb{R}\mathrm{nil}", bend left = 10]\ar[dd,"\Omega", bend right=14, swap]\ar[dd, phantom, "\simeq"]
	\\ &&\\
	\Ho(\algqim) \ar[rr,"\mathbb{L}\kern -1.3pt\id", bend left = 10, hook'] \ar[rr, phantom, "\perp"] \ar[uu,"B",bend right=14, swap]&& \Ho(\algmcm) \ar[ll,"\id", bend left = 10]\ar[uu,"\check B",bend right=14, swap]
\end{tikzcd}$$ where the vertical maps are equivalences, the maps running to the right are inclusions of coreflective subcategories, and the maps running to the left are the corresponding coreflectors. The diagrams of left adjoints and right adjoints commute separately. It follows that any two parallel compositions beginning in the left hand column are isomorphic.

\begin{rem}
	The identity adjunction $\algqim \leftrightarrow \algmcm$ is not a Quillen reflection, since the derived counit at $A$ is the map $\Omega BA \to A$, which is a quasi-isomorphism but in general fails to be an MC equivalence. Similarly $\iota \dashv \mathrm{nil}$ is not a Quillen reflection since its derived counit at $C$ is $B\Omega C \to \check B \Omega C$, which need not be an MC equivalence.
\end{rem}

\begin{rem}\label{bousfieldrem}
	The identity functor $\algmcm \to \algqim$ is not a right Bousfield localisation, since by \ref{sfcounter} it does not reflect fibrations. However, the model category $\algmcm$ is combinatorial and right proper, so does admit a Bousfield localisation at the quasi-isomorphisms. Call this localisation the exotic model structure on dg algebras and write it as $\algexm$. Since we have a right Quillen functor $\algmcm \to \algqim$, the universal property of Bousfield localisation yields a right Quillen functor $\algexm \to \algqim$ which is necessarily a Quillen equivalence. There is no contradiction here since the property of being a (right) Bousfield localisation is not invariant under Quillen equivalences. 
	\end{rem}

\begin{rem}
One can repeat the above arguments verbatim in the setting of augmented dg algebras and conilpotent dg coalgebras.
	\end{rem}

\subsection{Categorical Koszul duality}
We next study how our MC model structure interacts with the model structure on pointed coalgebras for categorical Koszul duality constructed in \cite{HL2020}. This will be formally similar to the previous section; indeed a pointed curved coalgebra is conilpotent with respect to its coradical. We begin by reviewing some results from \cite{HL2020}.

A pointed curved coalgebra is a curved coalgebra $C$ such that:
\begin{itemize}
	\item The coradical $R$ of $C^\#$ is a direct sum of copies of $\ground$.
	\item The restriction of $d$ to $R$ is the zero map.
	\item $C$ is equipped with a coalgebra retract $\epsilon:C \to R$ of the inclusion map $R \into C$.
\end{itemize}

Note that the retract $\epsilon$ is considered as part of the data; in particular morphisms of pointed curved coalgebras must respect the retraction. We denote the category of pointed curved coalgebras by $\pcog$ and its finalisation by $\pcog_*$. There is a categorical cobar construction $\OCat:\pcog_* \to \dgcat'$. It has a right adjoint $B$, the bar construction of a dg category.

\begin{theorem}[\kern -5pt{\cite{HL2020}}]
	The category $\pcog_*$ admits a left proper combinatorial model structure. Weak equivalences are the maps which $\OCat$ sends to quasi-equivalences, and cofibrations are generated by injections between finite dimensional pointed curved coalgebras. The categorical bar-cobar adjunction is a Quillen equivalence.
\end{theorem}

Recall from \ref{mcdgadj} the construction of the reduced MC algebra $\algmc(\mathcal{D})$ of a dg category $\mathcal{D}$.

\begin{lem}\label{algocatlem}
	Let $C$ be a pointed curved coalgebra. There is a natural isomorphism $$\algmc(\OCat(C))\cong \Omega(C).$$
\end{lem}
\begin{proof}
	Let $\bar C$ be the cokernel of $R\into C$, so that we have a linear splitting $C\cong R\oplus \bar C$. The objects of $\OCat(C)$ are the irreducible coalgebra summands of $R$, and the set of maps of $\OCat(C)$ is given by the cotensor algebra $T_R(\bar C[-1])$. Choose a summand $\ground$ of $R$ and let $R'$ be the complement, so that we have $R\cong \ground \oplus R'$. Then the generators of the reduced MC algebra of $\OCat(C)$ are given by
	
	\begin{itemize}
		\item $\bar r$ of cohomological degree $1$, for every irreducible summand $r$ of $R'$.
		\item $\bar g$ of cohomological degree $n$, for every $g \in T_R(\bar C[-1])^n$.
	\end{itemize}
	Since composition in $\OCat(C)$ is sent to multiplication in the MC algebra, a smaller set of generators is
	\begin{itemize}
		\item $\bar r$ of cohomological degree $1$, for every irreducible summand $r$ of $R'$.
		\item $\bar g$ of degree $n$, for every $g \in \bar C^{n-1}$.
	\end{itemize}
	Observe that (since $g\mapsto \bar g$ is linear) this is precisely the space of generators of $\Omega C$. In particular, there is a natural algebra isomorphism $F:\Omega(C)^\#\to\left(\algmc(\OCat(C))\right)^\#$ which sends a generator $c\in C[-1]$ to $\bar c$. We wish to prove that $F$ is compatible with the differential, but this follows from the following additional equations imposed in the MC algebra:
	\begin{enumerate}
		\item $d(\bar r) = {\bar r} ^2$; note that this is usual cobar differential.
		\item If $g: r \to s$ then $\overline{dg} = d(\bar g) + \bar s\bar g - \tilde{\bar g} \bar r$. Again, this is the usual cobar differential.
	\end{enumerate}
	Hence $F$ is an isomorphism of dg (in particular curved) algebras, as required.
\end{proof}

\begin{prop}\label{algmciscoref}
	The adjunction $$\algmc:\dgcat' \longleftrightarrow \calgp: \MCdg$$ is a Quillen coreflection. 
\end{prop}
\begin{proof}
	By \ref{mcdgquillen}, $\MCdg$ is right Quillen. For coreflectivity, since every curved algebra is fibrant, the derived unit at a dg category $\mathcal{D}$ is of the form $Q\mathcal{D} \to \MCdg \algmc\mathcal{D}$, where $Q$ is a cofibrant replacement functor. Letting $C$ be the pointed curved coalgebra $B\mathcal{D}$, we may take $Q\mathcal{D}$ to be $\OCat C$. Using \ref{algocatlem}, the derived unit then becomes the natural morphism $\OCat C \to \MCdg(\Omega C)$. We wish to show that this map is a quasi-equivalence. To do this we will use module-comodule Koszul duality, both for curved algebras and dg categories. The morphism in question factors as a composite
	
	$$\OCat (C) \longrightarrow \OCat (C)\mathbf{\text{-}Mod} \longrightarrow C\mathbf{\text{-}Comod} \longrightarrow \Omega (C)\mathbf{\text{-}Mod}$$where the first arrow is the Yoneda embedding, the second arrow is the Koszul duality equivalence of \cite{HL2020}, and the third arrow is given by the usual Koszul duality equivalence of \cite{Positselski11}. Note that the morphism $ \OCat (C)\mathbf{\text{-}Mod} \longrightarrow C\mathbf{\text{-}Comod} $ itself factors as the composite of two quasi-equivalences
	$$\OCat (C)\mathbf{\text{-}Mod} \longrightarrow B\OCat (C)\mathbf{\text{-}Comod} \longrightarrow C\mathbf{\text{-}Comod} $$where the first is again Koszul duality and the second is induced by the unit $C \to B\OCat C$ of the categorical bar-cobar adjunction. By \cite[3.44]{HL2020}, the representable $\OCat(C)$-modules are identified with the one-dimensional $B\OCat (C)$-comodules, and hence with the one-dimensional $C$-comodules. But usual module-comodule Koszul duality identifies the full dg subcategory of one-dimensional $C$-comodules with $\MCdg(\Omega C)$. Hence the composite map $\OCat(C) \to \MCdg(\Omega C)$ is a quasi-equivalence, as desired.
\end{proof}

In particular, there is a fully faithful inclusion of a coreflective subcategory $$\mathbb{L}\algmc:\Ho(\dgcat') \into \Ho(\calgp) $$with coreflector given by $\MCdg$. Composing with the equivalence $\Ho(\dgcat)  \xrightarrow{\simeq} \Ho(\dgcat')$, we get a fully faithful functor $\Ho(\dgcat) \into \Ho(\calgp)$. So the homotopy theory of dg categories embeds fully faithfully into the homotopy theory of curved algebras.

\begin{rem}
	Let $\mathcal{E}$ be a dg category and $C\coloneqq B\mathcal{E}$ its bar construction. There is a natural quasi-equivalence $\OCat(C) \to \mathcal{E}$, and hence applying the derived MC algebra functor we obtain an MC equivalence $\mathbb{L}\algmc\OCat(C) \to \mathbb{L}\algmc\mathcal{E}$. Since $\OCat(C)$ is cofibrant, the source is $\algmc\OCat (C) \cong \Omega(C)$. By \ref{MCisDII}, we obtain an equivalence $\Dcoderivedc(\Omega C) \simeq \Dcoderivedc(\mathbb{L}\algmc\mathcal{E})$. The proof of \ref{algmciscoref} shows that if $C$ is any pointed curved coalgebra then there is a Koszul duality equivalence $ \Dcoderivedc(\Omega C)\simeq D(\OCat C)$, so we obtain an equivalence $D(\OCat C)\simeq \Dcoderivedc(\mathbb{L}\algmc\mathcal{E})$. But $\OCat C$ is quasi-equivalent to $\mathcal{E}$, and since the usual derived category is invariant under quasi-equivalences, we obtain an equivalence $D(\mathcal{E}) \simeq \Dcoderivedc(\mathbb{L}\algmc\mathcal{E})$. In particular if $\mathcal{E}$ is a cofibrant dg algebra, we obtain an equivalence $\Dcoderivedc(\mathcal{E}) \simeq \Dcoderivedc(\algmc\mathcal{E})$. In this sense, the (derived) MC algebra can be thought of as a `category algebra of the second kind'. Unlike the usual category algebra of a dg category, it is a functorial construction. 
\end{rem}

Let $\iota: \pcog_* \into \ccog_*$ be the inclusion. 
\begin{theorem}\label{categoricalcorefl}
	There is a square of Quillen adjunctions 
	
	$$
	\begin{tikzcd}
		\pcog_* \ar[rr,"\iota", bend left=8]\ar[dd,"\OCat", bend right=10, swap]& \bot&\ccog_*\ar[ll, bend left=7] \ar[dd,"\Omega", bend right=10, swap]\\
		\dashv & &\dashv\\
		\dgcat' \ar[rr,"\algmc", bend left=8]\ar[uu,"B", bend right=10, swap]& \bot &\calgp\ar[ll, "\MCdg",bend left=7]\ar[uu,"\check B", bend right=10, swap]
	\end{tikzcd}.
	$$
	The square of left (resp. right) adjoints commutes. The maps running vertically are Quillen equivalences and the maps running horizontally are Quillen coreflections.
\end{theorem}
\begin{proof}
	We begin by showing that $\iota$ is a left adjoint. By \ref{saft} it suffices to check that $\iota$ preserves coproducts and coequalisers. Coproducts in both $\pcog_*$ and $\ccog_*$ are created by the forgetful functor to graded vector spaces, so $\iota$ preserves coproducts. The proof of \cite[3.30]{HL2020} gives an explicit description of coequalisers in $\pcog_*$, and the proof of \ref{convlimits} gives an explicit description of coequalisers in $\ccog_*$. Using this one can check that $\iota$ preserves coequalisers and hence admits a right adjoint. The key observation is that coequalisers of pointed coalgebras, computed in $\ccog_*$, are naturally already pointed.
	
	It is clear that $\iota$ preserves cofibrations. To show that it is left Quillen, it is enough to check that it preserves weak equivalences. Let $f:C \to C'$ be a morphism of pointed curved coalgebras such that $\OCat(f)$ is a quasi-equivalence. By Ken Brown's Lemma combined with \ref{algocatlem}, $\algmc(\OCat(f))\cong \Omega(f)$ is an MC equivalence of curved algebras. By \ref{prop:MCcharacterization1}, the map $f$ is hence an MC equivalence of curved coalgebras.
	
	So we obtain a square of Quillen adjunctions, with vertical maps Quillen equivalences. The square of left adjoints commutes by \ref{algocatlem}, and hence the square of right adjoints also commutes. Since the bottom adjunction is a Quillen coreflection by \ref{algmciscoref}, the top adjunction must also be a Quillen coreflection, since one can test this on homotopy categories.
\end{proof}
If $K$ is a simplicial set, recall that $C_*(K)$ denotes the curved coalgebra of unnormalised simplicial chains on $K$. This construction is functorial and by applying it to the standard cosimplicial space we hence obtain a cosimplicial curved coalgebra $C_*(\Delta^\bullet)$. Similarly, $C^*(K)$ denotes the curved algebra of unnormalised cochains on $K$, and $C^*(\Delta^\bullet)$ is a simplicial curved algebra. If $A$ is a curved algebra, we define its Deligne category $\underline{\MC}(A)$ to be the simplicial set $\MC(C^*(\Delta^\bullet)\otimes A)$. This is named by analogy with the classical Deligne groupoid from deformation theory. The following result answers a question put to the authors by Sebastian Opper:

\begin{prop}\label{deligneprop}
    Let $A$ be a curved algebra. Then there is a natural isomorphism
    $$\underline{\MC}(A)\cong N_\mathrm{dg}(\MCdg(A))$$where $N_\mathrm{dg}$ denotes the dg nerve. In particular, $\underline{\MC}(A)$ is a quasicategory.
\end{prop}
\begin{proof}
    As in the proof of \cite[Lemma 4.11]{HL2020} (which is a standard argument using the category of simplices) the functor $C_*$ is left adjoint to the realisation functor $R\coloneqq \Hom_{\ccog}(C_*(\Delta^\bullet),-)$. Hence the composition $\Omega C_*$ is left adjoint to $R\check B$. But we have natural isomorphisms $$R\check B(A)=  \Hom_{\ccog}(C_*(\Delta^\bullet),\check BA)\cong \MC\Hom(C_*(\Delta^\bullet),A)\cong \MC(C^*(\Delta^\bullet)\otimes A) \cong \underline\MC(A)$$since $C_*(\Delta^n)$ is finite dimensional. So it suffices to check that $\Omega C_*$ is left adjoint to $N_{\mathrm{dg}}\MCdg$. To see this, first use \ref{algocatlem} to see that $\Omega C_*$ factors as $\algmc \OCat C_*$. Since $\algmc$ is left adjoint to $\MCdg$, and $\OCat C_*$ is left adjoint to $N_\mathrm{dg}$ by \cite{HL2020}, we are done.
\end{proof}

\subsection{Enrichment over dg categories}
In this section, we show that modified versions of the MC dg category can be used to detect the weak equivalences in the sliced model structures of \ref{sliceadjns}. Our method will be to show that the MC dg category is actually a sort of external hom valued in dg categories. Since $\dgcat'$ is not a monoidal model category, we will work only with homotopy categories. Recall from \cite{Toen06} that the homotopy category $\Ho(\dgcat)$ is closed symmetric monoidal with respect to the derived tensor product of dg categories. The natural equivalence $\Ho(\dgcat') \to \Ho(\dgcat)$ provides a closed symmetric monoidal structure on $\Ho(\dgcat')$, with product given again by the derived tensor product. We note that neither $\dgcat$ nor $\dgcat'$ are monoidal model categories with respect to the tensor product.

Let $C,C'$ be curved coalgebras and consider the functor sending $(C,C')$ to $ \MCdg(C,\Omega C')$. It is clear that this sends MC equivalences in both variables to quasi-equivalences of dg categories, and hence descends to a functor $\Ho(\ccogp)^\mathrm{op}\times \Ho(\ccogp) \to \Ho(\dgcat')$ which we denote by $\{C,C'\}$. Similarly, if $\mathcal{D}$ is a dg category, the functor $(\mathcal{D},C)\mapsto B(\mathcal{D})\otimes C$ preserves MC equivalences in the coalgebra variable by \ref{mccogtensor}, and sends quasi-equivalences to MC equivalences by \ref{categoricalcorefl}. Hence it descends to a functor $\Ho(\dgcat')\times \Ho(\ccogp) \to \Ho(\ccogp)$ which we denote by $\mathcal{D}\sw C$. Finally, consider the functor $(\mathcal{D},C)\mapsto \check{B}\Hom(B\mathcal{D},\Omega C)$, which preserves weak equivalences in both variables by \ref{mcequivtensor}. As before it descends to a functor $\Ho(\dgcat')^\mathrm{op}\times \Ho(\ccogp) \to \Ho(\ccogp)$ which we denote by $\underline{\Hom}(\mathcal{D},C)$.

\begin{theorem}
	The category $\Ho(\ccogp)$ is enriched, tensored, and cotensored over $\Ho(\dgcat')$.
\end{theorem}
\begin{proof}
	The external hom will be given by $\{C,C'\}$, the tensoring by $\mathcal{D}\sw C$, and the cotensoring by $\underline{\Hom}(\mathcal{D},C)$. We will apply \ref{adjmoduleenrich}. To begin, there are natural isomorphisms
	\begin{align*}
		\dgcat'(\OCat B\mathcal{D}, \MCdg(C,\Omega C'))&\cong \calgp(\algmc\OCat B\mathcal{D}, \Hom(C,\Omega C'))& \text{by \ref{mcdgadj}}\\
		&\cong \calgp(\Omega B\mathcal{D}, \Hom(C,\Omega C'))& \text{by \ref{algocatlem}}\\
		&\cong \ccogp(B\mathcal{D}\otimes C, \check B\Omega C')& \text{by hom-tensor}
	\end{align*}
	As a consequence, since $\OCat B\mathcal{D} \to \mathcal{D}$ and $C \to \check B \Omega C$ are natural isomorphisms in their respective homotopy categories, we obtain a natural isomorphism
	$$\Ho(\dgcat')(\mathcal{D}, \{C, C'\})\cong \Ho(\ccogp)(\mathcal{D}\sw C, C')$$showing that $\sw$ is indeed a tensor. Similarly, there are natural isomorphisms
	\begin{align*}
		\dgcat'(\OCat B\mathcal{D}, \MCdg(C,\Omega C')) &\cong \ccogp(B(\mathcal{D})\otimes C, \check B\Omega C')& \text{as above}\\
		&\cong \ccogp(C, \check B\Hom(B\mathcal{D},\Omega C'))& \text{by hom-tensor}
	\end{align*}
	and as before, this gives us a natural isomorphism $$\Ho(\dgcat')(\mathcal{D}, \{C, C'\})\cong \Ho(\ccogp)(C,\underline{\Hom}(\mathcal{D}, C'))$$showing that $\underline{\Hom}$ is a cotensor. So to finish we just need to check that $\sw$ makes $\Ho(\ccogp)$ into a $\Ho(\dgcat')$-module. By \cite[5.1]{HL2022}, the total derived functor $\mathbb{L}\OCat$ is strong monoidal, and hence so is its inverse $\mathbb{R}B$. Since every dg category is fibrant, it follows that if $\mathcal{D},\mathcal{D}'$ are dg categories then there is a natural isomorphism $B(\mathcal{D}\otimes \mathcal{D}')\cong B\mathcal{D} \otimes B\mathcal{D}'$ in $\Ho(\ccogp)$. In particular this yields natural isomorphisms in $\Ho(\ccogp)$
	\begin{align*}(\mathcal{D} \otimes \mathcal{D} ') \sw C&\cong B(\mathcal{D} \otimes \mathcal{D} ')\otimes C\\
		&\cong B\mathcal{D} \otimes B\mathcal{D} '\otimes C\\
		&\cong B\mathcal{D} \otimes (\mathcal{D} '\sw C)\\
		&\cong \mathcal{D} \sw (\mathcal{D} '\sw C)
	\end{align*}which satisfy the associativity condition because the tensor product of curved coalgebras is associative. Since $B$ applied to the one-object dg category $\ground$ is the curved coalgebra $\ground$, we have a natural isomorphism $\ground \sw C \cong C$ which satisfies the required unitality conditions.
\end{proof}

\begin{cor}
		The category $\Ho(\calgp)$ is enriched, tensored, and cotensored over $\Ho(\dgcat')$.
	\end{cor}
\begin{proof}Simply use the equivalence $\Ho(\calgp) \simeq \Ho(\ccogp)$ provided by Koszul duality. Concretely, the external hom is given by $\{A,A'\}\coloneqq \MCdg(\check B A, A')$, the tensoring by $\mathcal{D}\sw A \coloneqq \Omega (B\mathcal{D}\otimes \check B A)$ and the cotensoring by $\underline{\Hom}(\mathcal{D},A)\coloneqq\Hom(B\mathcal{D},A)$.
	\end{proof}

\begin{rem}
	The category $\pcog_*$ of finalised pointed curved coalgebras is a symmetric monoidal model category, Quillen equivalent to $\dgcat'$. As in the proof of \ref{enrtenscotens}, one can construct a model enrichment of $\ccogp$ over $\pcog_*$ using Kan extensions.
\end{rem}

	Let $A$ be a dg algebra and $\ground \to C$ be a coaugmented curved coalgebra. Since $A$ is a dg algebra, there is a canonical map $\ground \to \MCdg(A)\cong\MCdg(\ground,A)$ picking out the MC element $0$. We define $\MCdg(\bar C,A)$ to be the fibre of the natural map $\MCdg(C,A) \to \MCdg(\ground,A)$. Observe that this is a functor in both variables. Since $\ground \to C$ is an injection, $\MCdg(\bar C,A)$ can also be computed as the homotopy fibre. In particular, $\MCdg(\bar C,A)$ preserves MC equivalences in both variables and hence descends to a functor $\Ho(\accog) \times \Ho(\alg) \to \Ho(\dgcat')$. 
	
	Similarly, if $(A\to \ground) \in \acalg$ and $C \in \cog$ then we may define $\MCdg(C,\bar A)$ to be the fibre of $\MCdg(C,A) \to \MCdg(C,\ground)$. This descends to homotopy categories because $A \to \ground$ admits a section and is hence a fibration in the MC model structure.
	
	Finally, if $(A\to \ground) \in \aalg$ and $(\ground \to C) \in \acog$ then we may define $\MCdg(\bar C, \bar A)$ to be the fibre of $\MCdg(\bar C ,A) \to \MCdg(\bar C, \ground)$; equivalently it is also the fibre of $\MCdg(C , \bar A) \to \MCdg(\ground, \bar A)$. As before this descends to homotopy categories.
	\begin{theorem}\label{dgcatbars}\phantom{}\hfill
		\begin{enumerate}
			\item A map $f:A \to A'$ in $\alg$ is a weak equivalence in the sliced model structure if and only if for all $C \in \accog$ the natural map $\MCdg(\bar C,f)$ is a quasi-equivalence. A map $g:C \to C'$ in $\accog$ is a weak equivalence in the sliced model structure if and only if for all $A \in \alg$ the natural map $\MCdg(\bar g, A)$ is a quasi-equivalence.
			\item A map $f:A \to A'$ in $\acalg$ is a weak equivalence in the sliced model structure if and only if for all $C \in \cog$ the natural map $\MCdg(C,\bar f)$ is a quasi-equivalence. A map $g:C \to C'$ in $\cog$ is a weak equivalence in the sliced model structure if and only if for all $A \in \acalg$ the natural map $\MCdg(g,\bar A)$ is a quasi-equivalence.
			\item A map $f:A \to A'$ in $\aalg$ is a weak equivalence in the sliced model structure if and only if for all $C \in \acog$ the natural map $\MCdg(\bar C,\bar f)$ is a quasi-equivalence. A map $g:C \to C'$ in $\acog$ is a weak equivalence in the sliced model structure if and only if for all $A \in \aalg$ the natural map $\MCdg(\bar g, A)$ is a quasi-equivalence.
			\end{enumerate}
		\end{theorem}
	\begin{proof}
		We prove (1); the other claims are similar. There is a functor $F:\Ho(\alg) \to \Ho(\calgp)_{\ground /}$ which is the identity on objects and morphisms (although not itself an equivalence). By the definition of the slice model structure, $F$ reflects isomorphisms, so a map $f:A \to A'$ is an MC equivalence if and only if $F[f]$ is an isomorphism.
		
		If $\mathcal{V}$ is a monoidal category with enough pullbacks, $\mathcal{C}$ is a $\mathcal{V}$-enriched category, and $c\in \mathcal{C}$, then the slice category $\mathcal{C}_{c/}$ is also a $\mathcal{V}$-enriched category, with hom-objects given by taking fibres. Since $\MCdg(C,A) \to \MCdg(\ground,A)$ has a section, its fibre taken in $\Ho(\dgcat')$ exists and is the homotopy class of its usual fibre. Hence $\Ho(\calgp)_{\ground/}$ is also a $\Ho(\dgcat')$-enriched category, with enrichment given by $\{A,A'\}\coloneqq \MCdg(\bar{\check B}A,A').$ In $\Ho(\accog)$ every object is of the form $\check BA$, so the enriched Yoneda lemma tells us that a map $f':A \to A'$ in $\Ho(\calgp)_{\ground/}$ is an isomorphism if and only if for every $C \in \accog$ the map $\MCdg(\bar C,f')$ is an isomorphism in $\Ho(\dgcat')$. In particular, $F[f]$ is an isomorphism if and only if $\MCdg(\bar C, f)$ is a weak equivalence, which gives the first half of $(1)$. The second half is similar and uses that $\Ho(\ccogp)_{\ground/}$ is enriched over $\Ho(\dgcat')$, with external homs given by the expression $\MCdg(\bar{C},\Omega C').$ 
		\end{proof}

\section{Moduli problems}\label{section:moduli}
We apply our results on global Koszul duality to the study of moduli problems. Lurie defines an $\mathbb{E}_1$-formal moduli problem - an $\infty$-categorical analogue of a noncommutative deformation functor - as a certain kind of limit-preserving $\infty$-functor from connective Artinian local dg algebras to simplicial sets \cite{luriedagx}. Our MC model structures allow us to give an analogous definition where connective Artinian local dg algebras are now replaced with curved algebras. Even in the uncurved world, such deformation functors are new: they correspond to moduli problems defined on all finite dimensional dg algebras, rather than the connective local ones. Geometrically, these are global moduli problems, since finite dimensional algebras may have many different closed points; the prototypical examples are given by pseudocompact completions rather than completions at maximal ideals. When restricted to Artinian local dg algebras, our moduli functors are a nonconnective version of Lurie's.

Our main results here consist of prorepresentability theorems for global moduli problems. Such results are naturally obtained in the nonconnective setting, where every such functor is prorepresentable, in contrast to moduli problems defined on connective algebras. Since pro-objects correspond to left exact functors, the category $\cog$ of dg coalgebras is equivalent to the category $\mathrm{Lex}(\alg^{\mathrm{fd}},\mathbf{Set})$ of left exact functors from finite dimensional dg algebras to sets. This is an old observation; for an example in the commutative case see \cite{demazure}. Our representability theorem for MC stacks can be considered a derived version of this result.

\subsection{MC stacks}
Let $\fdccog_*$ be the category of finite dimensional curved coalgebras, along with the curved coalgebra $*$. Observe that $\fdccog_*$ has finite colimits. There is a natural equivalence $\mathbf{ind}(\fdccog_*)\simeq \ccog_*$ given by sending an ind-object to its colimit.

Regarding $\ccogp$ as an $\infty$-category, we may regard $\fdccog_*$ as a full $\infty$-subcategory. Since an injection of curved coalgebras is a cofibration, every curved coalgebra is a filtered homotopy colimit of finite dimensional curved coalgebras, and it follows that there is an equivalence of $\infty$-categories $\mathbf{ind}(\fdccog_*)\simeq \ccog_*$. 

Dually, let $\fdcalg_\varnothing$ be the category of finite dimensional curved algebras, which is finitely complete. If $\pccalg_\varnothing\simeq \ccogp^\mathrm{op}$ denotes the initialised category of pseudocompact curved algebras, there is an equivalence $\mathbf{pro}(\fdcalg_\varnothing)\simeq \pccalg_\varnothing$ given by sending a pro-object to its limit. We may regard $\fdcalg_\varnothing$ as an $\infty$-category, and we have an equivalence of $\infty$-categories $\mathbf{pro}(\fdcalg_\varnothing)\simeq \pccalg_\varnothing$.

\begin{prop}\label{finlims}
	The $\infty$-category $\fdcalg_\varnothing$ has finite limits.
	\end{prop}
\begin{proof}
	It is enough to check that $\fdcalg$ is closed under finite homotopy limits in the model category $\calgp$. Since homotopy products are usual products, it is closed under homotopy products. So we just need to show that it is closed under homotopy pullbacks. Equivalently, we need to show that the homotopy pushout of a diagram in $\fdccog_*$ remains an object of $\fdccog_*$. To do this, recall that we may compute homotopy pushouts in terms of cylinder objects: since every coalgebra is cofibrant, the homotopy pushout of a diagram $D \leftarrow C \to E$ of coalgebras may be computed as the usual pushout of the diagram $D \sqcup E \leftarrow C \sqcup C \to C\cdot I$ where $C\cdot I$ is a cylinder object for $C$. Since coalgebras are a monoidal model category, if $I$ is an interval object in coalgebras then $C\otimes I$ is an interval object for $C$. But the finite dimensional coalgebra $I_3$ is an interval object, and hence the above homotopy pushout can be computed as the pushout of $D \sqcup E \leftarrow C \sqcup C \to C\otimes I_3$. This pushout is finite dimensional.
	\end{proof}

\begin{defi}Let $\mathcal{D}$ be a finitely complete $\infty$-category. An MC stack with values in $\mathcal{D}$ is a pullback-preserving $\infty$-functor $X:\fdcalg_\varnothing\to \mathcal{D}$ such that $X(0)$ is the terminal object of $\mathcal{D}$.
\end{defi}

	One can identify the pullback preservation condition in simpler terms:

\begin{prop}\label{mcstackeasy}
	Let $\mathcal{D}$ be a finitely complete $\infty$-category and $X:\fdcalg_\varnothing \to \mathcal{D}$ an $\infty$-functor which sends $0$ to the terminal object. Then $X$ is an MC stack if and only if it preserves pullbacks along the following two types of morphisms:
	\begin{enumerate}
		\item square zero extensions.
		\item surjections of curved semisimple algebras.
	\end{enumerate}
\end{prop}
\begin{proof}
	The forward direction is clear, so assume that $X$ satisfies the two conditions. First observe that pullbacks in the $\infty$-category $\fdcalg_\varnothing$ can be computed as homotopy pullbacks in the model category $\calgp$. Since $\calgp$ is right proper, and every surjection of finite dimensional curved algebras is a fibration by \ref{thm:strongcof}, $X$ preserves pullbacks if and only if it preserves pullbacks along surjections. The proof of \ref{CogStructThm} shows that if $A'\twoheadrightarrow A$ is a surjection of finite dimensional curved algebras, then it fits into a diagram of the form

    $$\begin{tikzcd}
		A' \ar[r, "f"] & A'' \ar[r]\ar[d] \arrow[dr, phantom, "\lrcorner", very near start]& A \ar[d]\\
		& R' \ar[r, two heads]& R
	\end{tikzcd}$$where $R$ and $R'$ are curved semisimple and $f$ is a nilpotent extension. In particular if $B \to A$ is a map of curved algebras, then letting $B'' \to A''$ and $B' \to A'$ be its pullbacks, we obtain a commutative diagram $$\begin{tikzcd}
		B' \ar[r]\ar[d]\arrow[dr, phantom, "\lrcorner", very near start]& B''\ar[r] \ar[d]\arrow[dr, phantom, "\lrcorner", very near start]& B \ar[d] \\
		A' \ar[r, "f"] & A'' \ar[r]\ar[d] \arrow[dr, phantom, "\lrcorner", very near start]& A \ar[d]\\
		& R' \ar[r, two heads]& R.
	\end{tikzcd}$$Since every nilpotent extension is a composition of square zero extensions, $X$ preserves pullbacks along nilpotent extensions. Hence applying $X$ to the above diagram, we obtain a diagram 
	$$\begin{tikzcd}
		XB' \ar[r]\ar[d]\arrow[dr, phantom, "\lrcorner", very near start]& XB''\ar[r] \ar[d]& XB \ar[d] \\
		XA' \ar[r,] & XA'' \ar[r]\ar[d] \arrow[dr, phantom, "\lrcorner", very near start]& XA \ar[d]\\
		& XR' \ar[r]& XR
	\end{tikzcd}$$ in $\mathcal{D}$. The square $$\begin{tikzcd}
		XB''\ar[r] \ar[d]& XB \ar[d] \\
		XR' \ar[r]& XR
	\end{tikzcd}$$ is a pullback square in $\mathcal{D}$, so by pasting we conclude that the square $$\begin{tikzcd}
		XB'\ar[r] \ar[d]& XB \ar[d] \\
		XA' \ar[r]& XA
	\end{tikzcd}$$ is a pullback square, as desired.
\end{proof}

There is a natural $\infty$-category $\mathbf{St}_\mathrm{MC}(\mathcal{D})$ of $\mathcal{D}$-valued MC stacks, defined as a full subcategory of the functor category $\mathbf{Fun}(\fdcalg_\varnothing,\mathcal{D})$. In fact, since pullbacks and the terminal object generate all finite limits, the $\infty$-category $\mathbf{St}_\mathrm{MC}(\mathcal{D})$ is precisely the $\infty$-category $\mathrm{Lex}(\fdcalg_\varnothing,\mathcal{D})$ of left exact functors from $\fdcalg_\varnothing$ to $\mathcal{D}$ (recall that a functor is left exact if it preserves finite limits).

	\begin{prop}\label{mcstackadj}Let $X:\fdcalg_\varnothing\to \mathcal{D}$ be an MC stack. If $\mathcal{D}$ is complete then $X$ admits a continuous extension $\hat X: \pccalg_\varnothing \to \mathcal{D}$ which has a left adjoint.
		\end{prop}
	\begin{proof}
		We pass to opposite categories and consider the functor $Y\coloneqq X^\mathrm{op}:\fdccog_* \to \mathcal{E}\coloneqq \mathcal{D}^\mathrm{op}$. By \cite[5.3.5.8]{Lurie11a} we may extend $Y$ to a filtered colimit preserving functor $\hat Y: \ccogp \to \mathcal{E}$; note that this extends $Y$ since the restricted Yoneda embedding $j:\mathcal{C} \to \mathbf{ind}\mathcal{C}$ is fully faithful. It suffices to check that $\hat Y$ is cocontinuous and has a right adjoint. The latter follows from the former by an application of \cite[4.13]{NRS} with the colimit-dense subcategory $\fdccog_*$. To check that $Y$ is cocontinuous, we need only check that it preserves coproducts and pushouts \cite[4.4.2.7]{Lurie11a}. To do this, we will use that $j$ preserves finite colimits \cite[5.3.5.14]{Lurie11a}. Note that by assumption the functor $Y$ preserves finite colimits. Let $I$ be a set and $\{C_i\}_{i\in I}$ a collection of curved coalgebras. Let $I'$ be the filtered set of finite subsets of $I$, so that $\coprod_{i\in I}C_i\simeq \varinjlim_{J\in I'}\coprod_{i\in J}C_i$. Choose a filtered set $L$ such that each $C_i$ can be written as a filtered colimit of the form $\varinjlim_{l\in L}C_i^l$ with each $C_i^l$ finite dimensional. Putting this together we get an equivalence $\coprod_{i\in I}C_i\simeq \varinjlim_{J\in I'}\coprod_{i\in J}\varinjlim_{l\in L}C_i^l$. Since we chose $L$ independently of $i$ we may pass the inner colimit through the coproduct. Putting $S\coloneqq I'\times L$, which is filtered, we hence have an equivalence $\coprod_{i\in I}C_i\simeq \varinjlim_{(J,l)\in S}\coprod_{i\in J}C_i^l$. Note that for a fixed $(J,l)$ the coalgebra $\coprod_{i\in J}C_i^l$ is finite dimensional. We hence have equivalences 
		\begin{align*}
			\hat Y (\coprod_{i\in I}C_i) &\simeq \hat Y(\varinjlim_{(J,l)\in S}\coprod_{i\in J}C_i^l)&\\
&\simeq	 \varinjlim_{(J,l)\in S}\hat Y(\coprod_{i\in J}C_i^l)&\text{since }\hat Y\text{ preserves filtered colimits}\\
&\simeq	 \varinjlim_{(J,l)\in S}Y(\coprod_{i\in J}C_i^l)&\text{since }\hat Y\text{ extends }Y\\
&\simeq	 \varinjlim_{(J,l)\in S}\coprod_{i\in J}Y(C_i^l)&\text{since } Y\text{ preserves finite coproducts}\\
&\simeq	 \varinjlim_{J\in I'}\coprod_{i\in J}\varinjlim_{l\in L}{Y}(C_i^l)&\text{passing the $L$-colimit back into the coproduct}\\
&\simeq	 \varinjlim_{J\in I'}\coprod_{i\in J}{\hat Y}(C_i)&\text{since }\hat Y\text{ preserves filtered colimits}\\
&\simeq	 \coprod_{i\in I}{\hat Y}(C_i)
			\end{align*}
		So $\hat Y$ preserves coproducts. Checking that $\hat Y$ preserves pushouts is similar, but easier since a pushout diagram is finite. Let $C' \leftarrow C \to C''$ be a span in $\ccogp$ with pushout $P$. Let $L$ be a filtered set such that $C' \leftarrow C \to C''$ is the colimit over $l \in L$ of diagrams of finite dimensional curved coalgebras of the form $C'_l \leftarrow C_l \to C''_l$. Let $P_l$ be the pushout of this diagram; it is a finite dimensional curved coalgebra. Since colimits commute, $P$ is the filtered colimit of the $P_l$. So $\hat Y(P)$ is the filtered colimit of the $Y(P_l)$. By assumption each $Y(P_l)$ is the pushout of the diagram $YC'_l \leftarrow YC_l \to YC''_l$, and hence $\hat Y(P)$ is the pushout of the diagram $\hat YC'_l \leftarrow \hat YC_l \to \hat  YC''_l$, as required.
		\end{proof}
		\begin{cor}\label{stackrepcor}
			Let $X:\fdcalg_\varnothing\to \sSet $ be an MC stack valued in simplicial sets. Then $X$ is prorepresentable: there is a pseudocompact curved algebra $A_X$ and a natural equivalence $$X(R)\simeq \mathrm{Map}_{\pccalg_\varnothing}(A_X,R).$$
			\end{cor}
		\begin{proof}
			Let $G$ be the left adjoint of $\hat X$, so that for a simplicial set $K$ we have natural equivalences $ \mathrm{Map}(GK,R)\simeq \mathrm{Map}(K,\hat XR)$. Putting $K=*$ we get an equivalence $\hat X(R)\simeq  \mathrm{Map}(G(*),R)$. Since $\hat X$ extends $X$, we have an equivalence $X(R)\simeq \hat X(R)$ and hence we may take the algebra $A_X$ to be $G(*)$.
			\end{proof}
		
		\begin{rem}
			Writing $A_X$ as a cofiltered limit $\varprojlim_i A^i_X$, there is a natural equivalence $$\mathrm{Map}_{\pccalg_\varnothing}(A_X,A)\simeq \varinjlim_i \mathrm{Map}_{\fdcalg_\varnothing}(A^i_X,A).$$
			\end{rem}
		Observe that the functor $\pccalg_\varnothing \to \calgp$ given by $A \mapsto \Omega(A^*)$ is a contravariant equivalence. We denote its inverse by $A^!\coloneqq (\check B A)^*$. The following is a global, curved and nonconnective analogue of the main result of \cite[\S3]{luriedagx}. 
		\begin{prop}\label{lurierem}
			The $\infty$-functor $\Psi:\calgp \to \mathbf{St}_\mathrm{MC}(\sSet)$ defined by
			$$\Psi(A)(R) \coloneqq \mathrm{Map}_{\calgp}(\Omega(R^*), A)$$
			is an equivalence from $\calgp$ to the $\infty$-category of MC stacks in simplicial sets.
			\end{prop}
		
		\begin{proof}
		$\Psi$ is fully faithful by the Yoneda lemma. By \ref{stackrepcor}, every MC stack $X$ in simplicial sets is equivalent to $\mathrm{Map}_{\pccalg_\varnothing}(A_X,A)$ for some pseudocompact $A$. Since $R \mapsto \Omega(R^*)$ is a contravariant equivalence we have $X(R)\simeq \mathrm{Map}_{\calgp}(\Omega(R^*),\Omega(A_X^*))\simeq \Psi(\Omega(A_X^*))(R)$. Hence $\Psi$ is essentially surjective.
			\end{proof}

		\subsection{MC stacks over different bases}

Let $\fdaalg$, $\alg^\mathrm{fd}$ and $\acalg_\mathrm{fd}$ be the $\infty$-categories of finite dimensional augmented dg algebras, dg algebras, and augmented curved dg algebras respectively. Exactly as in \ref{finlims} each of these $\infty$-categories has finite limits. If $\mathcal{D}$ is an $\infty$-category with finite limits, we define $\infty$-categories of MC stacks
\begin{align*}\mathbf{St}^\mathrm{aug,dg}_\mathrm{MC}(\mathcal{D})&\coloneqq \mathrm{Lex}(\fdaalg,\mathcal{D})\\
	\mathbf{St}^\mathrm{dg}_\mathrm{MC}(\mathcal{D})&\coloneqq \mathrm{Lex}(\alg^\mathrm{fd},\mathcal{D})\\
	\mathbf{St}^\mathrm{aug}_\mathrm{MC}(\mathcal{D})&\coloneqq \mathrm{Lex}(\acalg_\mathrm{fd} ,\mathcal{D}).
\end{align*}
Let $\aalg$, $\alg$, and $\acalg$ be the $\infty$-categories of augmented dg algebras, dg algebras, and augmented curved algebras respectively, all viewed up to MC equivalence. Exactly as in \ref{lurierem}, the $\infty$-functors
		\begin{align*}\Psi^{\mathrm{aug}}_\mathrm{dg}: \aalg_\mathrm{MC} &\longrightarrow \mathbf{St}^\mathrm{aug,dg}_\mathrm{MC}(\sSet)\\
			\Psi^{\mathrm{aug}}: \acalg&\longrightarrow \mathbf{St}^\mathrm{dg}_\mathrm{MC}(\sSet)\\
			\Psi_\mathrm{dg}: \alg&\longrightarrow\mathbf{St}^\mathrm{aug}_\mathrm{MC}(\sSet)
		\end{align*}
		are all equivalences. If $R$ is a (curved) algebra then let $pR\coloneqq R\oplus \ground$ be the associated (curved) augmented algebra. The commutative diagram of right adjoints
			$$\begin{tikzcd}
				\aalg  \ar[r,"incl"] & \acalg\\
				\alg \ar[u,"p"] \ar[r,"incl"] & \calg_\varnothing \ar[u,"p",swap]
			\end{tikzcd}$$from \ref{sliceadjns} preserves finite dimensional algebras: the horizontal maps are inclusions, and the vertical maps are adjoining a unit. Note that this is true as $1$-categories or as $\infty$-categories, since every algebra is fibrant. Hence by restricting we obtain a commutative diagram of $\infty$-categories and finite limit preserving functors
			$$\begin{tikzcd}
				\fdaalg  \ar[r,"incl"] & \acalg_\mathrm{fd}\\
				\alg^\mathrm{fd} \ar[u,"p"] \ar[r,"incl"]& \fdcalg_\varnothing .\ar[u,"p", swap]
			\end{tikzcd}$$If $\mathcal{D}$ is a finitely complete $\infty$-category, by taking MC stacks we obtain a commutative diagram
			$$\begin{tikzcd}
				\mathbf{St}^\mathrm{aug,dg}_\mathrm{MC}(\mathcal{D})  \ar[d,"p^*",swap] & 	\mathbf{St}^\mathrm{aug}_\mathrm{MC}(\mathcal{D}) \ar[l,"res", swap]\ar[d,"p^*"]\\
				\mathbf{St}^\mathrm{dg}_\mathrm{MC}(\mathcal{D}) & \mathbf{St}_\mathrm{MC}(\mathcal{D}) \ar[l,"res", swap]
			\end{tikzcd}$$of categories of left exact functors. The maps running horizontally are the restriction functors and we have $(p^*X)(R)\simeq X(pR)$. If $\mathcal{D}$ is $\sSet$ then the above commutative diagram corresponds along the various $\Psi$ equivalences to the transposed commutative diagram of right adjoints
			$$\begin{tikzcd}
				\aalg \ar[d,"incl",swap] & \alg \ar[l,"p",swap] \ar[d,"incl"]\\
				\acalg & \calg_\varnothing \ar[l,"p"].
			\end{tikzcd}$$
			In particular, given $A\in \alg$, we can view it as controlling four different moduli problems, defined on the four categories $\fdaalg$, $\acalg_\mathrm{fd}$, $\alg^\mathrm{fd}$, or $\fdcalg_\varnothing$. Note that this above diagram has left adjoints, so in the special case $\mathcal{D}=\sSet$ we obtain the commutative diagram
			$$\begin{tikzcd}
				\mathbf{St}^\mathrm{aug,dg}_\mathrm{MC}(\sSet)  \ar[dd,"p^*", {name=pl}, bend left=10]\ar[rr, "incl_!", bend left=8] &\bot&	\mathbf{St}^\mathrm{aug}_\mathrm{MC}(\sSet) \ar[ll,"res", bend left=7]\ar[dd,"p^*", bend left=10] \\
				\dashv&&\dashv\\
				\mathbf{St}^\mathrm{dg}_\mathrm{MC}(\sSet) \ar[uu, bend left=10, "H_!"]\ar[rr, "incl_!", bend left=8] &\bot& \mathbf{St}_\mathrm{MC}(\sSet) \ar[ll,"res", bend left=7]\ar[uu, bend left=10, "H_!"]
			\end{tikzcd}$$
			where $F_!(\Psi(A)) \coloneqq \Psi(FA)$.
			
			\subsection{Deformation theory}
			We discuss the interaction of our MC stacks with the deformation functors of \cite{pridhamunifying, luriedagx}. Say that a dg $\ground$-algebra $\Gamma$ is Artinian local if it is finite dimensional and augmented over $\ground$, with nilpotent augmentation ideal. This is equivalent to being finite dimensional and having a unique two-sided maximal dg ideal $\mathfrak{m}$ with $\Gamma/\mathfrak{m}\cong \ground$. Observe that a nontrivial field extension of $\ground$ is never Artinian local in this sense, although it is Artinian as an abstract ring. Let $\mathbf{Art}$ denote the category of Artinian local dg algebras and $\mathbf{proArt}$ its procategory. Since the opposite category $\mathbf{Art}^\mathrm{op}$ is equivalent to $\cncog_\mathrm{fd}$, the category of finite dimensional conilpotent dg coalgebras, there is an equivalence $\mathbf{proArt}^\mathrm{op}\simeq \cncog$. One can transfer the model structure on $\cncog$ to $\mathbf{proArt}$; this yields a model structure where a weak equivalence $f$ is a map such that $\Omega(f^*)$ is a quasi-isomorphism of dg algebras. A fibration is a map $f$ such that $\varinjlim f$ is a degreewise surjection.
			
			The category $\mathbf{proArt}^{\leq 0}$ of connective pro-Artinian local dg algebras is also a model category, with weak equivalences the quasi-isomorphisms. The inclusion functor $\mathbf{proArt}^{\leq 0} \to \mathbf{proArt}$ is right Quillen, with left adjoint the connective cover functor $\tau_{\leq 0}:\mathbf{Art} \to \mathbf{Art}^{\leq 0}$ \cite[\S3.3]{B22}.
			
			Let $\iota:\cncog \to \acog$ be the forgetful functor. Exactly as in the curved case, this is left adjoint to the $\mathrm{nil}$ functor, which sends a coaugmented dg coalgebra to its maximal conilpotent subcoalgebra. The functor $\iota$ is left Quillen by \ref{mcconil}, and as in \ref{qiMC} is part of a Quillen coreflection. Dualising, we obtain an inclusion $\mathbf{proArt} \to \pcaalg$ which is part of a Quillen reflection.

			We view $\mathbf{proArt}$ as an $\infty$-category. It is equivalent to the procategory of the $\infty$-category $\mathbf{Art}$ of Artinian local dg algebras regarded up to weak equivalence. Similarly, $\mathbf{proArt}^{\leq 0}$ is an $\infty$-category, equivalent to the procategory of $\mathbf{Art}^{\leq 0}$, connective Artinian local dg algebras viewed up to quasi-isomorphism. Moreover $\pcaalg\simeq \mathbf{pro}\fdaalg$ is an $\infty$-category under the MC equivalences. By the above we have a sequence of reflective inclusions of $\infty$-categories $$\mathbf{proArt}^{\leq 0} \into \mathbf{proArt} \into \pcaalg.$$
			Since reflective functors create limits, we may compute limits in any of these categories as limits in $\pcaalg$. Restricting to finite dimensional algebras, we obtain a similar sequence of inclusions of $\infty$-categories
			$$\mathbf{Art}^{\leq 0} \into \mathbf{Art} \into \fdaalg.$$Since the $\infty$-category $\fdaalg$ has finite limits, the $\infty$-categories $\mathbf{Art}^{\leq 0} $ and $\mathbf{Art} $ also have finite limits, and the inclusions in the above sequence preserve these limits.

			\begin{defi}
				Let $\mathcal{D}$ be a finitely complete $\infty$-category. A formal moduli problem with values in $\mathcal{D}$ is a pullback-preserving $\infty$-functor $X:\mathbf{Art}^{\leq0} \to \mathcal{D}$ with $X(\ground)$ the terminal object.  A nonconnective formal moduli problem with values in $\mathcal{D}$ is a pullback-preserving $\infty$-functor $X:\mathbf{Art}\to \mathcal{D}$ with $X(\ground)$ the terminal object.
			\end{defi}
			We denote the $\infty$-category of formal moduli problems by $\mathbf{Def}^{\leq 0}(\mathcal{D})$ and the $\infty$-category of nonconnective formal moduli problems valued in $\mathcal{D}$ by $\mathbf{Def}(\mathcal{D})$. The inclusions $$\mathbf{Art}^{\leq 0} \into \mathbf{Art} \into \fdaalg$$preserve finite limits, and pullback along them hence induces maps $$\mathbf{St}^\mathrm{aug,dg}_\mathrm{MC}(\mathcal{D}) \to \mathbf{Def}(\mathcal{D})\to\mathbf{Def}^{\leq 0}(\mathcal{D})$$between categories of formal moduli problems.

			\begin{prop}
				The $\infty$-category $\mathbf{Def}^{\leq 0}(\sSet)$ is equivalent to the $\infty$-category of formal $\mathbb{E}_1$-moduli problems from \cite{luriedagx}.
			\end{prop}
			\begin{proof}
				Since both are constructed as full subcategories of $\mathbf{Fun}(\mathbf{Art}^{\leq 0}_\mathrm{q.i.},\sSet)$ it suffices to show that a functor $X$ is a $\sSet$-valued formal moduli problem in our sense if and only if it is a formal $\mathbb{E}_1$-moduli problem in the sense of Lurie. As in the proof of \ref{mcstackeasy}, a functor $X$ is an $\sSet$-valued formal moduli problem if and only if $X(\ground)$ is contractible and $X$ preserves pullbacks along square zero extensions. The equivalence between the two notions now follows from \cite[3.2.4]{luriedagx}.
			\end{proof}
			\begin{prop}
			The $\infty$-functor $\Psi_\mathrm{q.i.}:\aalg_\mathrm{q.i.} \to \mathbf{Def}(\sSet)$ defined by
				$$\Psi_\mathrm{q.i.}(A)(R) \coloneqq \mathrm{Map}_{\aalg_\mathrm{q.i.}}(\Omega(R^*), A)$$
				is an equivalence.
			\end{prop}
			\begin{proof}
				Completely analogous to the proof of \ref{lurierem}: a formal moduli problem extends to a continuous functor defined on procategories, which gives representability by a conilpotent coalgebra. Usual conilpotent Koszul duality identifies these with augmented dg algebras.
			\end{proof}
			
			Since the connective cover functor $\tau_{\leq 0}:\mathbf{proArt} \to \mathbf{proArt}^{\leq 0}$ is a left adjoint, we obtain by dualising a functor $\tau_{\geq 0}:\cncog \to\cncog_{\geq 0}$ which is right adjoint to the inclusion. Say that a dg algebra $A$ is coconnective if the natural map $\ground \to \tau_{\leq 0}A$ is a quasi-isomorphism.
			
			\begin{defi}
				The coconnective cover functor $\tau_{>0}: \aalg_\mathrm{q.i.} \to \aalg_\mathrm{q.i.}$ is defined by the composition $\tau_{>0}A\coloneqq \Omega\tau_{\geq 0}BA$. 
			\end{defi}
			Observe that $\tau_{>0}A$ is indeed coconnective, since $\tau_{\geq 0}BA$ is concentrated in nonnegative cohomological degrees. Moreover, if $A$ was coconnective then $\tau_{\geq 0}BA \simeq BA$ and so $\tau_{>0}A\simeq A$. 
			
			\begin{prop}\label{lurieintertwine}
				Lurie's equivalence $\Psi: \aalg_\mathrm{q.i.}  \to \mathbf{Def}^{\leq 0}(\sSet)$ fits into a commutative diagram of $\infty$-categories
				$$\begin{tikzcd}
					\aalg_\mathrm{MC}  \ar[d,"\Psi_\mathrm{MC}"]\ar[r]& \aalg_\mathrm{q.i.} \ar[r,"\tau_{>0}"] \ar[d,"\Psi_\mathrm{q.i.}"] & \aalg_\mathrm{q.i.} \ar[d,"\Psi"]\\  
					\mathbf{St}^\mathrm{aug,dg}_\mathrm{MC}(\sSet) \ar[r]&  \mathbf{Def}(\sSet) \ar[r]& \mathbf{Def}^{\leq 0}(\sSet).
				\end{tikzcd}$$
			\end{prop}
			\begin{proof}
				Let $A \in 	\aalg_\mathrm{MC} $ and $R \in \mathbf{Art}$. We compute \begin{align*} \Psi_\mathrm{MC}(A)(R)&\coloneqq \mathrm{Map}_{\aalg_\mathrm{MC} }(\Omega(R^*),A)&\\
					& \simeq \mathrm{Map}_{\acog }(R^*,\check BA)&\\
					&  \simeq \mathrm{Map}_{\cncog }(R^*,\mathrm{nil}\check BA)& \text{since }R\in \mathbf{Art}\\
					&  \simeq \mathrm{Map}_{\cncog }(R^*,BA)&\\
					&  \simeq \mathrm{Map}_{\aalg_\mathrm{q.i.}}(\Omega(R^*),A)&\\
					& \eqqcolon \Psi_\mathrm{q.i.}(A)(R)&	
				\end{align*}which shows that the left-hand square commutes. As for the right-hand square, if $A \in 	\aalg_\mathrm{q.i.} $ and $R \in \mathbf{Art}^{\leq 0}$, then we have natural equivalences
				\begin{align*} \Psi_\mathrm{q.i.}(A)(R)&\coloneqq \mathrm{Map}_{\aalg_\mathrm{q.i.} }(\Omega(R^*),A)&\\
					&  \simeq \mathrm{Map}_{\cncog}(R^*,BA)& \\
					&  \simeq \mathrm{Map}_{\cncog}(R^*,\tau_{\geq 0}BA)& \text{since }R\in \mathbf{Art}^{\leq 0}\\
					&  \simeq \mathrm{Map}_{\aalg_\mathrm{q.i.}}(\Omega(R^*),\tau_{>0}A)&\\
					&\eqqcolon\Psi(\tau_{>0}A)(R)
				\end{align*}as required.
			\end{proof}
			\begin{cor}
				The image of the natural map $\mathbf{Def}(\sSet) \to \mathbf{Def}^{\leq 0}(\sSet)$ consists precisely of the prorepresentable functors.
			\end{cor}
			\begin{proof}
				This follows from \cite[3.2.7]{luriedagx}, since by \ref{lurieintertwine} the image consists of those formal moduli problems of the form $\Psi(A)$ for $A$ coconnective.
			\end{proof}
			\begin{cor}The $\infty$-functor
				$\mathbf{St}^\mathrm{aug,dg}_\mathrm{MC}(\sSet) \to \mathbf{Def}(\sSet)$ given by restriction admits a fully faithful left adjoint (i.e. is a coreflection).
			\end{cor}
			\begin{proof}
				By \ref{lurieintertwine}, the functor in question is equivalent to the natural functor $\aalg_\mathrm{MC} \to \aalg_\mathrm{q.i.}$. By the augmented version of \ref{qiMC}(1), this has a fully faithful left adjoint. 
			\end{proof}
			\begin{rem}
				Since truncation is not well-defined on MC homotopy types, there is no connective cover functor $\mathbf{Art}_\mathrm{MC} \to \mathbf{Art}_\mathrm{MC}^{\leq 0}$ and so we refrain from discussing moduli problems defined on the latter category.
			\end{rem}
			
			\begin{rem}
				One can also consider formal moduli problems defined on the category $\mathbf{cuArt}$ of curved Artinian local algebras. These are represented by conilpotent curved coalgebras, and hence the $\infty$-category of such formal moduli problems is equivalent to the $\infty$-category $\alg_\mathrm{q.i.}$ of dg algebras up to quasi-isomorphism.
				\end{rem}

		\subsection{Noncommutative moduli spaces}
		Since $\Ho(\ccogp)$ is enriched, tensored, and cotensored over $\Ho(\dgcat)$, it follows that its opposite category $\Ho(\pccalg_\varnothing)$ is also enriched, tensored, and cotensored over $\Ho(\dgcat)$. Concretely, the enrichment is given by the formula $\{A,A'\}\simeq \MCdg(A'^*, \Omega(A^*))$.
		\begin{defi}
			A noncommutative moduli space is an MC stack $X:\fdcalg_\varnothing\to \dgcat$ valued in dg categories such that its continuous extension $\hat X$ preserves cotensors up to homotopy: for every dg category $\mathcal{D}$ and finite dimensional curved algebra $A$ there is a natural isomorphism $\hat X(\underline{\Hom}(\mathcal{D},A))\simeq \underline\Hom(\mathcal{D},XA)$ in the homotopy category of dg categories.
		\end{defi}
		There is a natural $\infty$-category $\mathbf{NCMod}$ of noncommutative moduli spaces, constructed as a full subcategory of the functor category $\mathbf{Fun}(\fdcalg_\varnothing,\dgcat )$.
		\begin{rem}Since a cotensor can be viewed as an enriched limit, asking that $\hat X$ preserve cotensors is an enriched version of asking that it preserve limits.\end{rem}
		\begin{prop}\label{ncmodrep}
			If $X$ is a noncommutative moduli space then there exists a curved algebra $A_X$ and a natural isomorphism $X(A)\simeq \MCdg(A\otimes A_X)$ in the homotopy category of dg categories.
		\end{prop}
		\begin{proof}
			By \ref{mcstackadj}, the $\infty$-functor $\hat X$ has a left adjoint $G$. Let $\mathcal{D,E}$ be dg categories and let $A\in \fdcalg_\varnothing$. We have isomorphisms
			\begin{align*}
				\Ho(\dgcat)(\mathcal{D},\{G\mathcal{E},A\})&\simeq \Ho(\pccalg_\varnothing)(G\mathcal{E},\underline{\Hom}(\mathcal{D},A))& \text{by (co)tensoring}\\
				&\simeq \Ho(\dgcat)(\mathcal{E},\hat X\underline{\Hom}(\mathcal{D},A))& \text{by adjunction}\\
				&\simeq \Ho(\dgcat)(\mathcal{E},\underline{\Hom}(\mathcal{D},XA))& \text{by assumption on $\hat X$}\\
				&\simeq \Ho(\dgcat)(\mathcal{D},\underline{\Hom}(\mathcal{E},XA))& \text{by (co)tensoring}
			\end{align*}
			and hence by the Yoneda lemma there is a natural isomorphism $\{G\mathcal{E},A\}\simeq \underline{\Hom}(\mathcal{E},XA)$ in $\Ho(\dgcat)$. Taking $\mathcal{E}=k$ now gives us a natural isomorphism $\{G{k},A\}\simeq XA$ in $\Ho(\dgcat)$. Putting $C\coloneqq (Gk)^*$, we hence have a natural isomorphism $XA\simeq \MCdg(A^*,\Omega C)$ in $\Ho(\dgcat)$. But $\Hom(A^*, \Omega C)$ is simply the tensor product $A\otimes \Omega C$, and so $\MCdg(A^*, \Omega C)\simeq\MCdg(A\otimes \Omega C)$, and so we may take $A_X=\Omega C$.
		\end{proof}
		
		Let $\mathcal{W}:\dgcat \to \sSet$ denote the $\infty$-functor obtained by composing the dg nerve with the core functor. Since both are right adjoints, $\mathcal{W}$ is itself a right adjoint and hence preserves limits. Hence every noncommutative moduli space $X$ has an underlying MC stack in simplicial sets $\mathcal{W}X$. 
		
	The following is an enriched version of \ref{lurierem}.
		
		\begin{prop}\label{ncmodcat}The $\infty$-functor $\Phi:\calgp \to \mathbf{NCMod}$ defined by $$\Phi(A)(R)\coloneqq \MCdg(R\otimes A)$$ is an equivalence on homotopy categories. The $\infty$-functor $\mathcal{W}:\mathbf{NCMod} \to \mathbf{St}_\mathrm{MC}(\sSet)$ is an equivalence on homotopy categories. There is a natural equivalence $\mathcal{W}\Phi\simeq \Psi$.
		\end{prop}
		
		\begin{proof}
			The functor $\Phi$ is essentially surjective by \ref{ncmodrep}. If $A$ is a curved algebra, recall that we write $A^!\coloneqq \check B(A)^*$. We have natural isomorphisms in the homotopy category of dg categories
			\begin{align*}
				\Phi(A)(R)&\cong \MCdg(R^*,A)\\
				& \simeq \MCdg(R^*,\Omega((A^!)^*))\\
				&\simeq \{A^!,R\}
				\end{align*}so the enriched Yoneda lemma tells us that $\Ho(\Phi)$ is fully faithful. Since we have natural equivalences $\mathcal{W}\{A^!,R\}\simeq \mathrm{Map}_{\pccalg_\varnothing}(A^!,R) \simeq \mathrm{Map}_{\calgp}(\Omega(R^*),A)$, we see that $\mathcal{W}\Phi\simeq \Psi$. Since $\Psi$ is an equivalence of $\infty$-categories by \ref{lurierem}, it follows that $\Ho(\mathcal{W})$ is an equivalence, as required.
		\end{proof}
	\begin{rem}\label{wrmk}
		An immediate corollary of the above proof is that if $A$ is a curved algebra then there is a weak equivalence of simplicial sets $\mathcal{W}(\MCdg(A))\simeq \mathrm{Map}_{\calgp}(\ground, A)$, cf.\ \ref{deligneprop}.
		\end{rem}
		\begin{rem}One could lift $\Phi$ and $\mathcal{W}$ to equivalences of $\infty$-categories by using the symmetric monoidal $\infty$-category of pointed curved coalgebras in place of the $\infty$-category of dg categories; the key tool here is the enriched Yoneda embedding.
		\end{rem}
		
		\begin{rem}
			Since $\MCdg$ is a right adjoint, it induces a natural morphism $$\MCdg_*:\mathbf{St}_\mathrm{MC}(\calgp) \to \mathbf{St}_\mathrm{MC}(\dgcat)$$which sends a stack $X$ to the stack $ \MCdg \circ X$. Using \ref{mcequivtensor} one can check that if $A$ is a curved algebra, there is a well-defined functor $Y(A): \fdcalg_\varnothing \to \calgp$ which sends $R$ to $\Hom(R^*,A)$. It is easy to see that $Y(A)$ is an MC stack, that $Y$ is functorial in $A$, and moreover that the diagram
			$$\begin{tikzcd}
				\calgp \ar[d,"Y", swap]\ar[r,"\Phi"]& \mathbf{NCMod} \ar[d,"incl"]\\
				\mathbf{St}_\mathrm{MC}(\calgp) \ar[r,"\MCdg_*", swap]& \mathbf{St}_\mathrm{MC}(\dgcat)
				\end{tikzcd}
				$$commutes. Since $\calgp$ is equivalent to the closed monoidal $\infty$-category $\ccogp$, it follows that $\calgp$ is enriched over itself, with internal hom given by $[A,A'] \coloneqq \Hom(BA, A')$. We have an equivalence $Y(A)\simeq [\Omega(R^*),A]$ and hence, by the enriched Yoneda lemma, $Y$ is fully faithful. As with dg categories, the image of $Y$ consists of those functors which preserve enriched limits, and $\MCdg_*$ gives an equivalence of this image with $\mathbf{NCMod}$.
			\end{rem}
		
		\subsection{Tangent spaces}
		Let $X: \fdcalg_\varnothing \to \mathcal{D}$ be an MC stack valued in an $\infty$-category $\mathcal{D}$. Let $*$ be the terminal object of $\mathcal{D}$, and $x:* \to X(\ground)$ be a morphism. We define the tangent space to $X$ at $x$ to be the fibre of the natural map $X(\ground[\varepsilon]/\varepsilon^2) \to X(\ground)$, where $X(\ground)$ is pointed by the morphism $x$. We denote the tangent space by $T_xX$; it is an object of $\mathcal{D}$. 
		
		We will compute the tangent spaces associated to both noncommutative moduli spaces and MC stacks in simplicial sets. Before we do so, we need a lemma on square zero extensions. If $V$ is a dg vector space then we let $\ground \oplus V$ denote the square zero extension of $\ground$ by $V$; it is the dg algebra whose underlying chain complex is $\ground \oplus V$ and whose multiplication is given by $(\lambda+v)(\lambda'+v') = \lambda\lambda' + \lambda v ' + \lambda' v$.
		
		\begin{lem}\label{sqzlem}
			Let $V$ be a dg vector space and let $A\coloneqq \ground \oplus V$ be the square zero extension. The dg category $\MCdg(A)$ is quasi-equivalent to the disjoint union $\sqcup_{H^1(V)}H^*(A)$ of $H^1(V)$ copies of the one-object dg category $H^*(A)$.
			\end{lem}
		\begin{proof}
			An MC element of $A$ is precisely a degree 1 cocycle in $V$, and the hom-complex between two such cocycles $x,y$ is $A$ equipped with the differential $\lambda + v \mapsto dv +\lambda(y-x)$. If $x$ and $y$ are cohomologous in $V$ then there is some $v$ with $dv=y-x$, and the morphisms $1+v: x \to y$ and $1-v: y \to x$ are mutually inverse in the homotopy category. Conversely, if $x$ and $y$ are isomorphic in the homotopy category, an isomorphism between them is represented by an element $\lambda + v$ with $dv = \lambda(x-y)$. Since the corresponding morphism is invertible, $\lambda \neq 0$, and hence $x$ and $y$ are cohomologous. Choose a quasi-isomorphism $V \to H^*V$ which sends a cocycle to its cohomology class. This induces a morphism of dg categories $\MCdg(A) \to \MCdg(H^*A)$ which, by the above, gives a bijection $\MCmod(A) \to \MCmod(H^*A)$ on isoclasses of objects in the homotopy categories. Hence $A$ and $H^*A$ are MC equivalent, and so $\MCdg(A) \to \MCdg(H^*A)$ was actually a quasi-equivalence. Hence to prove the claim we may assume that $V$ has zero differential. In this case, the objects of $\MCdg(A)$ are in bijection with the elements of $V^1$, and $\MCdg(A)(x,y)$ is acyclic if $x\neq y$. Hence $\MCdg(A)$ is quasi-equivalent to the disjoint union $\sqcup_{x\in V^1}A^x$. But the twist $A^x$ is isomorphic to $A$, and the result follows.
		\end{proof}

		\begin{prop}\label{tgtncmod}
			Let $A$ be a curved algebra and let $X$ be the noncommutative moduli space $\Phi(A)$. A morphism $x:* \to X(\ground)$ is the same as an MC element of $A$, and the tangent space $T_xX$ is quasi-equivalent to the disjoint union $\sqcup_{H^{1}(A^x)}(\ground \oplus H^{*}(A^x))$ of copies of the square zero extension of $\ground$ by the cohomology of the dg vector space $A^x$.
			\end{prop}
		\begin{proof}
			The terminal object of $\dgcat$ is the one-object dg category $\ground$, and a morphism $\ground \to \mathcal{C}$ is the same as an object of $\mathcal{C}$. Since $X(\ground)\simeq \MCdg(A)$, the first claim follows. To prove the second claim, fix an MC element $x$ of $A$. The tangent space $T_xX$ is then the fibre of the fibration $\MCdg(A[\varepsilon]/\varepsilon^2) \to \MCdg(A)$. The objects of this dg category are the MC elements $y=x+z\varepsilon$ of $A[\varepsilon]/\varepsilon^2$, which are in bijection with the cohomological degree $1$ cocycles $z$ in the two-sided twist $A^{x}$. If $x+z\epsilon$ and $x+z'\varepsilon$ are two such MC elements, the complex of maps between them is the square zero extension $\ground \oplus A\varepsilon$, with differential given by $\lambda + a\varepsilon \mapsto\lambda z \varepsilon - \lambda z ' \varepsilon + d^x(a)\varepsilon$. This identifies $T_xX$ with the dg category $\MCdg(\ground \oplus A^x)$, and the result now follows from \ref{sqzlem}.
			\end{proof}
		
		If $X$ is a noncommutative moduli space, then recall that $\mathcal{W}X$ is an MC stack in simplicial sets, where $\mathcal{W}$ denotes the functor which sends a dg category to the core of its dg nerve. Since $\mathcal{W}$ is a right adjoint, it commutes with limits, and so we have a weak equivalence $\mathcal{W}(T_xX)\simeq T_x(\mathcal{W}X)$ of simplicial sets. 
		
		If $V$ is a dg vector space, it has a generalised Eilenberg--Mac Lane space, which is a simplicial set which we denote by $K(V)$; concretely it can be obtained by applying the Dold--Kan correspondence to the connective cover of $V$.

		\begin{prop}\label{tgtsset}
			Let $A$ be a curved algebra and let $Y$ be the MC stack in simplicial sets $\Psi(A)$. A morphism $x:* \to Y(\ground)$ is the same as an MC element of $A$, and the tangent space $T_xY$ is weakly equivalent to the coproduct $\coprod_{H^{1}(A^x)}K(\ground\oplus H^*(A^x))$.
			\end{prop}
		\begin{proof}
			A morphism $* \to Y(\ground)$ is a vertex of the simplicial set $\mathrm{Map}(\ground,A)$, hence a morphism $\ground \to A$, which is an MC element of $A$. Since $\Psi(A)\simeq \mathcal{W}\Phi(A)$ by \ref{ncmodcat}, we may compute $T_xY$ as $\mathcal{W}T_x\Phi(A)$. By \ref{tgtncmod}, this latter simplicial set is $\mathcal{W}(\coprod_{H^{1}(V)}(\ground\oplus V))$, where for brevity we write $V\coloneqq H^*(A^x)$. Observe that the coproduct $\coprod_{H^{1}(V)}(\ground \oplus V)$ is equivalently the product of dg categories $V^{1} \times (\ground \oplus V)$, where we regard the set $V^{1}$ as a discrete dg category and the dg algebra $\ground \oplus V$ as a one-object dg category. Since $\mathcal{W}$ is a right adjoint, it follows that $T_xY$ is equivalent to the product $\mathcal{W}(V^{1})\times \mathcal{W}(\ground \oplus V)$. Since $V^{1}$ is a discrete dg category, $\mathcal{W}(V^{1})$ is the discrete simplicial set $V^{1}$. If $B$ is any dg algebra then the simplicial set $\mathcal{W}(B)$ is equivalent to the mapping complex $\mathrm{Map}_{\dgcat}(\ground,B) \simeq \mathrm{Map}_{\alg}(\ground ,B)$, where we consider $\alg$ with weak equivalences the quasi-isomorphisms. In particular, since taking square-zero extensions is right adjoint to the forgetful functor to dg vector spaces, we have an equivalence $\mathrm{Map}_{\alg}(\ground ,\ground \oplus V)\simeq \mathrm{Map}_{\ground}(\ground, \ground\oplus V)\simeq K(\ground\oplus V)$. Hence $T_xY$ is the simplicial set $V^{1}\times K(V)\simeq \coprod_{V^{1}}K(\ground\oplus V)$. 
			\end{proof}

		\begin{rem}If $A$ is a dg algebra, let $Z:\acalg_\mathrm{fd} \to \dgcat$ be the MC stack from \ref{dgcatbars} given by $Z(R)\coloneqq \MCdg(\bar{R}^*,A)\simeq \mathrm{fib}(\MCdg(R\otimes A) \to \MCdg(A))$, where the dg category $\MCdg(A)$ is pointed by the zero MC element. Then the MC stacks $X\coloneqq\Phi(A)$ and $Z$ have the same tangent space: there are quasi-equivalences $Z(\ground )\simeq \ground$ and $Z(\ground[\varepsilon]/\varepsilon^2) \simeq T_0X$, and hence the tangent space to $Z$ at the unique object $* $ of the dg category $\ground$ is $T_*Z \simeq T_0X$.
			
		If $A$ is an augmented curved algebra, let $W:\alg^\mathrm{fd} \to \dgcat$ be the MC stack whose value on $R$ is the fibre of the natural map $\MCdg(R\otimes A) \to \MCdg(R)$, where the latter dg category is pointed by $0$. If $X$ is as above then the natural map $W(\ground) \to X(\ground)$ is a quasi-equivalence, and as in the proof of \ref{tgtncmod} if $x\in X(\ground)$ then the tangent space $T_xW$ is the coproduct $\sqcup_{H^{1}(\bar{A}^x)}(\ground \oplus \bar{A}^x)$.
		
		Finally, if $A$ is an augmented dg algebra, then the MC stack $V: \fdaalg \to \dgcat$ given by $R \mapsto \MCdg(\bar R^*, \bar A)$ has a unique point $*$ with tangent space $T_*V\simeq T_0W$, where $W$ is as above.
			\end{rem}
			
			\begin{rem}
				If $X$ is a formal moduli problem, one can define its tangent space $T_xX$ at a vertex $x\in X(\ground)_0$ in exactly the same manner as above. Since $X(\ground)$ is contractible, this is independent of the choice of vertex, and we just write $TX$. If $X=\Psi_\mathrm{q.i.}(A)$ for some augmented dg algebra $A$ then $TX$ is the generalised Eilenberg--Mac Lane space $K(\bar A [1])$: this holds since the linear dual of the dg coalgebra $B(\ground[\varepsilon]/\varepsilon^2)$ is simply $\ground[x]$, with $x$ in cohomological degree $1$, and one has $\mathrm{Map}_{\aalg_\mathrm{q.i.}}(\ground[x],A)\simeq K(\bar A [1])$. If $E_n$ denotes the augmented dg algebra $\ground[\varepsilon]/\varepsilon^2$ with $\varepsilon$ placed in homological degree $n$, then the $E_n$ assemble into a spectrum object $E$ in $\aalg_\mathrm{q.i.}$; cf.\ \cite[3.2.1]{luriedagx}. Since $X$ preserves finite limits, it follows that $X(E)$ is a spectrum object in simplicial sets, i.e.\ a spectrum, which is equivalent to the Eilenberg--Mac Lane spectrum of the chain complex $\bar A [1]$ \cite[3.2.6]{luriedagx}. We call $X(E)$ the tangent complex; note that the zeroth space of $X(E)$ is precisely the tangent space $TX$.
				
				Similarly, let $Y:\aalg_\mathrm{MC}\to \mathcal{D}$ be an MC stack on augmented dg algebras - the reason for the restriction is because we want our source category to have a zero object. Given a spectrum object $E\in\mathrm{Sp}(\aalg_\mathrm{MC})$, since $Y$ preserves finite limits we obtain a spectrum object $Y(E)\in \mathrm{Sp}(\mathcal{D})$. However, the collection $E_n$ of augmented dg algebras above do not seem to form a spectrum object in $\aalg_\mathrm{MC}$.

								\end{rem}

\subsection{Examples}
We give some examples of MC stacks and noncommutative moduli spaces. Typically, we will start with a classical moduli problem and then extend it to a derived moduli problem. For example, instead of classifying MC elements in an algebra up to gauge equivalence, we will classify them up to homotopy gauge equivalence. Or instead of classifying local systems of vector spaces, we will classify local systems of dg vector spaces.

Since there are four possible categories of algebras - namely with or without curvature, and with or without augmentation - every moduli problem will have four different variants, depending on which category of algebras it is defined. These four variants will be controlled by four slightly different algebras. We refrain from discussing every possible variant and usually restrict ourselves to moduli problems defined on curved, non-augmented algebras.

\subsubsection{Moduli of MC elements in an algebra}
This is the prototypical example. Let $A$ be a curved algebra and let $M_A$ be the functor which sends a finite dimensional curved algebra $R$ to the dg category $M_A(R)\simeq \MCdg(R\otimes A)\cong \MCdg(R^*,A)$. Then $M_A\simeq \Phi(A)$ and moreover we have $\mathcal{W}(M_A) \simeq \Psi(A)$.

If $A$ is a dg algebra, then there is a natural moduli functor $M_A^\mathrm{dg}$ which sends a finite dimensional augmented curved algebra $R$ to the dg category $\MCdg(\bar{R}^*,A)$. This is the moduli functor of MC elements in the dg algebra $A$. Note that we need to restrict the morphisms in this dg category to avoid gauges between MC elements that correspond to curved morphisms with nontrivial curvature.

If $A$ is a augmented curved algebra, there is a natural moduli functor $M_A^\mathrm{aug}$ which sends a finite dimensional dg algebra $R$ to the dg category $\MCdg(R^*,\bar{A})$. This is the moduli functor of MC elements in the augmented curved algebra $A$.

Finally, if $A$ is an augmented dg algebra, there is a natural moduli functor $M_A^\mathrm{aug, dg}$ which sends a finite dimensional augmented dg algebra $R$ to the dg category $\MCdg(\bar{R}^*, \bar{A})$. This is the moduli functor of MC elements in the augmented dg algebra $A$.

\subsubsection{Pseudocompact completions}
The $1$-functor $\pccalg_\varnothing \to \calgp$ which forgets the topology has a left adjoint, the pseudocompact completion, which we denote by $A \mapsto \check A$. Viewing $\pccalg_\varnothing$ as the procategory of $\fdcalg_\varnothing$, the pseudocompact completion of $A$ is the cofiltered system of finite dimensional quotients of $A$. Across the equivalence $\pccalg_\varnothing \simeq \ccogp^\mathrm{op}$, the functor forgetting the topology corresponds to the linear dual functor, which is equivalently the convolution algebra functor $\Hom(-,\ground)$. By \ref{modelenrich} this functor is right Quillen, and hence pseudocompact completion gives an $\infty$-functor $\calgp \to \pccalg_\varnothing$.

If $A$ is a pseudocompact curved algebra, let $X_A$ be the MC stack in simplicial sets defined by $X_A(R)\coloneqq\mathrm{Map}_{\pccalg_\varnothing}(A,R)$. Then $X_A\simeq \Psi(\Omega(A^*))$.
 If $A$ is a curved algebra, let $Y_A$ be the MC stack $Y_A(R)\coloneqq \mathrm{Map}_{\calg_\varnothing}(A,R)$. Then $Y_A\simeq X_{\check A}$ and hence $Y_A\simeq \Psi(\Omega(\check A^*))$.

\subsubsection{Moduli of flat connections}
Let $M$ be a smooth manifold, $E$ a vector bundle on $M$, and let $A\coloneqq\Omega(\End(E))$ be the graded algebra of $\End(E)$-valued differential forms. A connection $\nabla$ on $E$ is given by a $1$-form $x \in A^1$, and the curvature of $\nabla$ is the $2$-form $h\coloneqq dx + x^2$. Fixing a connection on $E$ we may view $A$ as a curved algebra, with curvature element $h$.

An arbitrary connection is flat if and only if its associated curvature form vanishes, which gives a bijection
$$\{\text{flat connections on }E\}\longleftrightarrow \{\text{MC elements of }A\}.$$ 
We regard two flat connections on $E$ as equivalent if they differ by a gauge equivalence; this is the case if and only if the associated MC elements of $A$ are gauge equivalent.

We claim that two MC elements of A are gauge equivalent if and only if they are homotopy gauge equivalent. To see this, let $x,y$ be two MC elements and let $(f,g,h_1,h_2)$ be a gauge equivalence between them, so that we have $gf = 1 + d^xh_1$ and $fg = 1 +d^yh_2$. Since $d^xh_1$ is nilpotent, $1+d^xh_1$ is a unit, and similarly $1+d^yh_2$ is a unit. Hence $f$ is invertible and we see that it gives a gauge equivalence between $x$ and $y$. So the noncommutative moduli space of flat connections on $E$ is the MC stack $\Phi(A)$.
\subsubsection{Moduli of complex structures}
This example is similar to the previous. Let $M$ be a complex manifold and $E$ a smooth vector bundle on $M$. A choice of almost complex structure on $E$ makes the graded algebra $A\coloneqq \Omega^{0,*}(\End(E))$ of $\End(E)$-valued antiholomorphic forms into a curved algebra. An MC element of $A$ is the same thing as a complex structure on $E$, and two complex structures are equivalent if and only if the associated MC elements are homotopy gauge equivalent. Hence the noncommutative moduli space of complex structures on $E$ is the MC stack $\Phi(A)$.

\subsubsection{Moduli of objects in dg categories}
Fix a dg category $\mathcal{D}$ and consider the functor $M_\mathcal{D}$ which takes a curved algebra $R$ to the dg category $\MCdg(R\otimes \algmc\mathcal{D})$. Note that $M_\mathcal{D}(\ground)\simeq \mathcal{D}$ by \ref{algmciscoref}, so we view $M_\mathcal{D}$ as the moduli functor of objects in $\mathcal{D}$. Clearly $M_\mathcal{D}$ is $\Phi(\algmc(\mathcal{D}))$. The functor that sends $\mathcal{D}$ to $M_D$ is left adjoint to the functor that sends a noncommutative moduli space $X$ to $X(\ground)$; we regard $X(\ground)$ as the underlying dg category of $X$.

Note that To\"en and Vaqui\'e's moduli functor $\mathcal{M}_\mathcal{D}$ describes essentially the moduli of perfect complexes on $\mathcal{D}$; indeed it is Morita invariant \cite{toenvaquie} whereas our functor is not.

We compute the tangent spaces of $M_\mathcal{D}$. An MC element in $\algmc(\mathcal{D})$ is the same as an object of $\mathcal{D}$. Fixing $d \in \mathcal{D}$, \ref{tgtncmod} tells us that the tangent space $T_dM_\mathcal{D}$ is the dg category $\MCdg(\ground \oplus H^*(\algmc(\mathcal{D})^d))$. Recall that the definition of $\algmc$ required a choice of object; without loss of generality we may take this to be $d$. Then the corresponding $MC$ element of $\algmc(\mathcal{D})$ is simply $0$. By \ref{algmciscoref} there is a quasi-equivalence $\MCdg(\algmc(\mathcal{D}))\simeq \mathcal{D}$ which sends $0$ to $d$, and this induces a quasi-isomorphism $\algmc(\mathcal{D})\simeq \End_\mathcal{D}(d)$. Hence the tangent space $T_dM_\mathcal{D}$ is quasi-equivalent to the dg category $\MCdg(\ground \oplus \mathrm{Ext}_\mathcal{D}^*(d,d))$, where we write $\mathrm{Ext}_\mathcal{D}^*(d,d)$ for the cohomology of $\End_\mathcal{D}(d)$. Compare this to the tangent complex to To\"en and Vaqui\'e's moduli functor at a pseudoperfect module $E$, which is $\mathrm{Ext}^*(E,E)[1]$ by \cite[3.17]{toenvaquie}.

\subsubsection{Moduli of twisted modules}
Fix a curved algebra $A$. A twisted module over $A$ is the same thing as a vector space $V$ and a choice of MC element $x \in A\otimes \End(V)$; the corresponding twisted module is $(A\otimes V)^{[x]}$. Hence the noncommutative moduli space of twisted modules on $A$ with underlying vector space $A\otimes V$ is the MC stack $\Phi(A\otimes \End(V))$. Note that if $V$ is finite dimensional then $\End(V)$ is a matrix algebra.

\subsubsection{Moduli of dg modules}
Fix a dg algebra $A$ and let $BA$ be its ordinary, non-extended bar construction, which is a conilpotent curved coalgebra. Fix a vector space $V$. Given an $A$-module structure on $V$ we obtain a $BA$-comodule structure on the tensor product $BA \otimes V$, and such a comodule structure is precisely an MC element in the convolution algebra $E\coloneqq\Hom(BA,\End(V))$. Hence the moduli space of dg-$A$-modules with fixed underlying vector space $V$ is the moduli space of MC elements in the curved algebra $E$; this is precisely the MC stack $\Phi(E)$.

\subsubsection{Moduli of local systems}
Let $X$ be a path connected pointed topological space. A local system of dg-$\ground$-vector spaces on $X$ is a dg-module over the group algebra $\ground [\pi_1(X)]$. The noncommutative moduli space of local systems with fixed fibre $V$ is the same as the noncommutative moduli space of $\ground [\pi_1(X)]$-modules with underlying module $V$, which is controlled by the dg algebra $\Hom(B\ground [\pi_1(X)],\End(V))$.

Similarly, an $\infty$-local system is a module over the dg algebra $C_*(\Omega X)\simeq \Omega C_*(X)$. Since $B\Omega C_*(X)\simeq C_*(X)$, the moduli space of $\infty$-local systems with fixed fibre $V$ is controlled by the dg algebra $\Hom(C_*(X), \End(V))\simeq C^*(X)\hat\otimes \End(V)$.

\begin{rem}Moduli problems that are naturally controlled by differential graded Lie algebras or $L_\infty$-algebras are not covered by our theory, since we do not have a notion of MC equivalence in this setting (for pronilpotent dg Lie algebras the relevant notion is that of a filtered quasi-isomorphism, c.f.\ \cite{GoldmanMillson}). For example, deformations of algebras over operads, deformations of complex structures, and deformations of $A_\infty$-structures on a given vector space are all moduli functors of this form.
\end{rem}

\bibliography{biblibrary}
\bibliographystyle{alpha}

\end{document}